\documentclass{amsart}

\usepackage{amsfonts}
\usepackage{amsmath,amscd,amssymb,amsthm}
\usepackage{mathrsfs}
\usepackage{pstricks}
\usepackage[cmtip,arrow]{xy}
\usepackage{pb-diagram,pb-xy}
\usepackage{url}
\usepackage[left=1in,top=1in,text={6.5in,9in}]{geometry}

\title{Symmetric Homology of Algebras}

\author{Shaun V. Ault}
\address{Department of Mathematics, Fordham University, Bronx, New York, 10461, 
  USA.}
\email{ault@fordham.edu}



\newtheorem{theorem}{Theorem}
\theoremstyle{plain}
\newtheorem{lemma}[theorem]{Lemma}
\newtheorem{prop}[theorem]{Proposition}
\newtheorem{cor}[theorem]{Corollary}
\newtheorem{conj}[theorem]{Conjecture}
\theoremstyle{definition}
\newtheorem{definition}[theorem]{Definition}
\theoremstyle{remark}

\newtheorem{rmk}[theorem]{Remark}

\newcommand{\Z}{\mathbb{Z}}
\newcommand{\ltimescirc}{\textrm{\textcircled{$\ltimes$}}}
\newcommand{\ds}{\displaystyle}
\newcommand{\co}{\;:\;}
\DeclareMathOperator*{\colim}{colim}
\DeclareMathOperator*{\hocolim}{hocolim}

\begin{document}

\begin{abstract}    % type your abstract below

  The symmetric homology of a unital algebra $A$ over a commutative
  ground ring $k$ is defined using derived functors and the symmetric
  bar construction of Fiedorowicz.  For a group ring $A = k[\Gamma]$,
  the symmetric homology is related to stable homotopy theory via
  $HS_*(k[\Gamma]) \cong H_*(\Omega\Omega^{\infty}
  S^{\infty}(B\Gamma); k)$.  Two chain complexes that compute
  $HS_*(A)$ are constructed, both making use of a symmetric monoidal
  category $\Delta S_+$ containing $\Delta S$.  Two spectral sequences
  are found that aid in computing symmetric homology.  The second
  spectral sequence is defined in terms of a family of complexes,
  $Sym^{(p)}_*$.  $Sym^{(p)}$ is isomorphic to the suspension of the
  cycle-free chessboard complex $\Omega_{p+1}$ of Vre\'{c}ica and
  \v{Z}ivaljevi\'{c}, and so recent results on the connectivity of
  $\Omega_n$ imply finite-dimensionality of the symmetric homology
  groups of finite-dimensional algebras.  Some results about the
  $k\Sigma_{p+1}$--module structure of $Sym^{(p)}$ are devloped.  A
  partial resolution is found that allows computation of $HS_1(A)$ for
  finite-dimensional $A$ and some concrete computations are included.

\end{abstract}

\maketitle

%%%%%%%%%%%%%%%%%%%%   Start of main body of article

%%%%%%%%%%%%%%%%%%%%%%%%%%%%%%%%%%%%%%%%%%%%%%%%%%%%%%%%%%%%%%%%%%%%%%%%%%%%%%%%
\section{Introduction and Definitions}
%%%%%%%%%%%%%%%%%%%%%%%%%%%%%%%%%%%%%%%%%%%%%%%%%%%%%%%%%%%%%%%%%%%%%%%%%%%%%%%%

The theory of symmetric homology, in which the symmetric groups
$\Sigma_k^\mathrm{op}$, for $k \geq 0$, play the role that the cyclic
groups do in cyclic homology, begins with the definition of the
category $\Delta S$, containing the simplicial category $\Delta$ as
subcategory.  Indeed, $\Delta S$ is an example of {\it crossed
  simplicial group}~\cite{FL}.  One develops a notion of a bar
resolution over crossed simplicial groups by analogy with the cyclic
bar resolution.  However, the na\"ive contravariant symmetric bar
resolution produces trivial results.  Fiedorowicz found that a {\it
  covariant} symmetric bar construction, $B^{sym}_*$, produces a
non-trivial and important homology theory~\cite{F}.  One defines the
symmetric homology of an algebra $A$ by
\begin{definition} $\ds{HS_*(A) = \mathrm{Tor}^{\Delta S}_*( \underline{k}, 
B^{sym}_*A )}$.
\end{definition}
Almost 20 years ago, Fiedorowicz found but did not publish the following 
remarkable result~\cite{F}.
\begin{theorem}[Fiedorowicz's Theorem]\label{thm.HS_group}
  If $\Gamma$ is a group, then
  \[
    HS_*(k[\Gamma]) \cong H_*\left(\Omega\Omega^{\infty}
    S^{\infty}(B\Gamma); k\right).
  \]
\end{theorem}

This formula shows in particular that $HS_*$ is an important and
non-trivial theory.  While it is true that
$H_*(\Omega^{\infty}S^{\infty}X) = H_*(QX)$ is well understood, the
same cannot be said of the homology of $\Omega\Omega^{\infty}
S^{\infty}X$.  Indeed, $H_*(QX)$ has been studied extensively by
Cohen, Lada and May~\cite{CLM}.  $H_*(QX)$ may be regarded as
the free allowable $AR$-Hopf algebra with conjugation generated by
$H_*(X)$.  On the other hand, results about the homology of $\Omega
\Omega^{\infty}S^{\infty}X$ are fewer and farther between.  Cohen and
Peterson~\cite{CP} computed $H_*(\Omega\Omega^{\infty}S^{\infty})$ ({\it
  i.e.}, the case when the space $X$ is the zero-sphere, $S^0$), but
there is little hope of extending this result to arbitrary $X$ using
the same methods.
 
We find that adjoining a ``unit'' to $\Delta S$ results in a
permutative category $\Delta S_+$.  This step is necessary in order
to prove Thm.~\ref{thm.HS_group} and related theorems.  

By reducing the standard
resolution that computes $HS_*(A)$ to one that involves only
epimorphisms of $\Delta S_+$, we develop two spectral sequences
abutting to $HS_*(A)$.  The first spectral sequence is based on the
work of S\l{}omi\'nska~\cite{Sl} on $E$-$I$-categories, and relates
symmetric homology to the homology of the symmetric groups.

The second spectral sequence makes use of a family of complexes,
$Sym_*^{(p)}$ which distill the relevent combinatorial data of the
nerve of $\mathrm{Epi}_{\Delta S}$.

\begin{theorem}\label{thm.SpecSeq2}
If $A$ has an augmentation ideal $I$ which is free as $k$--module, with
countable basis $X$, then there is a spectral sequence $E^1_{p,q} 
\Rightarrow \widetilde{H}S_{p+q}(A)$, with
\[
  E^1_{p,q} \cong \bigoplus_{u \in X^{p+1}/\Sigma_{p+1}}
  H_{p+q}\left(E_*G_u \ltimescirc_{G_u} Sym_*^{(p)}; k\right),
\]
where $G_{u}$ is the isotropy subgroup for the chosen representative
of $u \in X^{p+1}/ \Sigma_{p+1}$ and the symbol $\ltimescirc_{G_u}$
stands for the chain complex analog of equivariant half-smash
product for spaces, $\ltimes_{G_u}$.
\end{theorem}
The multiplicative structure of $A$ becomes encapsulated in the
differential $d^1_{p,q}$ on $E^1$, and so results about $HS_*(A)$ for
general algebras $A$ will follow from the structure of $Sym_*^{(p)}$.
There is an isomorphism of complexes, $k\left[ S\Omega^+_{p+1}\right]
\stackrel{\cong}{\longrightarrow} Sym_*^{(p)}$, where $\Omega_n^+$ is
the augmented cycle-free $(n \times n)$-chessboard complex of
Vre\'{c}ica and \v{Z}ivaljevi\'{c}~\cite{VZ}, which produces the
immediate important corollary:
\begin{cor}\label{cor.fin-gen}
  If $A$ is finitely-generated over a Noetherian ground ring $k$, then
  $HS_*(A)$ is finitely-generated over $k$ in each degree.
\end{cor}

In the final section, we develop a partial resolution of the trivial
$\Delta S^{\mathrm{op}}$--module $\underline{k}$ by projective modules,
leading to the following:
\begin{theorem}\label{thm.partial_resolution}
  $HS_i(A)$ for $i=0,1$ may be computed as the degree $0$ and degree $1$ 
  homology groups of the following (partial) chain complex:
  \[
    0\longleftarrow A \stackrel{\partial_1}{\longleftarrow} A\otimes
    A\otimes A \stackrel{\partial_2}{\longleftarrow}(A\otimes A\otimes
    A\otimes A)\oplus A,
  \]
  where 
  \[
    \partial_1 \co a\otimes b\otimes c \mapsto abc - cba,
  \]
  \[
    \partial_2 \co \left\{\begin{array}{lll} a\otimes b\otimes c\otimes
    d &\mapsto& ab\otimes c\otimes d + d\otimes ca\otimes b  \\
    &&\quad + bca\otimes 1\otimes d + d\otimes bc\otimes a,\\ a &\mapsto&
    1\otimes a\otimes 1.
       \end{array}\right.
  \]
\end{theorem}
In particular, we see that $HS_0(A) = A/([A,A])$, where $([A,A])$ is
the ideal generated by the commutator submodule $[A,A]$.  In other
words, $HS_0(A)$ is the symmetrization of $A$ as an algebra.  Compare
with the zeroth Hochschild or Cyclic homology of $A$, which is
$A/[A,A]$.

%%%%%%%%%%%%%%%%%%%%%%%%%%%%%%%%%%%%%%%%%%%%%%%%%%%%%%%%%%%%%%%%%%%%%%%%%%%%%%
\subsection{The category $\Delta S$}

Let $\Delta S$ be the category that has as objects, the ordered sets
$[n] = \{0, 1, \ldots, n\}$ for $n \geq 0$, and as morphisms, pairs
$(\phi, g)$, where $\phi \co [n] \to [m]$ is a non-decreasing map of
sets (\textit{i.e.}, a morphism in $\Delta$), and $g \in
\Sigma_{n+1}^{\mathrm{op}}$ (the opposite group of the symmetric group
acting on $[n]$).  The element $g$ represents an automorphism of
$[n]$, and as a set map, takes $i \in [n]$ to $g^{-1}(i)$.
Equivalently, a morphism in $\Delta S$ is a morphism in $\Delta$
together with a total ordering of the domain $[n]$.  Composition of
morphisms is achieved as in~\cite{FL}, namely, $(\phi, g) \circ (\psi,
h) = (\phi \cdot g^*(\psi), \psi^*(g) \cdot h)$.  Observe that the
properties of $g^*(\phi)$ and $\phi^*(g)$ stated in Prop.~1.6
of~\cite{FL} are formally rather similar to the properties of
exponents (except for properties 2.h and 2.v).  Indeed, the notation
$g^\phi \stackrel{def}{=} \phi^*(g), \phi^g \stackrel{def}{=}
g^*(\phi)$ will generally be used in lieu of the original notation in
what follows.  For reference, we restate Prop.~1.6 of~\cite{FL} using
the ``exponent'' notation:
\begin{prop}
  If $G_*$ is a crossed simplicial group, for $g, h \in G_n$, $\phi \in 
  \mathrm{Mor}_\Delta([m], [n])$, $\psi \in \mathrm{Mor}_\Delta([p], [m])$,
  \begin{eqnarray*}
      (1.h)' &\quad& g^{\phi\psi} = (g^\phi)^\psi \\
      (1.v)' &\quad& \phi^{gh} = (\phi^g)^h \\
      (2.h)' &\quad& (\phi\psi)^g = \phi^g\psi^{(g^\phi)} \\
      (2.v)' &\quad& (gh)^{\phi} = g^\phi h^{(\phi^g)} \\
      (3.h)' &\quad& g^{\mathrm{id}_n} = g \quad \textrm{and} \quad 1^\phi = 1 \\
      (3.v)' &\quad& \phi^{1} = \phi \quad \textrm{and} \quad \mathrm{id}_n^g = 
        \mathrm{id} \\
  \end{eqnarray*}
\end{prop}

It is often helpful to represent morphisms of $\Delta S$ as diagrams of points
and lines, indicating images of set maps.  Using these diagrams, we may see 
clearly how $\psi^g$ and $g^\psi$ are related to $\psi$ and $g$ (see 
Figure~\ref{diag.morphism}).

\begin{figure}[ht]
  \psset{unit=.75in}
  \begin{pspicture}(4,4)
  
  \psdots[linecolor=black, dotsize=4pt]
  (0.3, 2.4)(0.6, 2.7)(0.9, 3.0)(1.2, 3.3)(1.5, 3.6)
  (2.7, 2.4)(3.0, 2.7)(3.3, 3.0)(3.6, 3.3)
  (2.7, 0.3)(3.0, 0.6)(3.3, 0.9)(3.6, 1.2)
  (0.3, 0.3)(0.6, 0.6)(0.9, 0.9)(1.2, 1.2)(1.5, 1.5)
  
  \psline[linewidth=1pt, linecolor=black](0.3, 2.4)(0.3, 0.3)
  \psline[linewidth=1pt, linecolor=black](0.6, 2.7)(0.6, 0.6)
  \psline[linewidth=1pt, linecolor=black](0.9, 3.0)(1.5, 1.5)
  \psline[linewidth=1pt, linecolor=black](1.2, 3.3)(0.9, 0.9)
  \psline[linewidth=1pt, linecolor=black](1.5, 3.6)(1.2, 1.2)

  \psline[linewidth=1pt, linecolor=black](0.3, 0.3)(2.7, 0.3)
  \psline[linewidth=1pt, linecolor=black](0.6, 0.6)(2.7, 0.3)
  \psline[linewidth=1pt, linecolor=black](0.9, 0.9)(3.0, 0.6)
  \psline[linewidth=1pt, linecolor=black](1.2, 1.2)(3.0, 0.6)
  \psline[linewidth=1pt, linecolor=black](1.5, 1.5)(3.6, 1.2)

  \psline[linewidth=1pt, linecolor=black](0.3, 2.4)(2.7, 2.4)
  \psline[linewidth=1pt, linecolor=black](0.6, 2.7)(2.7, 2.4)
  \psline[linewidth=1pt, linecolor=black](0.9, 3.0)(3.0, 2.7)
  \psline[linewidth=1pt, linecolor=black](1.2, 3.3)(3.6, 3.3)
  \psline[linewidth=1pt, linecolor=black](1.5, 3.6)(3.6, 3.3)

  \psline[linewidth=1pt, linecolor=black](2.7, 2.4)(2.7, 0.3)
  \psline[linewidth=1pt, linecolor=black](3.0, 2.7)(3.6, 1.2)
  \psline[linewidth=1pt, linecolor=black](3.3, 3.0)(3.3, 0.9)
  \psline[linewidth=1pt, linecolor=black](3.6, 3.3)(3.0, 0.6)

  \rput(0.2, 2.4){$0$}
  \rput(0.5, 2.7){$1$}
  \rput(0.8, 3.0){$2$}
  \rput(1.1, 3.3){$3$}
  \rput(1.4, 3.6){$4$}

  \rput(0.2, 0.3){$0$}
  \rput(0.5, 0.6){$1$}
  \rput(0.8, 0.9){$2$}
  \rput(1.1, 1.2){$3$}
  \rput(1.4, 1.5){$4$}
  
  \rput(2.6, 2.5){$0$}
  \rput(2.9, 2.8){$1$}
  \rput(3.2, 3.1){$2$}
  \rput(3.5, 3.4){$3$}
  
  \rput(2.6, 0.4){$0$}
  \rput(2.9, 0.7){$1$}
  \rput(3.2, 1.0){$2$}
  \rput(3.5, 1.3){$3$}
    
  \rput(0.1, 1.2){$g^\psi$}
  \rput(2.4, 3.7){$\psi$}
  \rput(3.6, 2.3){$g$}
  \rput(1.8, 0.1){$\psi^g$}

  \end{pspicture}

  \caption[Morphisms of $\Delta S$]{Morphisms of $\Delta S$: $g \cdot \psi = 
  \big(g^*(\psi), \psi^*(g)\big)  = \big(\psi^g, g^{\psi}\big)$}
  \label{diag.morphism}
\end{figure}
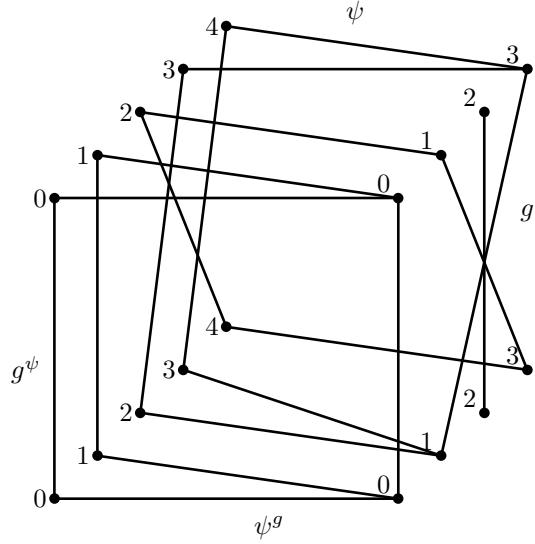

An equivalent characterization of $\Delta S$ comes from Pirashvili, as
the category $\mathcal{F}(\mathrm{as})$ of {\it non-commutative
  sets}~\cite{P}.  The objects are sets $\underline{n}
\stackrel{def}{=} \{1, 2, \ldots, n\}$ for $n \geq 0$.  By convention,
$\underline{0}$ is the empty set.  A morphism in
$\mathrm{Mor}_{\mathcal{F}(\mathrm{as})} (\underline{n},
\underline{m})$ consists of a set map $f \co \underline{n} \to
\underline{m}$ together with a total ordering on each preimage set
$f^{-1}(j)$.  There is an obvious inclusion of categories, $\Delta S
\hookrightarrow \mathcal{F}(\mathrm{as})$, taking $[n]$ to
$\underline{n+1}$, but there is no object of $\Delta S$ that maps to
$\underline{0}$.  It will be useful to define $\Delta S_+ \supset
\Delta S$ which is isomorphic to $\mathcal{F}(\mathrm{as})$:
\begin{definition}
  $\Delta S_+$ is the category consisting of all objects and morphisms of
  $\Delta S$, with the additional object $[-1]$, representing the empty set,
  and a unique morphism $\iota_n \co [-1] \to [n]$ for each $n \geq -1$.
\end{definition}
\begin{rmk}
  Pirashvili's construction is a special case of a more general
  construction due to May and Thomason \cite{MT}. This construction
  associates to any topological operad $\{\mathcal{C}(n)\}_{n \geq 0}$
  a topological category $\widehat{\mathcal{C}}$ together with a
  functor $\widehat{\mathcal{C}} \to \mathcal{F}$, where $\mathcal{F}$
  is the category of finite sets, such that the inverse image of any
  function $f \co \underline{m} \to \underline{n}$ is the space
  $\prod_{i=1}^n \mathcal{C}(\# f^{-1}(i))$.  Composition in
  $\widehat{\mathcal{C}}$ is defined using the composition of the
  operad. May and Thomason refer to $\widehat{\mathcal{C}}$ as the
  {\it category of operators} associated to $\mathcal{C}$. They were
  interested in the case of an $E_\infty$ operad, but their
  construction evidently works for any operad.  The category of
  operators associated to the discrete $A_\infty$ operad
  $\mathcal{A}ss$, which parametrizes monoid structures, is precisely
  Pirashvili's construction of $\mathcal{F}(as)$, {\it i.e.} $\Delta
  S_+$.
\end{rmk}
One very useful advantage in enlarging our category to $\Delta S$ to 
$\Delta S_+$ is the added structure inherent in $\Delta S_+$.
\begin{prop}\label{prop.deltaSpermutative}
  $\Delta S_+$ is a permutative category.
\end{prop}
\begin{proof}
  Define the monoid product on objects by $[n] \odot [m]
  \stackrel{def}{=} [n+m+1]$, (disjoint union of sets), and on
  morphisms $(\phi, g) \co [n] \to [n']$, $(\psi, h) \co [m] \to [m']$, by
  $(\phi,g) \odot (\psi,h) = (\eta, k) \co [n+m+1] \to [n'+m'+1]$, where
  $(\eta, k)$ is just the morphism $(\phi, g)$ acting on the first
  $n+1$ points of $[n+m+1]$, and $(\psi, h)$ acting on the remaining
  points.
  
  The unit object will be $[-1] = \emptyset$. $\odot$ is clearly associative, 
  and $[-1]$ acts as two-sided identity.    Finally, define the transposition 
  transformation $\gamma_{n,m} \co [n] \odot [m] \to [m] \odot [n]$ to be the 
  identity on objects, and on morphisms to be precomposition with the block 
  transposition that switches the first block of size $n+1$ with the second 
  block of size $m+1$.
\end{proof}

\begin{rmk}
  The fact that $\Delta S_+$ is permutative shall be exploited to prove that 
  $HS_*(A)$ admits homology operations in a forthcoming paper.
\end{rmk}

\begin{rmk}
  It will become convenient to include the object $[-1] = \emptyset$ in
  $\Delta$ as well.  Denote the enlarged category by $\Delta_+$.  
\end{rmk}

For the purposes of computation, a morphism $\alpha \co [n] \to [m]$ of
$\Delta S$ may be conveniently represented as a tensor product of
monomials in the formal non-commuting variables $\{x_0, x_1, \ldots,$
$x_n\}$.  Let $\alpha = (\phi, g)$, with $\phi \in
\mathrm{Mor}_\Delta([n],[m])$ and $g \in \Sigma_{n+1}^\mathrm{op}$.
The tensor representation of $\alpha$ will have $m + 1$ tensor
factors.  Each $x_i$ will occur exactly once, in the order $x_{g(0)},
x_{g(1)}, \ldots, x_{g(n)}$.  The $i^{th}$ tensor factor consists of
the product of $\#\phi^{-1}(i-1)$ variables, with the convention that
the empty product will be denoted $1$.  Thus, the $i^{th}$ tensor
factor records the total ordering of $\phi^{-1}(i)$.  As an example,
the tensor representation of the morphism depicted in
Fig.~\ref{diag.morphism-tensor} is $x_1x_0 \otimes x_3x_4 \otimes 1
\otimes x_2$.  With this notation, the composition of two morphisms
$\alpha = X_0 \otimes X_1 \otimes \ldots \otimes X_m \co [n] \to [m]$
and $\beta = Y_1 \otimes Y_2 \otimes \ldots Y_n \co [p] \to [n]$ is
given by, $\alpha \beta = Z_0 \otimes Z_1 \otimes \ldots \otimes Z_m$,
where $Z_i$ is determined by replacing each variable in the monomials
$X_i = x_{j_1} \ldots x_{j_s}$ by the corresponding monomials
$Y_{j_k}$ in $\beta$.  So, $Z_i = Y_{j_1} \ldots Y_{j_s}$.

\begin{figure}[ht]
  \psset{unit=.5in}
  \begin{pspicture}(4,4)
  
  \psdots[linecolor=black, dotsize=4pt]
  (0.3, 2.4)(0.6, 2.7)(0.9, 3.0)(1.2, 3.3)(1.5, 3.6)
  (2.7, 0.3)(3.0, 0.6)(3.3, 0.9)(3.6, 1.2)
  (0.3, 0.3)(0.6, 0.6)(0.9, 0.9)(1.2, 1.2)(1.5, 1.5)
  
  \psline[linewidth=1pt, linecolor=black](0.3, 2.4)(0.6, 0.6)
  \psline[linewidth=1pt, linecolor=black](0.6, 2.7)(0.3, 0.3)
  \psline[linewidth=1pt, linecolor=black](0.9, 3.0)(1.5, 1.5)
  \psline[linewidth=1pt, linecolor=black](1.2, 3.3)(0.9, 0.9)
  \psline[linewidth=1pt, linecolor=black](1.5, 3.6)(1.2, 1.2)

  \psline[linewidth=1pt, linecolor=black](0.3, 0.3)(2.7, 0.3)
  \psline[linewidth=1pt, linecolor=black](0.6, 0.6)(2.7, 0.3)
  \psline[linewidth=1pt, linecolor=black](0.9, 0.9)(3.0, 0.6)
  \psline[linewidth=1pt, linecolor=black](1.2, 1.2)(3.0, 0.6)
  \psline[linewidth=1pt, linecolor=black](1.5, 1.5)(3.6, 1.2)

  \rput(0.15, 2.4){$0$}
  \rput(0.45, 2.7){$1$}
  \rput(0.75, 3.0){$2$}
  \rput(1.05, 3.3){$3$}
  \rput(1.35, 3.6){$4$}

  \rput(0.15, 0.3){$0$}
  \rput(0.45, 0.6){$1$}
  \rput(0.75, 0.9){$2$}
  \rput(1.05, 1.2){$3$}
  \rput(1.35, 1.5){$4$}
  
  \rput(2.55, 0.4){$0$}
  \rput(2.85, 0.7){$1$}
  \rput(3.15, 1.0){$2$}
  \rput(3.45, 1.3){$3$}
    
  \rput(2.7, 2.4){$x_1x_0 \otimes x_3x_4 \otimes 1 \otimes x_2$}

  \end{pspicture}

  \caption[Morphisms of $\Delta S$ as Tensors]{Morphisms of $\Delta S$ in tensor
  notation}
  \label{diag.morphism-tensor}
\end{figure}
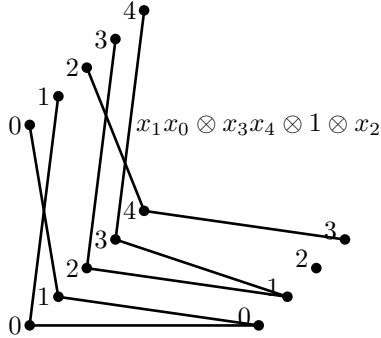

%%%%%%%%%%%%%%%%%%%%%%%%%%%%%%%%%%%%%%%%%%%%%%%%%%%%%%%%%%%%%%%%%%%%%%%%%%%%%%
\subsection{Homological Algebra of Functors}\label{sec.hom_alg_functors}

Recall, for a category $\mathscr{C}$, a \mbox{$\mathscr{C}$--module} is
covariant functor $F \co \mathscr{C} \to
\mbox{\textrm{$k$--\textbf{Mod}}}$.  Similarly, a
$\mathscr{C}^\mathrm{op}$--module is a contravariant functor $G \co
\mathscr{C} \to \mbox{\textrm{$k$--\textbf{Mod}}}$.  Let $M$ be a
$\mathscr{C}^\mathrm{op}$--module and $N$ be a
\mbox{$\mathscr{C}$--module}.  Following MacLane~\cite{ML}, define the
tensor product of functors as a coend, $M \otimes_\mathscr{C} N =
\int^X (MX) \otimes (NX)$.  That is,
\[
  M \otimes_{\mathscr{C}} N = \bigoplus_{X \in \mathrm{Obj}\mathscr{C}}
  M(X) \otimes_k N(X) / \approx,
\]
where the equivalence $\approx$ is generated by $y \otimes f_*(x)
\approx f^*(y) \otimes x$ for $f \in \mathrm{Mor}_{\mathscr{C}}(X,
Y)$, $x \in N(X)$ and $y \in M(Y)$.

\begin{rmk}
  Note, the existing literature based on the work of Connes, Loday and
  Quillen consistently defines the categorical tensor product in the
  reverse sense: $N \otimes_{\mathscr{C}} M$ is the direct sum of
  copies of $NX \otimes_k MX$ modded out by the equivalence $x \otimes
  f^*(y) \approx f_*(x) \otimes y$ for all $\mathscr{C}$-morphisms $f
  \co X \to Y$.  In this context, $N$ is covariant, while $M$ is
  contravariant.  I chose to follow the convention of Pirashvili and
  Richter \cite{PR} in writing tensor products as $M
  \otimes_{\mathscr{C}} N$ so that the equivalence $\xi \co
  $\mbox{$\mathscr{C}$--\textbf{Mod}} $\to$
  \mbox{$k[\mathrm{Mor}\mathscr{C}]$--\textbf{Mod}} passes to tensor
  products in a straightforward way: $\xi( M \otimes_{\mathscr{C}} N )
  = \xi(M) \otimes_{k[\mathrm{Mor}\mathscr{C}]} \xi(N)$.
\end{rmk}

The {\it trivial} $\mathscr{C}$--module,
resp.~$\mathscr{C}^\mathrm{op}$--module, denoted by $\underline{k}$
(for either variance), is the functor taking each object to $k$ and
each morphism to the identity.  As noted in~\cite{FL}, the category of
$\mathscr{C}$--modules is abelian and has enough projectives.  For any
$\mathscr{C}^{\mathrm{op}}$--module $M$, the functor $N \mapsto M
\otimes_{\mathscr{C}} N$ is right-exact and so admits derived functors
$\mathrm{Tor}_n^{\mathscr{C}}(M, -)$ such that
$\mathrm{Tor}_0^{\mathscr{C}} (M, N) = M \otimes_{\mathscr{C}} N$.

%%%%%%%%%%%%%%%%%%%%%%%%%%%%%%%%%%%%%%%%%%%%%%%%%%%%%%%%%%%%%%%%%%%%%%%%%%%%%%
\subsection{The Symmetric Bar Construction}\label{sub.symbar}

Now that we have defined the category $\Delta S$, the next step should
be to define an appropriate bar construction.  Recall that the {\it
  cyclic bar construction} is a functor $B^{cyc}_*A \co \Delta
C^{\mathrm{op}} \to k$--\textbf{Mod}.  One then takes the groups
$\mathrm{Tor}^{\Delta C}_n(B^{cyc}_*A, \underline{k})$ as the
definition of $HC_n(A)$ for $n \geq 0$.  However the results
of~\cite{FL} show that the cyclic bar construction does not extend to
a functor $\Delta S^{\mathrm{op}} \to k$--\textbf{Mod}.  Furthermore,
for {\it any} functor $F \co \Delta S^{\mathrm{op}} \to
k$--\textbf{Mod}, the groups $\mathrm{Tor}^{\Delta S}_n(F,
\underline{k})$ simply compute the homology of the underlying
simplicial module of $F$ (given by restricting $F$ to
$\Delta^{\mathrm{op}}$).  However, Fiedorowicz discovered that there
is a natural extension of the cyclic bar construction not to a {\it
  contravariant} functor on $\Delta S$, but to a {\it covariant}
functor~\cite{F}.
\begin{definition}\label{def.symbar}
  Let $A$ be an associative, unital algebra over a commutative ground ring $k$.
  Define a $\Delta S$--module ({\it i.e.}, a functor $\Delta S \to 
  k$--\textbf{Mod}),
  $B_*^{sym}A$ by:
  \[
    B_n^{sym}A = B_*^{sym}A[n] \stackrel{def}{=} A^{\otimes (n+1)}
  \]
  \[  
    B_*^{sym}A(\alpha) \co (a_0 \otimes a_1 \otimes \ldots \otimes a_n) \mapsto
       \alpha(a_0, \ldots, a_n),
  \]
  where $\alpha \co [n] \to [m]$ is represented in tensor notation, and
  evaluation at $(a_0, \ldots, a_n)$ simply amounts to substituting
  each $a_i$ for $x_i$ and multiplying the resulting monomials in $A$.
  If the pre-image $\alpha^{-1}(i)$ is empty, then the unit of $A$ is
  inserted.  Observe that $B_*^{sym}A$ is natural in $A$.
\end{definition}

Note that there are natural inclusions $\Delta \hookrightarrow \Delta
C \hookrightarrow \Delta S$.  The second inclusion is induced from
inclusions of groups $C_{n+1} \hookrightarrow \Sigma_{n+1}$.  To be
precise, for each $n$, let $\tau_n$ be the $(n+1)$-cycle $(0, n, n-1,
\ldots, 1) \in \Sigma_{n+1}$.  $\tau_n$ generates a subgroup
isomorphic to $C_{n+1}$.  $B_*^{sym}A$ may be regarded as a simplicial
\mbox{$k$--module} via the chain of functors, $\Delta^{\mathrm{op}}
\hookrightarrow \Delta C^{\mathrm{op}} \stackrel{\cong}{\to} \Delta C
\hookrightarrow \Delta S$.  Here, the isomorphism $D \co \Delta
C^{\mathrm{op}} \to \Delta C$ is the standard duality (see~\cite{L}).
Note that the {\it cyclic bar construction} can be recovered from the
covariant symmetric bar construction by $B_*^{cyc}A = B_*^{sym}A \circ
D$.

\begin{definition}
  The \textbf{symmetric homology} of an associative, unital $k$-algebra $A$
  is denoted $HS_*(A)$, and is defined as:
  \[
    HS_*(A) \stackrel{def}{=} \mathrm{Tor}_*^{\Delta S}\left(\underline{k},
    B_*^{sym}A\right)
  \]
\end{definition}

\begin{rmk}
  Since $\underline{k}\otimes_{\Delta S} M \cong \colim_{\Delta S}M$,
  for any $\Delta S$--module $M$, we can alternatively describe
  symmetric homology as derived functors of the colimit:
  \[
    HS_n(A) = {\colim_{\Delta S}}^{(n)} B_*^{sym}A.
  \]
\end{rmk}
In the language of Gabriel-Zisman (\cite{GZ}, Appendix II.3), $HS_*(A) = 
H_*(\Delta S, B_*^{sym}A)$.  Moreover, a simplicial module whose homology 
computes $H_*(\Delta S, B_*^{sym}A)$ is readily available (see~\cite{GZ},
p. 153).  In order to describe this simplicial module, we recall that the
{\it nerve} of a category $\mathscr{C}$ is the simplicial set 
$N\mathscr{C}$ such that $n$-chains are sequences of $n$ composable
morphisms of $\mathscr{C}$, $0$-chains are simply objects of $\mathscr{C}$,
and face maps are given by deletion of objects along with the 
composition of the corresponding morphisms.  If $F \co \mathscr{C}
\to \mathscr{D}$ is a morphism and
$\chi \in N\mathscr{C}$ is an $n$-chain,
\[
  \chi_n \stackrel{f_n}{\gets} \chi_{n-1} \stackrel{f_{n-1}}{\gets}
  \cdots \stackrel{f_2}{\gets} \chi_1 \stackrel{f_1}{\gets}\chi_0,
\]
then let $F\chi = F\chi_0$.  The following simplicial module computes 
$H_*(\Delta S, B_*^{sym}A)$:
\[
  \{C_n(\Delta S, B_*^{sym}A)\}_{n \geq 0} = \left\{\bigoplus_{\chi \in 
  N_n{\Delta S}} B_*^{sym}A\chi \right\}_{n \geq 0},
\]
with face maps and degeneracies defined by:

\begin{eqnarray*}
  d_i(\chi, x) &=& \left\{ \begin{array}{cc}
                   \left(d_0\chi, \,(f_1)_*(x)\right), & \quad i=0,\\
                   \left(d_i\chi, \,x\right), & \quad 1 \leq i \leq n
                   \end{array}\right. \\
  s_i(\chi, x) &=& (s_i\chi, x)
\end{eqnarray*}
In this formula, we have $\chi \in N_n\Delta S$ and $x \in B^{sym}_*A\chi$.

$C_n(\Delta S, B_*^{sym}A)$ may be interpreted in terms of the
categorical tensor product construction mentioned in
Section~\ref{sec.hom_alg_functors}, using the {\it under-category}
functor, $( - \setminus \Delta S) \co \Delta S^{\mathrm{op}} \to
\mathbf{cat}$, which associates to each object $[p]$ of $\Delta S$,
the category $[p]\setminus \Delta S$ of objects under $[p]$
(See~\cite{Q},~\cite{ML2}, etc.).  The $n$--chains of $N([p] \setminus
\Delta S)$ may be identified with $(n+1)$--chains of $N\Delta S$ which
have domain $[p]$.  Let $k[ - ]$ be the functor that associates to any
simplicial set the corresponding simplicial $k$--module.  We may
identify,
\[
  C_*(\Delta S, B^{sym}_*A) = k\left[N(- \setminus \Delta S)\right] 
    \otimes_{\Delta S} B_*^{sym}A
\]

Indeed, observe that every element of $k\left[N(-\setminus \Delta S)\right]
\otimes_{\Delta S} B^{sym}_*A$ of the form
\[
  \left(
  \begin{diagram}
    \node[4]{[p]}
    \arrow{sw}
    \arrow{swww}
    \arrow{s,r}{\alpha}\\
    \node{[q_n]}
    \node{\cdots}
    \arrow{w,t}{\beta_n}
    \node{[q_1]}
    \arrow{w,t}{\beta_2}
    \node{[q_0]}
    \arrow{w,t}{\beta_1}
  \end{diagram}
  \right) \otimes x
\]
is equivalent to one in which the incoming morphism $\alpha$ becomes an 
identity:
\[
  \left(
  \begin{diagram}
    \node[4]{[q_0]}
    \arrow{sw}
    \arrow{swww}
    \arrow{s,r}{\mathrm{id}}\\
    \node{[q_n]}
    \node{\cdots}
    \arrow{w,t}{\beta_n}
    \node{[q_1]}
    \arrow{w,t}{\beta_2}
    \node{[q_0]}
    \arrow{w,t}{\beta_1}
  \end{diagram}
  \right) \otimes \alpha_*(x),
\]
the latter element being identified with 
\[
  \left( [q_n] \stackrel{\beta_n}{\gets} \cdots
  \stackrel{\beta_2}{\gets} [q_1] \stackrel{\beta_1}{\gets} [q_0], \,
  \alpha_*(x)\right) \in C_n(\Delta S, B_*^{sym}A).
\]

\begin{prop}\label{prop.SymHomComplex}
  For an associative, unital $k$-algebra $A$,
  \[  
    HS_*(A) = H_*\left( k[N(- \setminus \Delta S)] \otimes_{\Delta S} 
    B_*^{sym}A;\,k  \right).
  \]
\end{prop}

\begin{rmk}\label{rmk.HC}
  By duality of $\Delta C$, it is clear that the related complex
  $k[N(- \setminus \Delta C)] \otimes_{\Delta C} B_*^{sym}A$ computes
  $HC_*(A)$, where we understand that the functor $B_*^{sym}A$ is
  restricted to $\Delta C \hookrightarrow \Delta S$.
\end{rmk}

%%%%%%%%%%%%%%%%%%%%%%%%%%%%%%%%%%%%%%%%%%%%%%%%%%%%%%%%%%%%%%%%%%%%%%%%%%%%%%
\subsection{Symmetric Homology of the Ground Ring}

We now have enough tools to compute $HS_*(k)$.  First, we need to show:
\begin{lemma}\label{lem.DeltaScontractible}
  $N(\Delta S)$ is contractible.
\end{lemma}
\begin{proof}
  Define a functor $\mathscr{F} \co \Delta S \to \Delta S$ on objects by 
  $\mathscr{F}[n] = [0] \odot [n]$, and on morphisms by $\mathscr{F}f = 
  \mathrm{id}_{[0]} \odot f$,  using 
  the monoid multiplication $\odot$ defined in 
  Prop.~\ref{prop.deltaSpermutative}.
  There is a natural transformation $\mathrm{id}_{\Delta S} \to \mathscr{F}$
  given by the following commutative diagram for each $f \co [m] \to [n]$:
  \[
    \begin{diagram}%1
      \node{ [m] }
      \arrow{e,t}{f}
      \arrow{s,l}{\delta_0}
      \node{ [n] }
      \arrow{s,r}{\delta_0}\\
      \node{ [m+1] }
      \arrow{e,t}{\mathrm{id} \odot f}
      \node{ [n+1] }
    \end{diagram}
  \]
  Here, $\delta^{(k)}_j \co [k-1] \to [k]$ is the $\Delta$ morphism that misses 
  the point $j \in [k]$.
  
  Consider the constant functor $\Delta S \stackrel{[0]}{\to} \Delta S$ that
  sends all objects to $[0]$ and all morphisms to $\mathrm{id}_{[0]}$.  There is
  a natural transformation $[0] \to \mathscr{F}$
  given by the following commutative diagram for each $f \co [m] \to [n]$.
  \[
    \begin{diagram}%2
      \node{ [0] }
      \arrow{e,t}{\mathrm{id}}
      \arrow{s,l}{0_0}
      \node{ [0] }
      \arrow{s,r}{0_0}\\
      \node{ [m+1] }
      \arrow{e,t}{\mathrm{id} \odot f}
      \node{ [n+1] }
    \end{diagram}
  \]
  Here, $0^{(k)}_j \co [0] \to [k]$ is the morphism that sends the point $0$ to
  $j \in [k]$.
  
  Natural transformations induce homotopy equivalences (see~\cite{Se} or
  Prop.~1.2 of~\cite{Q}), so in particular, the identity map on $N(\Delta S)$ is
  homotopic to the map that sends $N(\Delta S)$ to the nerve of a trivial
  category.  Thus, $N(\Delta S)$ is contractible.
\end{proof}

\begin{cor}\label{cor.HS_of_k}
  The symmetric homology of the ground ring $k$ is isomorphic to $k$, 
  concentrated in degree $0$.
\end{cor}
\begin{proof}
  $HS_*(k)$ is the homology of the chain complex generated (freely) over $k$ by
  the chains
  \[
    \left\{ [q_n]\stackrel{\beta_{n}}{\gets} \cdots 
    \stackrel{\beta_2}{\gets} [q_1] \stackrel{\beta_1}{\gets}[q_0] \,\otimes\,
    (1 \otimes \ldots \otimes 1) \right\},
  \]
  where $\beta_i \in \mathrm{Mor}_{\Delta S}\left( [q_{i-1}], [q_i] \right)$.  
  Each such chain may be identified with the chain $[q_n] \stackrel{\beta_n}
  {\gets} \cdots \stackrel{\beta_2}{\gets} [q_1] \stackrel{\beta_1}{\gets}
  [q_0]$ of $N(\Delta S)$, and this defines a chain isomorphism to 
  $N(\Delta S)$. The result now follows from Lemma~\ref{lem.DeltaScontractible}.
\end{proof}

%%%%%%%%%%%%%%%%%%%%%%%%%%%%%%%%%%%%%%%%%%%%%%%%%%%%%%%%%%%%%%%%%%%%%%%%%%%%%%%%
\section{Symmetric Homology with Coefficients and the UCT}\label{sec.coeff}
%%%%%%%%%%%%%%%%%%%%%%%%%%%%%%%%%%%%%%%%%%%%%%%%%%%%%%%%%%%%%%%%%%%%%%%%%%%%%%%%

Following the conventions for Hochschild and cyclic homology in Loday~\cite{L},
when we need to indicate explicitly the ground ring $k$ over which we compute
symmetric homology of $A$, we shall use the notation:  $HS_*(A\;|\;k)$.  On
the other hand, for a $k$--module $M$, $HS_*(A ; M)$ will denote the homology
of the complex $C_*(\Delta S, B_*^{sym}A) \otimes_k M$.

Two easy propositions are collected here for reference.
\begin{prop}

\hfill

  \begin{enumerate}
    \item If $M$ is flat over $k$, then $HS_*(A ; M) \cong HS_*(A) \otimes_k M$.
    \item If $B$ is a commutative $k$-algebra, then
      $HS_*(A \otimes_k B \;|\; B) \cong HS_*(A ; B)$.
  \end{enumerate}
\end{prop}
\begin{proof}
  Left to the reader.
\end{proof}

\begin{theorem}[Universal Coefficient Theorem]\label{thm.univ.coeff.}
  If $A$ is a flat $k$-algebra, and $B$ is a commutative $k$-algebra, 
  then there is a spectral sequence with
  \[
    E_2^{p,q} = \mathrm{Tor}^k_p\left( HS_q( A \;|\; k) , B \right) 
    \Rightarrow HS_*( A \;|\; B).
  \]
\end{theorem}
\begin{proof}
  Let $T_q \co k\textrm{-$\mathbf{Mod}$} \to k\textrm{-$\mathbf{Mod}$}$
  be the functor $HS_q( A ; - )$.  Observe, since $A$ is flat,
  $\{T_q\}$ is a long exact sequence of additive covariant functors
  (See~\cite{D}, Definition~1.1 and also~\cite{Mc}, section~12.1 for
  details).  $T_q = 0$ for sufficiently small $q$ (indeed, for $q <
  0$) and $T_q$ commutes with arbitrary direct sums.  Hence, by the
  Universal Coefficient Theorem of Dold (2.12 of~\cite{D}.  See also
  McCleary~\cite{Mc}, Thm.~12.11), there is a spectral sequence with
  $E_2^{p,q} = \mathrm{Tor}^k_p\left( T_q(k) , B \right) \Rightarrow
  T_*(B)$.
\end{proof}

As an immediate consequence, we have the following result.
\begin{cor}\label{cor.iso_in_HS_with_coeff.}
  If $f \co A \to A'$ is a $k$-algebra map between flat algebras which
  induces an isomorphism in symmetric homology, $HS_*(A)
  \stackrel{\cong}{\to} HS_*(A')$, then for a commutative $k$-algebra
  $B$, the chain map $f \otimes \mathrm{id}_B$ induces an
  isomorphism $HS_*(A;B) \stackrel{\cong}{\to} HS_*(A' ; B)$.
\end{cor}

Under stronger hypotheses, the universal coefficient spectral sequence
reduces to short exact sequences.

\begin{cor}\label{cor.univ.coeff.ses}
  If $k$ has weak global dimension $\leq 1$, then the spectral sequence of 
  Thm.~\ref{thm.univ.coeff.} reduces to short exact sequences,
  \[
    0 \longrightarrow HS_n(A\;|\;k) \otimes_k B \longrightarrow
    HS_n(A ; B) \longrightarrow \mathrm{Tor}^k_1( HS_{n-1}(A \;|\;k), B)
    \longrightarrow 0.
  \]
  Moreover, if $k$ is hereditary and and $A$ is projective over $k$,
  then these sequences split (unnaturally). (See~\cite{I} for precise
  definitions of weak global dimension and hereditary).
\end{cor}
\begin{proof}
  Assume first that $k$ has weak global dimension $\leq 1$.  So
  $\mathrm{Tor}_p^k(HS_q(A \; |\; k), B) = 0$ for all $p > 1$.
  Following Dold's argument (Corollary~2.13 of~\cite{D}), we obtain
  the required exact sequences.  Assume further that $k$ is hereditary
  and $A$ is projective.  Then Theorem~8.22 of Rotman~\cite{R3} then
  gives us the desired splitting.
\end{proof}

\begin{rmk}
  The proof given above also proves UCT for cyclic homology.  A
  partial result along these lines exists in Loday (\cite{L}, 2.1.16).
  There, he shows $HC_*(A\;|\; k) \otimes_k K \cong HC_*(A \;|\; K)$
  and \mbox{$HH_*(A\;|\; k) \otimes_k K \cong$} \mbox{$HH_*(A \;|\;
    K)$} in the case that $K$ is a localization of $k$, and $A$ is a
  $K$--module, flat over $k$.  I am not aware of a statement of UCT for
  cyclic or Hochschild homology in its full generality in the
  literature.
\end{rmk}

%%%%%%%%%%%%%%%%%%%%%%%%%%%%%%%%%%%%%%%%%%%%%%%%%%%%%%%%%%%%%%%%%%%%%%%%%%%%%%%%
\section{Integral Symmetric Homology and a Bockstein Spectral Sequence}
%%%%%%%%%%%%%%%%%%%%%%%%%%%%%%%%%%%%%%%%%%%%%%%%%%%%%%%%%%%%%%%%%%%%%%%%%%%%%%%%

We shall obtain a converse to Cor.~\ref{cor.iso_in_HS_with_coeff.} in
the case $k = \Z$.
\begin{theorem}\label{thm.conv.HS_iso}
  Let $f \co A \to A'$ be an algebra map between torsion-free
  $\Z$-algebras.  Suppose for $B = \mathbb{Q}$ and $B = \Z/p\Z$ for
  any prime $p$, the map $f \otimes \mathrm{id}_B \co C_*(\Delta S,
  B_*^{sym}A) \otimes_k B \longrightarrow C_*(\Delta S, B_*^{sym}A')
  \otimes_k B$ induces an isomorphism $HS_*(A ; B) \to HS_*(A' ; B)$.
  Then $f$ also induces an isomorphism $HS_*(A) \stackrel{\cong}{\to}
  HS_*(A')$.
\end{theorem}

First, note that if $A$ is flat over $k$, one can construct the
Bockstein homomorphisms $\beta_n \co HS_n(A ; Z) \to HS_{n-1}(A ; X)$
associated to a short exact sequence of $k$--modules, $0 \to X \to Y
\to Z \to 0$.  These Bockstein maps are natural in $A$, since $HS_*(A;
-)$ is a long exact sequence of functors (see
Section~\ref{sec.coeff}).  Moreover if the induced map $f_*\co HS_*(A;W)
\to HS_*(A';W)$ is an isomorphism for any two of $W=X$, $W=Y$, $W=Z$,
then it is an isomorphism for the third -- an easy exercise for the
reader.  We shall now proceed with the proof of
Thm.~\ref{thm.conv.HS_iso}.  All tensor products will be over $\Z$ in
what follows.

\begin{proof}
  Let $A$ and $A'$ be torsion-free $\Z$--modules.  Let $f \co A \to A'$
  be an algebra map inducing isomorphism in symmetric homology with
  coefficients in $\mathbb{Q}$ and also in $\Z/p\Z$ for any prime
  $p$. For $m \geq 2$, there is a short exact sequence,
  \[
    0 \longrightarrow \Z/p^{m-1}\Z \stackrel{p}{\longrightarrow} \Z/p^m\Z
    \longrightarrow \Z/p\Z \longrightarrow 0.
  \]
  A straightforward induction argument shows that the maps induced
  by $f$ are isomorphisms,
  \begin{equation}\label{eq.HS_A_Zpm}
    f_* \co HS_*(A; \Z/p^{m}\Z) \stackrel{\cong}{\longrightarrow} HS_*(A'; 
    \Z/p^{m}\Z)
  \end{equation}
  
  Denote $\ds{\Z / p^\infty \Z = \lim_{\longrightarrow} \Z/ p^m\Z}$.
  Note, this is a {\it direct limit} in the sense that it is a colimit
  over a directed system.  The direct limit functor is exact
  (Prop.~5.3 of~\cite{S2}), so the maps $HS_n(A ; \Z / p^\infty \Z)
  \to HS_n(A' ; \Z / p^\infty \Z)$ induced by $f$ are isomorphisms,
  given by the chain of isomorphisms below:
  \[
    HS_n(A; \Z / p^\infty \Z)
    \cong
    \lim_{\longrightarrow}H_*\left(
      k\left[C_*(\Delta S, B_*^{sym}A)\right] \otimes \Z/p^m\Z \right)
    \stackrel{f_*}{\longrightarrow}
  \]
  \[
    \qquad\qquad\qquad
    \lim_{\longrightarrow}H_*\left(
      k\left[C_*(\Delta S, B_*^{sym}A')\right] \otimes \Z/p^m\Z\right)
    \cong
    HS_*(A'; \Z / p^\infty \Z)
  \]
  (Note, $f_*$ here stands for $\ds{\lim_{\longrightarrow}H_n(
  k\left[C_*(\Delta S, B_*^{sym}f)\right] \otimes \mathrm{id})}$.)
    
  Finally, consider the short exact sequence of abelian groups,
  \[
    0 \longrightarrow \Z \longrightarrow \mathbb{Q} \longrightarrow
    \bigoplus_{\textrm{$p$ prime}} \Z / p^\infty \Z \longrightarrow 0
  \]
    
  The isomorphism $f_* \co HS_*(A; \Z / p^\infty \Z) \to HS_*(A'; \Z / p^\infty 
  \Z)$ passes to direct sums, giving isomorphisms for each $n$,
  \[
    f_* \co HS_{n}\left(A;\, \bigoplus_p \Z/p^\infty \Z\right) \stackrel{\cong}
    {\longrightarrow} HS_{n}\left(A';\, \bigoplus_p \Z/p^\infty \Z\right).
  \]
  Since $HS_*(A; \mathbb{Q}) \to HS_*(A'; \mathbb{Q})$ is an
  isomorphism, we have the required isomorphism in symmetric homology,
  $f_* \co HS_n(A \,|\, \Z) \stackrel{\cong}{\longrightarrow}
  HS_n(A'\,|\, \Z)$.
\end{proof}

\begin{rmk}
  Theorem~\ref{thm.conv.HS_iso} may be useful for determining integral
  symmetric homology, since rational computations are generally
  simpler, and computations mod $p$ may be made easier due to the
  presence of additional structure, such as homology operations (see
  the author's forthcoming paper).
\end{rmk}

Finally, we state a result along the lines of McCleary~\cite{Mc}, Thm.~10.3.
Denote the torsion submodule of a graded module $H_*$ by 
$\tau\left(H_*\right)$.

\begin{theorem}[Bockstein spectral sequence]\label{thm.bockstein_spec_seq}
  Suppose $A$ is free of finite rank over $\Z$.  Then there is a singly-graded
  spectral sequence with
  \[
    E_*^1 = HS_*( A ; \Z/p\Z) \Rightarrow HS_*(A)/\tau\left(HS_*(A) \right)
    \otimes \Z/p\Z,
  \]
  with differential map $d^1 = \beta$, the standard Bockstein map associated to
  $0 \to \Z / p\Z \to \Z / p^2\Z \to \Z / p\Z \to 0$.  Moreover, the convergence
  is strong.
\end{theorem}
\begin{proof}
  The proof McCleary gives on p.~459 carries over to our case intact.
  All that is required for this proof is that each $HS_n(A)$ be a
  finitely-generated abelian group.  The hypothesis that $A$ is
  finitely-generated, coupled with the result of
  Cor.~\ref{cor.fin-gen-restated} of
  Section~\ref{sub.connectivity_of_HS}, guarantees this.
\end{proof}

%%%%%%%%%%%%%%%%%%%%%%%%%%%%%%%%%%%%%%%%%%%%%%%%%%%%%%%%%%%%%%%%%%%%%%%%%%%%%%%%
\section{Symmetric Homology Using $\Delta S_+$}\label{sec.deltas_plus} 
%%%%%%%%%%%%%%%%%%%%%%%%%%%%%%%%%%%%%%%%%%%%%%%%%%%%%%%%%%%%%%%%%%%%%%%%%%%%%%%%

In this section, we shall show that replacing $\Delta S$ by $\Delta S_+$ in an
appropriate way does not affect the computation of $HS_*$.  We shall define a 
functor $B_*^{sym_+}A \co \Delta S_+ \to k$-$\mathbf{Mod}$ such that
$H_*(\Delta S_+,\, B_*^{sym_+}A) \cong H_*(\Delta S,\, B_*^{sym}A)$.

%%%%%%%%%%%%%%%%%%%%%%%%%%%%%%%%%%%%%%%%%%%%%%%%%%%%%%%%%%%%%%%%%%%%%%%%%%%%%%%%
\subsection{Extension of the Bar Construction to $\Delta S_+$}

Recall for a functor $\mathscr{F} \co \mathscr{D} \to \mathscr{C}$ of
small categories, the induced functor from the category of
$\mathscr{C}$--modules to $\mathscr{D}$--modules given by $M \mapsto
M\circ \mathscr{F}$ admits a left adjoint functor, $\mathscr{F}^*$
(see~\cite{K} or~\cite{GZ}, App.  II.3).  In particular, the inclusion
$i \co \Delta S \hookrightarrow \Delta S_+$ induces a functor $i^* \co
\Delta S$-$\mathbf{Mod} \longrightarrow \Delta S_+$-$\mathbf{Mod}$.
It is instructive to analyze the functor $i^*(B_*^{sym}A) \co \Delta
S_+ \to \textrm{$k$-$\mathbf{Mod}$}$.  Note, $i^*(B_*^{sym}A)$ is just
the left Kan extension of $B_*^{sym}A$ along $i$.  By definition,
$i^*(B_*^{sym}A)$ is the functor, $[n] \mapsto \colim_{(i/[n])}
B_*^{sym}A \circ \pi_1$, where $\pi_1 \co (i/[n]) \to \Delta S$ is the
projection $(i[m] \stackrel{\phi}{\to} [n]) \mapsto [m]$.  Note, for
$n = -1$, $\colim_{(i/[-1])} B_*^{sym}A \circ \pi_1 \cong 0$, the
initial object of $k$-$\mathbf{Mod}$, since the category $(i/[-1])$ is
empty.  For $n \geq 0$,
\begin{equation}\label{eqn.colim_F}
  \colim_{(i/[n])} B_*^{sym} \circ \pi_1 = \left(\bigoplus_{([n] 
  \stackrel{\phi}{\gets} [m]) \in 
  \mathrm{Mor}\Delta S} B_m^{sym}A \right) \bigg/ (\phi\psi, x) \approx (\phi, 
  \psi_*(x)) \quad \cong \quad B_n^{sym}A
\end{equation}
The last isomorphism results from the ``identity trick'' (see, for
example, Section~\ref{sub.symbar}), namely $(\phi, x) \approx
(\mathrm{id}, \phi_*(x))$.  For any morphism $\beta \in
\mathrm{Mor}\Delta S$, the morphism induced by $i^*(B_*^{sym}A)$ is
simply $\beta$ itself.  On the other hand, the morhpism $\iota_n \co
[-1] \to [n]$ of $\Delta S_+$ becomes the zero map $0 \to A^{\otimes
  (n+1)}$ under $i^*(B_*^{sym}A)$.

\begin{definition}\label{def.symbar_plus_0}
  For an associative, unital algebra, $A$, over a commutative ground ring $k$,  
  let $B_*^{sym^0_+}A$ be the functor $i^*(B_*^{sym}A) \co \Delta S_+ \to 
  k$-$\mathbf{Mod}$ defined above.
\end{definition}

\begin{lemma}\label{lem.HS_to_HS_+^0}
$\ds{ H_*(\Delta S_+,\, B_*^{sym_0^+}A) \cong H_*(\Delta S,\, B_*^{sym}A)}$.
\end{lemma}
\begin{proof}
  By~\cite{GZ}, Thm.~3.6, there is a spectral sequence with
  \begin{equation}\label{eqn.specseq_GZ}
    E_{p, q}^2 = {\colim_{\Delta S_+}}^{(p)} (L_qi^*)(B_*^{sym}A) \Rightarrow
    {\colim_{\Delta S}}^{(p+q)} B_*^{sym}A,
  \end{equation}
  where $i$ is the inclusion $\Delta S \hookrightarrow \Delta S_+$ as above and
  $L_qi^*$ is the $q^{th}$ left-derived functor of the right-exact functor 
  $i^*$.  The target, ${\colim_{\Delta S}}^{(p+q)} B_*^{sym}A$, is the same as 
  $HS_{p+q}(A)$.  I claim that $i^*$ is in fact left-exact.
  
  Let $0 \to M \to N \to P \to 0$ be an exact sequence of $\Delta S$--modules.
  By the same argument displayed in Eqn.~(\ref{eqn.colim_F}), we have
  \[
    \colim_{(i/[n])} M \circ \pi_1 \cong \left\{\begin{array}{cc}
      0, & \quad \textrm{if $n = -1$} \\
      M_n, & \quad \textrm{if $n \geq 0$}
    \end{array}\right.
  \]
  Similar results follow for $N$ and $P$.  Thus, the sequence $0 \to
  i^*M \to i^*N \to i^*P \to 0$ is exact, since each sequence of
  $k$--modules, $0 \to M_n \to N_n \to P_n \to 0$ is exact, and
  certainly the sequence of $0$ modules corresponding to $n = -1$ is
  exact.  Now, since $i^*$ is shown to be an exact functor, its
  left-derived functors $L_qi^*$ must be trivial for $q \geq 1$.  The
  spectral sequence~(\ref{eqn.specseq_GZ}) is trivial in columns $q
  \geq 1$, and thus strongly converges
  \[
    {\colim_{\Delta S_+}}^{(p)} i^*B_*^{sym}A 
    = H_p(\Delta S_+,\, B_*^{sym_+^0}A) \Rightarrow HS_p(A)
  \]
\end{proof}

%%%%%%%%%%%%%%%%%%%%%%%%%%%%%%%%%%%%%%%%%%%%%%%%%%%%%%%%%%%%%%%%%%%%%%%%%%%%%%%%
\subsection{A More Useful Extension}

Unfortunately, the $\Delta S_+$--module $B_*^{sym^0_+}A$ does not seem
to reflect the useful properities of $\Delta S_+$.  In particular,
the module corresponding to the object $[-1]$ is trivial in $B_*^{sym^0_+}A$.
Define a related functor which is non-trivial at the object $[-1]$:
\begin{definition}\label{def.symbar_plus}
  For an associative, unital algebra, $A$, over a commutative ground ring $k$,  
  define a functor $B_*^{sym_+}A \co \Delta S_+ \to k$--\textbf{Mod} by:
  \[
    \left\{
    \begin{array}{lll}
      B_n^{sym_+}A &=& A^{\otimes (n+1)}, \quad \textrm{for $n \geq 0$} \\
      B_{-1}^{sym_+}A &=& k,
    \end{array}
    \right.
  \]
  Morphisms are mapped in the same way as for $B_*^{sym_+^0}A$ with
  the exception that $\iota_n$ gets mapped to the inclusion $k
  \hookrightarrow A^{\otimes (n+1)}$ as algebras.
\end{definition}

One very useful consequence of defining symmetric in terms of the
homology groups $H_*(\mathscr{C},\, F)$ for an appropriate category
$\mathscr{C}$ and functor $F$, is that the functor $H_*(\mathscr{C},\,
-)$, takes short exact sequences of $\mathscr{C}$--modules to long
exact sequences in homology~\cite{GZ}.  We use this fact to prove:

\begin{theorem}\label{thm.SymHom_plusComplex}
  For an associative, unital $k$-algebra $A$, $HS_*(A) = H_*\left( 
  \Delta S_+,\, B_*^{sym_+}A\right)$.
\end{theorem}
\begin{proof}
  Consider the following sequence of $\Delta S_+$--modules:
  \begin{equation}\label{eqn.ses_functor_homology}
    0 \to B_*^{sym_+^0}A \to B_*^{sym_+}A \to Q \to 0
  \end{equation}
  Here, $Q$ is the quotient $B_*^{sym_+}A/B_*^{sym_+^0}A$, or in other words,
  $Q$ is the functor $\Delta S_+ \to k$--\textbf{Mod} taking $[-1] \mapsto
  k$ and $[n] \mapsto 0$ for all $n \geq 0$, which has no dependence
  on the choice of algebra $A$.  Indeed, with $A = k$, we have the
  exact sequence:
  \[
    0 \to B_*^{sym_+^0}k \to B_*^{sym_+}k \to Q \to 0
  \]
  The corresponding long exact sequence in homology is rather easy to
  analyze.  Using Lemma~\ref{lem.HS_to_HS_+^0},
  \[
    H_p\left( \Delta S_+, \, B_*^{sym_+^0}k\right) = HS_p(k)
    = \left\{ \begin{array}{ll}
      k, \quad p = 0\\
      0, \quad p > 0
    \end{array}\right.
  \]
  \[
    \cong H_*(N(\Delta S_+)) \cong H_*\left( \Delta S_+, \,
    B_*^{sym_+}k \right)
  \]
  The last isomorphism is obtained in the same way as in
  Cor.~\ref{cor.HS_of_k}.  Thus $H_*(\Delta S_+,\, Q) = 0$.  The
  $5$-lemma applied to the long exact sequence corresponding to
  diagram~(\ref{eqn.ses_functor_homology}) then proves $H_n(\Delta
  S_+,\,B_*^{sym_+^0}A) \cong H_n(\Delta S_+,\, B_*^{sym_+}A)$ in
  every degree. Lemma~\ref{lem.HS_to_HS_+^0} provides the link to
  symmetric homology.
\end{proof}

%%%%%%%%%%%%%%%%%%%%%%%%%%%%%%%%%%%%%%%%%%%%%%%%%%%%%%%%%%%%%%%%%%%%%%%%%%%%%%%%
\section{Tensor Algebra Decomposition and Fiedorowicz's Theorems}
\label{sec.tensoralgebras}
%%%%%%%%%%%%%%%%%%%%%%%%%%%%%%%%%%%%%%%%%%%%%%%%%%%%%%%%%%%%%%%%%%%%%%%%%%%%%%%%

For a general $k$-algebra $A$, finding $HS_*(A)$ using the standard
resolution is often too difficult.  The multiplicative structure of
$A$ is hard to control.  On the other hand, tensor algebras may be
easier to deal with since their multiplicative structure is so clean.
Indeed, tensor algebra arguments are also key in the proof of
Fiedorowicz's Theorem about the symmetric homology of group algebras.

%%%%%%%%%%%%%%%%%%%%%%%%%%%%%%%%%%%%%%%%%%%%%%%%%%%%%%%%%%%%%%%%%%%%%%%%%%%%%%%%
\subsection{Resolution of Algebras Using the 2-Sided Bar Construction}
\label{sub.resolution_barconstruction}

Let $T \co k$--\textbf{Mod} $\to k$--\textbf{Alg} be the functor sending a
$k$--module to the tensor algebra generated by that $k$--module.  In
other words, $TM = \bigoplus_{n \geq 0} M^{\otimes n}$.  There is an
algebra homomorphism $\theta_A \co TA \to A$, defined by multiplying
tensor factors, $\theta_A( a_0 \otimes a_1 \otimes \ldots \otimes a_k
) = a_0a_1 \cdots a_k$.  By abuse of notation, let $T$ also stand for
the functor $k$--\textbf{Mod} $\to k$--\textbf{Mod} obtained by
composing with the forgetful functor $k$--\textbf{Alg} $\to
k$--\textbf{Mod}.  $T$ becomes a monad with multiplication $\theta_{TA}
\co T^2 \to T$ and unit transformation $h \co \mathrm{Id} \to T$, defined
by sending the module $M$ identically onto the summand corresponding
to $n = 1$.  The homomorphism $\theta_A$ then expresses $A$ as a
$T$-algebra.

Observe that May's 2-sided bar construction $B_*(T, T, A)$ (See
chapter 9 of~\cite{M}) gives a resolution of $A$ by tensor algebras:
\begin{equation}\label{eq.res_tensor_alg}
  0 \gets A \stackrel{\theta_A}{\gets} TA \stackrel{\theta_1}{\gets} T^2A 
  \stackrel{\theta_2}{\gets} \ldots
\end{equation}
The boundary maps $\theta_n$ are defined in the standard way:
$\theta_n \stackrel{def}{=} \sum_{i = 0}^{n} (-1)^i
T^{n-i}\theta_{T^iA}$.  If we denote by $A_*$ the simplicial algebra
consisting of $A$ all degrees and whose faces and degeneracies are all
$\mathrm{id}_A$, then there is a strong deformation retract $B_*(T, T,
A) \to A_*$\cite{M}.

%%%%%%%%%%%%%%%%%%%%%%%%%%%%%%%%%%%%%%%%%%%%%%%%%%%%%%%%%%%%%%%%%%%%%%%%%%%%%%%%
\subsection{Monoid Algebras}

In the case that $A = k[M]$ for a monoid, $M$, we find a remarkable
interpretation of the complex $C_*( \Delta S_+, B_*^{sym_+}TA )$ as
an $E_{\infty}$-algebra.  First, define a set-valued variant of the
symmetric bar construction:
\begin{definition}
  Let $M$ be a monoid.  Define a functor $B_*^{sym}M \co \Delta S \to 
  \textbf{Sets}$ by:
  \[
    B_n^{sym}M = B_*^{sym}M[n] \stackrel{def}{=}  M^{n+1}, \;
    \textrm{(set product)}
  \]
  \[  
    B_*^{sym}M(\alpha) \co (m_0, \ldots, m_n) \mapsto \alpha(m_0, \ldots, m_n), 
    \qquad \textrm{for $\alpha \in \mathrm{Mor}\Delta S$}.
  \]
  where $\alpha \co [n] \to [k]$ is represented in tensor notation, and
  evaluation at $(m_0, \ldots, m_n)$ is as in
  definition~\ref{def.symbar}.

  The functor $B_*^{sym_+}M \co \Delta S_+ \to \textbf{Sets}$ will be
  defined by the above as well as by setting $B_{-1}^{sym_+}M
  \stackrel{def}{=} \emptyset$, and $B_*^{sym_+}M\iota_n$ is the
  inclusion $\emptyset \hookrightarrow M^{n+1}$.
\end{definition}

Regarding $N(- \setminus \Delta S_+)$ as a contravariant functor to
\textbf{sets}, define $N(- \setminus \Delta S_+) \times_{\Delta S_+}
B^{sym_+}_*M$ as a coend construction similar to that in
section~\ref{sec.hom_alg_functors}.  It is immediate that this
construction is a simplicial set whose homology computes $HS_*(k[M])$.

Now, in the context of monoids, the James construction is the
appropriate analog of the the tensor algebra construction of
section~\ref{sub.resolution_barconstruction}.  If $M$ is a free monoid
on a generating set $X$, that is, $M = JX_+$, then $k[M] = k[JX_+] =
T(k[X])$ is the free tensor algebra over $k$ on the set $X$ with
disjoint basepoint.  In this case, we have the following:
\begin{lemma}\label{lem.HS_tensoralg}
  \begin{equation}\label{eq.TX-decomp}
    HS_*\left(k[JX_+])\right) \cong H_*\left( \coprod_{n\geq -1}
    \widetilde{X}_n; k\right),
  \end{equation}
  where
  \[
    \widetilde{X}_n = \left\{\begin{array}{ll}
                               N(\Delta S_+), &n = -1\\
                               N\left([n] \setminus \Delta S_+\right) 
                               \times_{\Sigma_{n+1}^\mathrm{op}} X^{n+1},
                               &n \geq 0
                             \end{array}\right.
  \]
\end{lemma}
\begin{proof}
  Using the simplicial set $N(- \setminus \Delta S_+) \times_{\Delta S_+}
  B^{sym_+}_*JX_+$, a typical generator has the form,
  \begin{equation}\label{eqn.typ_gen}
    \left(
    \begin{diagram}
      \node[4]{[p]}
      \arrow{sw}
      \arrow{swww}
      \arrow{s,r}{\alpha}\\
      \node{[q_n]}
      \node{\cdots}
      \arrow{w,t}{\beta_n}
      \node{[q_1]}
      \arrow{w,t}{\beta_2}
      \node{[q_0]}
      \arrow{w,t}{\beta_1}
    \end{diagram}
    \right) \otimes u
  \end{equation}
  where
  \[
    u = \left(\bigotimes_{x\in X_0}x\right) \otimes
              \left(\bigotimes_{x\in X_1}x\right) \otimes \ldots \otimes
              \left(\bigotimes_{x\in X_p}x\right),
  \]
  and $X_0, X_1, \ldots, X_p$ are finite ordered lists of elements of
  $X$.  The idea is to ``expand'' the element $u$ into a tensor of
  individual elements of $X$, and at the same time to get rid of the
  units that may be present.  Indeed, each $X_j$ may be thought of as
  an element of the set product $X^{m_j}$ for some $m_j$.  If $X_j =
  \emptyset$, then set $m_j = 0$.  We use the convention that an empty
  tensor product is equal to $1_k$ (think of this as the disjoint
  basepoint appended to $X$), and say that the corresponding
  tensor factor is {\it trivial}.  Now, let $m = \left(\sum m_j
  \right) - 1$.  Let $\pi \co X^{m_0} \times X^{m_1} \times \ldots
  \times X^{m_p} \longrightarrow X^{m+1}$ be the evident isomorphism.
  Let $X_m = \pi( X_0, X_1, \ldots, X_p )$.
  
  {\bf Case 1.} If $u$ is non-trivial (\textit{i.e.}, $X_m \neq \emptyset$), 
  then construct the element
  \[
    u' = \bigotimes_{x \in X_m}x
  \]
  Next, construct a $\Delta$-morphism $\zeta_u \co [m] \to [p]$ as follows:  
  For each $j$, $\zeta_u$ maps the points
  \[
    \sum_{i=0}^{j-1} m_i, \left(\sum_{i=0}^{j-1} m_i\right)+1, \ldots,
    \left(\sum_{i=0}^{j} m_i\right) - 1 \;\mapsto\; j
  \]
  Observe, $(\zeta_u)_*(u') = u$. Under $\Delta S_+$-equivalence, 
  expression~\ref{eqn.typ_gen} is equivalent to
  \begin{equation}
    \left(
    \begin{diagram}
      \node[4]{[m]}
      \arrow{sw}
      \arrow{swww}
      \arrow{s,r}{\alpha\zeta_u}\\
      \node{[q_n]}
      \node{\cdots}
      \arrow{w,t}{\beta_n}
      \node{[q_1]}
      \arrow{w,t}{\beta_2}
      \node{[q_0]}
      \arrow{w,t}{\beta_1}
    \end{diagram}
    \right) \otimes u'
  \end{equation}
  The choice of $u'$ and $\zeta_u$ is well-defined with respect to the
  $\Delta S_+$-equivalence up to isomorphism of $[m]$ (an element of
  $\Sigma_{m+1}^{\mathrm{op}}$).  This shows that any such non-trivial
  chain may be written uniquely as an element of $\ds{ N([m] \setminus
    \Delta S_+) \times_{k\Sigma_{m+1}^\mathrm{op}} X^{m+1}}$.  An
  example should clarify this process.  If $u = (x_{i_0} \otimes
  x_{i_1}) \otimes 1 \otimes (x_{i_2} \otimes x_{i_3}) \in
  \left(T(k[X])\right)^{\otimes 3}$, then $u' = x_{i_0} \otimes
  x_{i_1} \otimes x_{i_2} \otimes x_{i_3}$ and $\zeta_u = x_0 x_1
  \otimes 1 \otimes x_2$ (written in tensor notation -- these $x_i$'s
  are formal variables, not elements of $X$).  Clearly, $\zeta_u(u') =
  u$.
  
  {\bf Case 2.} If $u$ is trivial, then $u = 1_k^{\otimes(p + 1)}$,
  and we have 
  \[
    ([q_n] \gets \ldots \gets [q_0] \stackrel{\alpha}{\gets} [p])
    \otimes u \approx ([q_n] \gets \ldots \gets [q_0]
    \stackrel{\iota_{q_0}}{\gets} [-1]) \otimes 1_k.
  \]
  This element can be identified uniquely with the chain 
  $[q_n] \gets \ldots \gets [q_0] \in N(\Delta S_+)$.
  
  Thus, the isomorphism~(\ref{eq.TX-decomp}) is proven.  Note that the
  total number of non-trivial tensor factors is preserved under
  $\Delta S$ morphisms.  This shows that the differential respects the
  direct sum decomposition.
\end{proof}

\begin{rmk}
  This proof works the same using $\Delta S$ rather than $\Delta S_+$,
  by replacing the argument of Case 2 with the ``identity trick'',
  \[
    ([q_n] \gets \ldots \gets [q_0] \stackrel{\alpha}{\gets} [p])
    \otimes u \approx ([q_n] \gets \ldots \gets [q_0]
    \stackrel{\mathrm{id}}{\gets} [q_0]) \otimes 1_k^{\otimes(q_0+1)}.
  \]
  However the extra structure of $\Delta S_+$ will be necessary to
  prove the $E_{\infty}$-structure needed for the results in the next
  section.
\end{rmk}

Breaking normal convention, let $X^0$ stand for the one-point set,
$\{ \ast \}$, (or $\{ 1_k \}$, which signifies the purpose of this
set).  Then the content of Lemma~\ref{lem.HS_tensoralg} can
be more concisely,

\begin{equation}\label{eq.JX-decomp}
  HS_*(k[JX_+]) \cong \bigoplus_{n \geq -1} H_*\left( N\left([n]
  \setminus \Delta S_+\right) \times_{\Sigma_{n+1}^\mathrm{op}}
  X^{n+1}; k\right)
\end{equation}

\begin{rmk}
  It will be useful to work with a resolution of the monoid $M$ by
  free monoids.  We obtain an equivalence, $B_*(J, J, M) \to M_*$,
  where $M_*$ is the simplicial monoid consisting of $M$ in all
  degrees and whose faces and degeneracies are all $\mathrm{id}_M$.
\end{rmk}

The next lemma will provide an essential link in the computation
of $HS_*(k[M])$.

\begin{lemma}\label{lem.E-infty-algebra}
  Let $\mathscr{N} \stackrel{def}{=} N(- \setminus \Delta S_+)
  \times_{\Delta S_+} B^{sym_+}_*J$.  There is an equivalence of
  functors $\Theta \co \mathscr{N} \stackrel{\simeq}{\to} C_{\infty}$,
  where $C_{\infty}$ is the monad associated to the
  $E_{\infty}$-operad, $\mathcal{C}_{\infty}$.  Moreover, $\Theta$
  induces an equivalence
  \begin{equation}\label{eqn.equiv_of_bars}
    \Theta_* \co B_*(\mathscr{N}, J, M) \to B_*(C_{\infty}, J, M)
  \end{equation}
\end{lemma}

\begin{rmk}
  $\mathscr{N}$ itself turns out not to be a full-fledged operad,
  since it fails the right-unit condition.  However, in order to prove
  Fiedorowicz's Theorems below, all that we require is an equivalence
  of functors which preserves the simplicial maps of the bar
  construction.
\end{rmk}

\begin{proof}
  By definition,
  \[
    \mathscr{N}X = \coprod_{n\geq 0} N\left([n-1] \setminus \Delta
    S_+\right) \times_{\Sigma_{n}^\mathrm{op}} X^{n},
  \]
  which has been re-indexed to begin at $n = 0$.  Let $\mathscr{D}$
  denote the operad with $\mathscr{D}(n) = E\Sigma_n$.  Since
  $\mathscr{D}$ is an $E_{\infty}$-operad, and the monads associated
  to different $E_{\infty}$-operads are equivalent, it suffices to
  show the equivalence, $\mathscr{N}X \simeq \mathscr{D}X$.  

  Note, the $0$-chains (objects) of $N\left([n-1] \setminus \Delta
  S_+\right)$ are pairs $(\phi, g) \in \mathrm{Mor}_{\Delta_+}([n-1],
  [m]) \times \Sigma_{n}^{\mathrm{op}}$, so the $\Sigma_n$-action is
  free.  Since $N\left([n-1] \setminus \Delta S_+\right)$ is a
  contractible simplicial set, there is an equivalence induced on
  $0$-chains by the map sending $(\phi, g) \mapsto g$:
  \[
    B\left( [n-1] \setminus \Delta S_+\right) \stackrel{\simeq}{\to}
    E\Sigma_n^{\mathrm{op}},
  \]
  where $B\mathscr{C}$ is the classifying space of a category
  $\mathscr{C}$.  The map $g \to g^{-1}$ provides an equivalence
  $E\Sigma_n^{\mathrm{op}} \simeq E\Sigma_n$.

  These equivalences induce
  \begin{equation}\label{eqn.equiv_to_show}
    \coprod_{n\geq 0} N\left([n-1] \setminus \Delta
    S_+\right) \times_{\Sigma_{n}^{\mathrm{op}}} X^{n}
    \simeq \coprod_{n \geq 0} E_*\Sigma_n^{\mathrm{op}}
    \times_{\Sigma_n^{\mathrm{op}}} X^n
    \simeq
    \coprod_{n \geq 0} E_*\Sigma_n
    \times_{\Sigma_n} X^n
  \end{equation}
  Eqn.~(\ref{eqn.equiv_to_show}) gives the equivalence of functors,
  $\mathscr{N}\to C_{\infty}$.  For the second part of the lemma,
  we need only consider the $0$-faces of each bar construction
  in Eqn.~(\ref{eqn.equiv_of_bars}), since
  the other face maps and all of the degeneracies only depend on the
  monad $J$ and the monoid $M$.  The $0$-faces correspond to
  the transformations $\mathscr{N}J \to \mathscr{N}$ on the one hand,
  and $\mathscr{D}J \to \mathscr{D}$ on the other.  Since the map
  $\mathscr{N}X \to \mathscr{D}X$ is natural in $X$, the diagram
  below commutes:  
  \[
    \begin{diagram}
    \node{\mathscr{N}JX}
    \arrow{e,t}{\simeq}
    \arrow{s}
    \node{\mathscr{D}JX}
    \arrow{s}
    \\
    \node{\mathscr{N}X}
    \arrow{e,t}{\simeq}
    \node{\mathscr{D}X}
    \end{diagram}
  \]
  Thus, the equivalence passes to the bar construction with no problem.
\end{proof}

%%%%%%%%%%%%%%%%%%%%%%%%%%%%%%%%%%%%%%%%%%%%%%%%%%%%%%%%%%%%%%%%%%%%%%%%%%%%%%%%
\subsection{Fiedorowicz's Theorems}\label{sec.symhommonoid}  

The results of this section have been known for some time now, but not
previously published.  
Fiedorowicz studied
the symmetric homology of monoid algebras and group algebras~\cite{F}.  
All proofs of this section are either based on or inspired by the
corresponding proofs given in that preprint. The
following theorem is the first major result, which serves as a
stepping stone to the group-algebra case.
\begin{theorem}\label{thm.HS_monoidalgebra}
  $HS_*(k[M]) \cong H_*\left(B(C_{\infty}, C_1, M); k\right)$, where
  $C_1$ is the little $1$-cubes monad and $C_{\infty}$ is the little
  $\infty$-cubes monad.
\end{theorem}
\begin{proof}
  The arguments become clearer if we use homotopy colimits throughout.
  The link is provided by the following fact of homological algebra:
  \[
    {\colim_{\mathscr{C}}}^{(*)}F = H_*\left( \hocolim_{\mathscr{C}}F \right)
  \]
  In our case, we are interested in $\mathscr{C}= \Delta S_+$ and $F =
  B_*^{sym_+}M$.  By Section~\ref{sub.resolution_barconstruction}, we
  can replace $M$ by its resolution $B_*(J, J, M)$.  
  So there is an equivalence,
  \[
    \hocolim_{\Delta S_+} B_*^{sym_+}M \simeq \hocolim_{\Delta S_+} 
    B_*^{sym_+}B_*(J, J, M)
  \]
  Now using Lemma 9.7 of~\cite{M},
  \[
    \hocolim_{\Delta S_+} B_*^{sym_+}B_*(J, J, M) \cong
    B_*\left( \hocolim_{\Delta S_+} B_*^{sym_+}J, J, M\right)
    = B_*\left( \mathscr{N}, J, M \right)
  \]
  By Lemma~\ref{lem.E-infty-algebra} the latter is equivalent to $B_*(
  C_{\infty}, J, M)$.  The proof is completed by using the equivalence
  $J \simeq C_1$.
\end{proof}

As we turn our attention to group algebras, the following lemma
will be necessary.

\begin{lemma}\label{lem.top_monoid}
  For any topological monoid, $M$, there is a natural equivalence,
  \begin{equation}\label{eqn.equivalence}
    B(\Omega^{\infty}S^{\infty}, C_1, M) \simeq \Omega\Omega^{\infty}
    S^{\infty}(BM)
  \end{equation}
\end{lemma}
\begin{proof}
  Consider the bar construction $B_*(S, C_1, M)$, where $S$ is 
  suspension.  May's Lemma 9.7~\cite{M} gives the isomorphism
  \[
    B_*(\Omega^{\infty}S^{\infty}S, C_1, M) \cong
    \Omega^{\infty}S^{\infty}B_*(S, C_1, M).
  \]
  Looping on both sides and taking geometric realization, we get
  \[
    B(\Omega\Omega^{\infty}S^{\infty}S, C_1, M)
    \cong
    \Omega\Omega^{\infty}S^{\infty}B(S, C_1, M).
  \]
  Now $B(\Omega\Omega^{\infty}S^{\infty}S, C_1, M) \simeq
  B(\Omega^{\infty}S^{\infty}, C_1, M)$.  Finally, by
  Thomason~\cite{T2} and Fiedorowicz~\cite{F2}, $B(S, C_1, M)$ is
  naturally equivalent to the bar construction $BM$, which yields the
  equivalence~(\ref{eqn.equivalence}).
\end{proof}

When the monoid is a group, we obtain the following important
consequence.
\begin{cor}\label{thm.HS_group-restated}
  If $\Gamma$ is a group, then
  $\ds{HS_*(k[\Gamma]) \cong H_*\left(\Omega\Omega^{\infty}S^{\infty}(B\Gamma); 
  k\right)}$.
\end{cor}
\begin{proof}
  From Thm.~\ref{thm.HS_monoidalgebra}, $HS_*(k[\Gamma]) \cong
  H_*\left(B(C_{\infty}, C_1, \Gamma); k\right)$.  Since $\Gamma$ is a
  group, $B(C_{\infty}, C_1, \Gamma)$ is grouplike, so $B(C_{\infty},
  C_1, \Gamma) \stackrel{\simeq}{\to} B(\Omega^{\infty}S^{\infty},
  C_1, \Gamma)$.  Lemma~\ref{lem.top_monoid} then applies.
\end{proof}

The analysis of $HS_*(k[\Gamma])$ then becomes interesting in its
own right.
We shall have more specific to say about $HS_1(k[\Gamma])$ 
in section~\ref{sub.2-torsion}.

%%%%%%%%%%%%%%%%%%%%%%%%%%%%%%%%%%%%%%%%%%%%%%%%%%%%%%%%%%%%%%%%%%%%%%%%%%%%%%%%
\section{The Category $\mathrm{Epi}\Delta S$ and a Smaller Resolution}
%%%%%%%%%%%%%%%%%%%%%%%%%%%%%%%%%%%%%%%%%%%%%%%%%%%%%%%%%%%%%%%%%%%%%%%%%%%%%%%%
\label{sec.epideltas}
The complexes for computing $H_*(\Delta S,\,B_*^{sym}A)$ or
$H_*(\Delta S_+,\,B_*^{sym_+}A)$ are extremely large and unwieldy for
computation.  Fortunately, when the algebra $A$ is equipped with an
augmentation, $\epsilon \co A \to k$, we may use a more manageable
subcomplex of $C_*(\Delta S_+,\,B_*^{sym_+}A)$, related to the
subcategory of $\Delta S_+$ consisting only of epimorphisms.

%%%%%%%%%%%%%%%%%%%%%%%%%%%%%%%%%%%%%%%%%%%%%%%%%%%%%%%%%%%%%%%%%%%%%%%%%%%%%%
\subsection{Basic and Reduced Tensors}

Recall, if $A$ has an augmentation $\epsilon$, then there is an
augmentation ideal $I$ and the exact sequence $0 \to I \to A
\stackrel{\epsilon}{\to} k \to 0$ splits as $k$--modules.  So $A \cong
I \oplus k$, and every $x \in A$ can be decomposed uniquely as $x = a
+ \lambda$ for some $a \in I$, $\lambda \in k$.
\begin{definition}\label{def.B_JA}
  Define $B_{-1,\emptyset}A = k$.  For $n \geq 0$, if $J \subseteq [n]$, define
  \[
    B_{n,J}A \stackrel{def}{=} B^J_0 \otimes B^J_1 \otimes \ldots
    \otimes B^J_n, \quad \textrm{where}\;\; B^J_j =
    \left\{\begin{array}{ll} I & \textrm{if $j \in J$}\\ k &
    \textrm{if $j \notin J$}
                 \end{array}\right.
  \]  
\end{definition}

\begin{rmk} For each $n \geq -1$, there is a direct sum decomposition 
  of $k$--modules, $\ds{B_n^{sym_+}A \cong \bigoplus_{J \subseteq [n]}
    B_{n,J}A}$.
\end{rmk}

\begin{definition}\label{def.basictensors}
  A {\it basic tensor} is any tensor $w_0 \otimes w_1 \otimes \ldots
  \otimes w_n$, where each $w_j$ is in $I$ or is equal to the unit of
  $A$.  Call a tensor factor $w_j$ \textit{trivial} if it is the unit
  of $A$.  If all factors of a basic tensor are trivial, then the
  tensor is called {\it trivial}, and if no factors are trivial, the
  tensor is called {\it reduced}.
\end{definition}

For a basic
tensor $Y \in B_n^{sym_+}A$, we shall define a map $\delta_Y \in
\mathrm{Mor}\Delta_+$ as follows: If $Y$ is trivial, let $\delta_Y =
\iota_n$.  Otherwise, $Y$ has $\overline{n} + 1$ non-trivial factors
for some $\overline{n} \geq 0$.  Define $\delta_Y \co [\overline{n}] \to
[n]$ to be the unique injective map that sends each point $0, 1,
\ldots, \overline{n}$ to a point $p \in [n]$ such that $Y$ is
non-trivial at factor~$p$.  Let $\overline{Y}$ be the tensor obtained
from $Y$ by omitting all trivial factors if such exist, or
$\overline{Y} = 1_k$ if $Y$ is trivial.  Note, $\overline{Y}$ is the
unique basic tensor such that $(\delta_Y)_*(\overline{Y}) = Y$.

\begin{prop}\label{prop.BsymI}
  Any chain $([q_n] \gets [q_{n-1}] \gets \cdots \gets [q_0] \gets
  [q]) \otimes Y \;\in\; k[N(-\setminus \Delta S_+)] \otimes_{\Delta
    S_+} B_*^{sym_+}A$, where $Y$ is a basic tensor, is equivalent to
  a chain $([q_n] \gets [q_{n-1}] \gets \cdots \gets [q_0] \gets
  [\overline{q}]) \otimes \overline{Y}$, where either $\overline{Y}$
  is reduced or $\overline{Y} = 1_k$ and $\overline{q} = -1$.
\end{prop}
\begin{proof}
  Let $\delta_Y$ and $\overline{Y}$ be defined as above, and let
  $[\overline{q}]$ be the domain of $\delta_Y$.  Then $Y =
  (\delta_Y)_*(\overline{Y})$, and
  \begin{eqnarray*}
    \lefteqn{([q_n] \gets [q_{n-1}] \gets \cdots \gets [q_0] 
      \stackrel{\phi}{\gets} [q])\otimes Y}\\
    &\approx&
    ([q_n] \gets [q_{n-1}] \gets \cdots \gets [q_0] \stackrel{\phi\delta_Y}
    {\longleftarrow} [\overline{q}]) \otimes \overline{Y}
  \end{eqnarray*}
\end{proof}
  
%%%%%%%%%%%%%%%%%%%%%%%%%%%%%%%%%%%%%%%%%%%%%%%%%%%%%%%%%%%%%%%%%%%%%%%%%%%%%%
\subsection{Reducing to Epimorphisms}
We now turn our attention to the morphisms in the chains.  Our goal is
to reduce to those chains that involve only epimorphisms.

\begin{definition}
  Let $\mathscr{C}$ be a category.  The category $\mathrm{Epi}\mathscr{C}$ 
  (resp.~$\mathrm{Mono}\mathscr{C}$) is the subcategory of $\mathscr{C}$ 
  consisting of the same objects as $\mathscr{C}$ and only those morphisms 
  $f \in \mathrm{Mor}\mathscr{C}$ that are epic (resp.~monic).
  The set of morphisms of $\mathrm{Epi}\mathscr{C}$ from $X$ to $Y$ may be 
  denoted $\mathrm{Epi}_{\mathscr{C}}(X, Y)$.  Similarly, the set of morphisms 
  of $\mathrm{Mono}\mathscr{C}$ from $X$ to $Y$ may be denoted 
  $\mathrm{Mono}_{\mathscr{C}}(X, Y)$.  
\end{definition}
\begin{rmk}
  A morphism $\alpha = (\phi, g) \in \mathrm{Mor}\Delta S_+$ is epic
  (resp.~monic) if and only if $\phi$ is epic (resp.~monic) as
  morphism in $\Delta_+$.
\end{rmk}

\begin{prop}\label{prop.decomp}
  Any morphism $\alpha \in \mathrm{Mor}\Delta S_+$ decomposes uniquely as 
  $(\eta, \mathrm{id}) \circ \gamma$, where $\gamma \in \mathrm{Mor}(
  \mathrm{Epi}\Delta S_+)$ and $\eta \in \mathrm{Mor}(\mathrm{Mono}\Delta_+)$.
\end{prop}
\begin{proof}
  Suppose $\alpha$ has source $[-1]$ and target $[n]$.  Then $\alpha =
  \iota_n$ is the only possibility, and this decomposes as $\iota_n
  \circ \mathrm{id}_{[-1]}$.  Now suppose the source of $\alpha$ is
  $[p]$ for some $p \geq 0$.  Write $\alpha = (\phi, g)$, with $\phi
  \in \mathrm{Mor}_\Delta([p], [r])$ and $g \in
  \Sigma_{p+1}^{\mathrm{op}}$.  Suppose $\phi$ hits $q+1$ distinct
  points in $[r]$.  Then there is a morphism $\pi \co [p] \to [q]$
  induced by $\phi$ by maintaining the order of the points hit.  Let
  $\eta$ be the obvious order-preserving monomorphism $[q] \to [r]$ so
  that $\eta \pi = \phi$ as morphisms in $\Delta$.  To get the
  required decomposition in $\Delta S$, use: $\alpha = (\eta,
  \mathrm{id}) \circ (\pi, g)$.
  
  Now, if $(\xi, \mathrm{id})\circ (\psi, h)$ is also a decomposition
  of $\alpha$, with $\xi$ monic and $\psi$ epic, then $(\xi,
  \mathrm{id}) \circ (\psi, h) = (\eta, \mathrm{id}) \circ (\pi, g)$
  implies $(\xi \psi, g^{-1}h) = (\eta\pi, \mathrm{id})$, proving $g =
  h$.  Uniqueness will then follow from uniqueness of such
  decompositions entirely within the category $\Delta$.  The latter
  follows from Theorem~B.2 of~\cite{L}.
\end{proof}
This decomposition will be written:  $[p] \to [r] \;=\; [p] \twoheadrightarrow
\mathrm{im}([p] \to [r]) \hookrightarrow [r]$, and we call the
{\it epimorphism construction}, the rule $\mathscr{E}_p( [p] \to [r] )
= [p] \twoheadrightarrow \mathrm{im}([p] \to [r])$.
\begin{prop}\label{prop.epiconstruction-functorial}
  The epimorphism construction is a functor $\mathscr{E}_p \co [p] \setminus 
  \Delta S_+ \to [p] \setminus \mathrm{Epi}\Delta S_+$.
\end{prop}
\begin{proof}
  If $[p] \stackrel{\beta}{\to} [r_1] \stackrel{\alpha}{\to} [r_2]$,
  then there is an induced map $\mathrm{im}([p] \to [r_1])
  \stackrel{\overline{\alpha}}{\to} \mathrm{im}([p] \to [r_2])$ making
  the diagram commute:
  \begin{equation}\label{eq.epidiagram}
    \begin{diagram}%20
    \node{ [r_1] }
    \arrow[2]{e,t}{ \alpha }
    \node[2]{ [r_2] }
    \\
    \node[2]{ [p] }
    \arrow{nw,b}{ \beta }
    \arrow{ne,b}{ \alpha\beta }
    \arrow{sw,t,A}{ \pi_1 }
    \arrow{se,t,A}{ \pi_2 }
    \\
    \node{ \mathrm{im}([p] \to [r_1]) }
    \arrow[2]{n,l,J}{ \eta_1 }
    \arrow[2]{e,b,A}{ \overline{\alpha} }  
    \node[2]{ \mathrm{im}([p] \to [r_2]) }
    \arrow[2]{n,r,J}{ \eta_2 }
    \end{diagram}
  \end{equation}
  $\overline{\alpha}$ is the epimorphism induced from the map $\alpha \eta_1$.  
  Furthermore, for morphisms $[p] \to [r_1] \stackrel{\alpha_1}{\to} [r_2] 
  \stackrel{\alpha_2}{\to} [r_3]$, we have:  $\overline{\alpha_2 \alpha_1} 
  = \overline{\alpha_2} \circ \overline{\alpha_1}.$
\end{proof}

\begin{rmk}
  If $\alpha \co [p] \to [r]$ is an epimorphism of $\Delta S_+$, then 
  $\mathscr{E}_p(\alpha) = \alpha$.
\end{rmk}

Define a variant of the symmetric bar construction using $\mathrm{Epi}
\Delta S$:
\begin{definition}
  $B_*^{sym_+}I \co \mathrm{Epi}\Delta S_+  \to k$--\textbf{Mod} is the functor 
  defined by:
  \[
    \left\{
    \begin{array}{lll}
      B_n^{sym_+}I &=& I^{\otimes (n+1)}, \quad n \geq 0, \\
      B_{-1}^{sym_+}I &=& k,
    \end{array}
    \right.
  \]
  \[  
    B_*^{sym_+}I(\alpha) \co (a_0 \otimes a_1 \otimes \ldots \otimes a_n) \mapsto
      \alpha(a_0, \ldots, a_n), \;\textrm{for $\alpha \in 
      \mathrm{Mor}(\mathrm{Epi}\Delta S_+$)}
  \]
\end{definition}
This definition makes sense since $I$ is an ideal, and $\alpha$ is required to
be epimorphic.  Note, the simple tensors $w_0 \otimes \ldots \otimes w_n$ in
$B_n^{sym_+}I$ are by definition reduced.  Consider the simplicial $k$--module:
\begin{equation}\label{epiDeltaS_complex}
  C_*(\mathrm{Epi}\Delta S_+,\,B_*^{sym_+}I) 
  = k[ N(- \setminus \mathrm{Epi}\Delta S_+) ]
  \otimes_{\mathrm{Epi}\Delta S_+} B_*^{sym_+}I
\end{equation}
There is an obvious inclusion, $f \co C_*(\mathrm{Epi}\Delta
S_+,\,B_*^{sym_+}I) \longrightarrow C_*(\Delta S_+,\,B_*^{sym_+}A)$.
Define a chain map $g$ in the opposite direction as follows.  First,
by Prop.~\ref{prop.BsymI} and observations above, we only need to
define $g$ on the chains $([q_n] \gets \cdots \gets [q_0] \gets [q])
\otimes Y$ where $Y$ is reduced (or $Y = 1_k$).  In this case, define
component maps $g(q) = N(\mathscr{E}_q) \otimes \mathrm{id}$.  A
priori, this definition is well-defined only when the tensor product
is over $k$.  We would like to assemble the maps $g(q)$ into a chain
map $g$.  In order to do this, we must show that the maps are
compatible under $\Delta S_+$-equivalence.

Suppose $v = ([p_n] \gets \ldots \gets [p_0] \stackrel{\phi\psi}{\gets} [p])
\otimes Z$ and $v' = ([p_n] \gets \ldots \gets [p_0] \stackrel{\phi}{\gets}
[q]) \otimes \psi_*(Z)$, 
where $\psi \in \mathrm{Mor}_{\Delta S_+}\left([p], 
[q]\right)$.  If $Z$ is a basic tensor, then so is $\psi_*(Z)$.  In order to 
apply $g$ to $v$ or $v'$, each must first be put into a reduced form.

{\bf Case 1} Suppose $Z$ is trivial.  Then $v$ and $v'$ both reduce to
$([p_n] \gets \ldots \gets [p_0] \gets [-1]) \otimes 1$,
hence $g(v) = g(v')$.

{\bf Case 2} Suppose $Z$ is non-trivial.  For the sake of clean notation, let
$W = \psi_*(Z)$.  Construct $\delta_Z$, $\overline{Z}$, $\delta_W$ and
$\overline{W}$ such that $Z = (\delta_Z)_*(\overline{Z})$ and $W = 
(\delta_W)_*(\overline{W})$ as in Prop.~\ref{prop.BsymI}, and reduce both 
chains:
\begin{equation}\label{eq.epi-reduction}
  \begin{diagram}%21
    \node{ ([p_n] \gets \ldots \gets [p_0] \stackrel{\phi\psi}{\longleftarrow}
      [p]) \otimes Z}
    \arrow{e,b,!}{\approx}
    \arrow{s,lr}{reduce}{\approx}
    \node{ ([p_n] \gets \ldots \gets [p_0] \stackrel{\phi}{\longleftarrow}
      [q]) \otimes W}
    \arrow{s,lr}{reduce}{\approx}
    \\
    \node{([p_n] \gets \ldots \gets [p_0] \stackrel{\phi\psi\delta_Z}
      {\longleftarrow} [\overline{p}]) \otimes \overline{Z}}
    \node{([p_n] \gets \ldots \gets [p_0] \stackrel{\phi\delta_W}
      {\longleftarrow} [\overline{q}]) \otimes \overline{W}}
  \end{diagram}
\end{equation}
Observe that number of distinct points hit by $\psi\delta_Z$ is exactly 
$\overline{q} + 1$; indeed, $W=\psi_*(Z)$ has $\overline{q} + 1$ non-trivial 
factors.  Thus, $[\overline{q}] = \mathrm{im}([\overline{p}] \to [q])$.
Now, Prop.~\ref{prop.decomp} implies that there is precisely one 
$\Delta S$-epimorphism $\gamma \co [\overline{p}] \to [\overline{q}]$ making 
Diagram~(\ref{eq.epi-decomp}) commute.
\begin{equation}\label{eq.epi-decomp}
  \begin{diagram}%22
    \node{ [p]  }
    \arrow{e,t}{\psi}
    \node{ [q] }
    \\    
    \node{ [\overline{p}] }
    \arrow{n,l}{\delta_Z}
    \arrow{e,t,.}{}
    \arrow{see,b,A}{\gamma}
    \arrow{ne,t}{\psi\delta_Z}
    \node{ [\overline{q}] }
    \arrow{n,l}{\delta_W}
    \arrow{se,t,=}{}
    \\
    \node[3]{ \mathrm{im}([\overline{p}] \to [q]) }
    \arrow{nnw,t,L}{}    
  \end{diagram}
\end{equation}
That is to say, there exists an epimorphism $\gamma$ such that
$\gamma_*(\overline{Z}) = \overline{W}$ and $\psi\delta_Z =
\delta_W\gamma$.  So we may replace the first morphism of the chain in
the lower left of Diagram~(\ref{eq.epi-reduction}) with
$\phi\delta_W\gamma$.  Then when we apply $g$ to the chain, the first
morphism becomes $\mathscr{E}_{\overline{p}} ( \phi\delta_W\gamma ) =
\mathscr{E}_{\overline{p}}( \phi\delta_W) \circ \gamma$, since
$\gamma$ is epic.  Let $\pi \stackrel{def}{=} \mathscr{E}_{\overline{p}} (
\phi\delta_W)$.  Then the result of applying $g$ to each side
Diagram~(\ref{eq.epi-reduction}) is shown below:
\[
  \begin{diagram}%23
    \node{ ([p_n] \gets \ldots \gets [p_0] \stackrel{\phi\psi\delta_Z}
      {\longleftarrow} [\overline{p}]) \otimes \overline{Z} }
    \arrow{s,l}{ g(\overline{p}) }
    \\
    \node{ (\mathrm{im}([\overline{p}] \to [p_n]) \gets \ldots \gets
      \mathrm{im}([\overline{p}] \to [p_0]) \stackrel{\pi\gamma}{\gets}
      [\overline{p}]) \otimes \overline{Z} }
  \end{diagram}
\]
\[
  \begin{diagram}
    \node{ ([p_n] \gets \ldots \gets [p_0] \stackrel{\phi\delta_W}
      {\longleftarrow}  [\overline{q}])\otimes \overline{W} }
    \arrow{s,r}{ g(\overline{q}) }
    \\
    \node{ (\mathrm{im}([\overline{q}] \to [p_n])  \gets \ldots \gets
      \mathrm{im}([\overline{q}] \to [p_0])  \stackrel{\pi}{\gets} 
      [\overline{q}]) \otimes \overline{W} } 
  \end{diagram}
\]
Observe that there is equality of objects and morphisms up to the
morphisms $\pi\gamma$ on the left and $\pi$ on the right.  Since
$\gamma$ is epic, the $\mathrm{Epi}\Delta S_+$-equivalence allows us
to transport the morphism $\gamma$ to the right of the tensor, showing
that $g(v) \approx g(v')$, hence $g$ is well-defined.

\begin{prop}\label{prop.epi}
  If $A$ has augmentation ideal $I$, then
  \[
    HS_*(A) = H_*(\mathrm{Epi}\Delta S_+,\,B_*^{sym_+}I) 
    = H_*\left(k[ N(- \setminus \mathrm{Epi}\Delta S_+) ] 
    \otimes_{\mathrm{Epi}\Delta S_+} B_*^{sym_+}I;\,k \right).
  \]
\end{prop}
\begin{proof}
  Clearly $gf = \mathrm{id}$, where $f$ is the inclusion
  mentioned above.  It remains to show
  that $fg \simeq \mathrm{id}$.  Assume $Y$ is a basic tensor in
  $B_q^{sym_+}I$.  Define maps $h_j^{(n)}$ as follows:
  \[
    h_j^{(n)} \co \left([q_n] \gets \ldots \gets [q_0] \gets [q]\right)
    \otimes Y \mapsto
  \]
  \[
    \left([q_n] \gets \dots \gets [q_j] \gets \mathrm{im}([q] \to [q_j])
    \gets \ldots \gets \mathrm{im}([q] \to [q_0]) \gets [q]\right)
    \otimes Y.
  \]
  $h_j^{(n)}$ is well-defined by the functorial properties of the
  epimorphism construction, and a routine, but tedious, verification
  shows that $h$ defines a presimplicial homotopy from $fg$ to
  $\mathrm{id}$.
\end{proof}

%%%%%%%%%%%%%%%%%%%%%%%%%%%%%%%%%%%%%%%%%%%%%%%%%%%%%%%%%%%%%%%%%%%%%%%%%%%%%%%%
\subsection{Reduced Symmetric Homology}
Denote by $B_*^{sym}I$, the restriction of $B_*^{sym_+}I$ to $\Delta
S$.  That is, $B_n^{sym}I \stackrel{def}{=} I^{\otimes(n+1)}$ for all $n \geq 0$.
Then there is a splitting of the chain complex,
\[
  C_*( \mathrm{Epi}\Delta S_+,\,B_n^{sym_+}I) =
  C_*( \mathrm{Epi}\Delta S,\,B_n^{sym}I) \oplus k[ N(\ast) ],
\]
where $\ast$ is the full subcategory of $\mathrm{Epi}{\Delta S_+}$
consisting only of the object $[-1]$.  Indeed, since there are no
epimorphisms $[-1] \to [n]$ or $[n] \to [-1]$ for $n \geq 0$, we may
think of $[-1]$ as a disconnected basepoint.  Hence, we have $HS_*(A)
\cong H_*(\mathrm{Epi}\Delta S,\,B_n^{sym}I) \oplus k_0$, where $k_0$
is the graded $k$--module consisting of $k$ concentrated in degree $0$.
\begin{definition}\label{def.reducedHS}
  The {\it reduced symmetric homology} of $A$ is defined, 
  \[
    \widetilde{H}S_*(A) \stackrel{def}{=} H_*(\mathrm{Epi}\Delta
    S,\,B_n^{sym}I).
  \]
\end{definition}

\begin{rmk}
  Reduction to epimorphisms seems to hinge on the property that $A$
  has an augmentation ideal.  This condition may be lifted (as Richter
  conjectures), if it can be shown that $N(\mathrm{Epi}\Delta S)$ is
  contractible.  As partial progress along these lines, it can be
  shown that $N(\mathrm{Epi}\Delta S)$ is simply-connected.
\end{rmk}

%%%%%%%%%%%%%%%%%%%%%%%%%%%%%%%%%%%%%%%%%%%%%%%%%%%%%%%%%%%%%%%%%%%%%%%%%%%%%%%%
\section{Interpretation of $HS_*(A)$ in terms of Group Homology}
\label{sec.specseq}
%%%%%%%%%%%%%%%%%%%%%%%%%%%%%%%%%%%%%%%%%%%%%%%%%%%%%%%%%%%%%%%%%%%%%%%%%%%%%%%%

$\mathrm{Epi}\Delta S$ is an E-I-A-category, that is, a category in
which all endomorphisms are isomorphisms and also each isomorphism
class of an object has only a single member (all isomorphism are
automorphisms).  In this section, we use results of
S\l{}omi\'nska~\cite{Sl} about E-I-A-categories to provide an
interpretation of symmetric homology as the homology of products of
symmetric groups with coefficients in certain modules.  As a
corollary, we find that when the ground ring is a field of
characteristic $0$, $HS_*(A)$ can be computed as the coinvariants of a
group action.

Fix a unital associative algebra $A$ over commutative ground ring $k$,
and assume $A$ has an augmentation, with augmentation ideal $I$.  By
Prop.~\ref{prop.epi} and definition~\ref{def.reducedHS},
\[
  \widetilde{H}S_*(A) =  H_*(\mathrm{Epi}\Delta S,\,B_*^{sym}I) 
  = {\colim_{\mathrm{Epi}\Delta S}}^{(*)} B_*^{sym}I,
\]
where ${\colim}^{(*)}$ represents the left derived functors of the
colimit.  

Let $\mathcal{S}_0$ be the category whose objects are
ordered tuples of non-negative integers, $(q_n < q_{n-1} < \cdots < q_0)$,
with a unique morphism $(q_n < \cdots < q_0) \to (q'_{n'} < \cdots < q'_0)$
if $\{ q'_{n'}, \ldots, q'_0\} \subseteq \{q_n, \ldots, q_0\}$.
Let $\mathcal{A}ut$ be the functor ($\mathcal{S}_0$-group) defined by
\[
  \mathcal{A}ut(q_n < \cdots < q_0) = \prod_{i=0}^{n}
  \Sigma^{\mathrm{op}}_{q_i + 1}
  = \prod_{i=0}^{n} \mathrm{Aut}_{\mathrm{Epi \Delta S}}([q_i])
\]
Morphisms of $\mathcal{S}_0$ map under $\mathcal{A}ut$ to the evident
projections.  Let $\mathcal{E}pi$ be the functor $\mathcal{S}_0 \to
\mathbf{Set}$ defined by
\[
  \mathcal{E}pi(q_n < \cdots < q_0) = \left\{\begin{array}{cc}
    \{ q_0 \}, &\quad \textrm{if $n = 0$}\\
    \prod_{i=1}^n \mathrm{Epi}_{\Delta S}\left( [q_{i-1}], [q_i]\right),
    &\quad \textrm{if $n \geq 1$}
  \end{array}\right.
\]
If $\{ q'_{n'}, \ldots, q'_0\} \subseteq \{q_n, \ldots, q_0\}$, then
the set map $\mathcal{E}pi(q_n < \cdots < q_0) \to \mathcal{E}pi
(q'_{n'} < \cdots < q'_0)$ is defined by the appropriate projections
and compositions of morphisms, if $n' \geq 1$, and by the unique set
map into $\{q'_0\}$ if $n = 0$.
 
There is an action of $\mathcal{A}ut(q_n < \cdots < q_0)$ on
$\mathcal{E}pi(q_n < \cdots < q_0)$ given by
\[
  \mu \co (g_n, \ldots, g_0)\otimes(\phi_n, \ldots, \phi_1) \mapsto
  (g_n \phi_n g_{n-1}^{-1}, \ldots, g_1\phi_1 g_0^{-1})
\]
The maps induced by $(q_n < \cdots < q_0) \to (q'_{n'} < \cdots <
q'_0)$ are equivariant with respect to this action in the sense that
the following diagram commutes:

\begin{equation}\label{eqn.comm_diag_A_action}
  \begin{diagram}
    \node{(g_n, \ldots, g_0) \otimes (\phi_n, \ldots, \phi_1)}
    \arrow{e,t}{\mu}
    \arrow{s}
    \node{(g_n \phi_n g_{n-1}^{-1}, \ldots, g_1\phi_1 g_0^{-1})}
    \arrow{s}
    \\
    \node{(g'_{n'}, \ldots, g'_0) \otimes (\phi'_{n'}, \ldots, \phi'_1)}
    \arrow{e,t}{\mu}
    \node{(g'_{n'} \phi'_{n'} (g'_{n'-1})^{-1}, \ldots, g'_1\phi'_1 (g'_0)^{-1})}
  \end{diagram}
\end{equation}

Now, given the augmentation ideal $I$ of $A$, define a related functor,
$\mathcal{E}pi_{I} \co \mathcal{S}_0 \to k$-$\mathbf{Mod}$ by
\[
  (q_n < \cdots < q_0) \mapsto k\left[ \mathcal{E}pi(q_n < \cdots <
  q_0)\right] \otimes_k B_{q_0}^{sym}I
\]
The effect of $\mathcal{E}pi_{I}$ on morphisms of $\mathcal{S}_0$
requires some explanation.  If $(q_n < \cdots < q_0) \to (q'_{n'} <
\cdots < q'_0)$ is a morphism of $\mathcal{S}_0$, then the map $
\mathcal{E}pi_I(q_n < \cdots < q_0) \to \mathcal{E}pi_I(q'_{n'} <
\cdots < q'_0)$ is defined on generators $(\phi_n, \ldots, \phi_1)
\otimes x$ by applying the map $\mathcal{E}pi((q_n < \cdots < q_0) \to
(q'_{n'} < \cdots < q'_0)) \otimes \phi_{r}\cdots \phi_{1}$, if $q'_0
= q_r$.  That is, morphisms ``on the left'' are thrown out; morphisms
``in the middle'' are composed; and morphisms ``on the right'' are
applied to the element of $B_{q_0}^{sym}I$ to produce an element of
$B_{q'_0}^{sym}I$.  Each $k\left[\mathcal{E}pi(q_n < \cdots <
  q_0)\right] \otimes B_{q_0}^{sym}I$ becomes an $\mathcal{A}ut(q_n <
\cdots < q_0)$--module in the obvious way.
Diagram~\ref{eqn.comm_diag_A_action} implies that there are
well-defined functors, $H_q(\mathcal{A}ut-, \mathcal{E}pi_I-) \co
\mathcal{S}_0 \to k$--\textbf{Mod}.

\begin{theorem}\label{thm.group_homology}
  With $\mathcal{A}ut$ and $\mathcal{E}pi_I$ defined as above, there is a
  spectral sequence
  \[
    {\colim_{\mathcal{S}_0}}^{(p)} H_q(\mathcal{A}ut-, \mathcal{E}pi_I-) 
    \Rightarrow \widetilde{H}S_{p+q}(A)
  \]
\end{theorem}

\begin{proof}
  This is essentially Prop.~1.6 of~\cite{Sl}, reinterpreted in terms
  of the derived functors of $\colim$.  Compare also Prop.~1.8, where
  variance of the functor $N$ is opposite, and cohomology is computed.
  Let $Q$ stand for $(q_n < \cdots < q_0)$.  For each $Q$, let
  $\mathcal{G}Q$ be the category consisting of a single object whose
  endomorphism set is the group $\mathcal{A}ut(Q)$.  Then we obtain:
  \begin{eqnarray*}
    {\colim_{Q \in \mathcal{S}_0}}^{(p)}
    H_q(\mathcal{A}ut(Q),\mathcal{E}pi_I(Q)) &=&
    {\colim_{Q \in \mathcal{S}_0}}^{(p)} {\colim_{\mathcal{G}Q}}^{(q)}
    \left(\mathcal{E}pi_I(Q)\right)\\
    &\Rightarrow& {\colim_{\mathrm{Epi}
        \Delta{S}}}^{(p+q)} B_*^{sym}I \cong \widetilde{H}S_{p+q}A.
  \end{eqnarray*}
\end{proof}

%%%%%%%%%%%%%%%%%%%%%%%%%%%%%%%%%%%%%%%%%%%%%%%%%%%%%%%%%%%%%%%%%%%%%%%%%%%%%%
\subsection{Implications in Characteristic 0}
If $k$ is a field of characteristic 0, then for any finite group $G$
and $kG$--module $M$, $H_q( G, M ) = 0$ for all $q > 0$ (see~\cite{B},
for example).  Thus, the spectral sequence of
Thm.~\ref{thm.group_homology}, would collapse giving the following

\begin{cor}\label{cor.char_0}
  If $k$ is a field of characteristic $0$, and $A$ is a $k$-algebra,
  then
  \[
    \widetilde{H}S_p(A) = {\colim_{Q \in \mathcal{S}_0}}^{(p)} 
    \mathcal{E}pi_I(Q)_{\mathcal{A}ut(Q)},
  \]
  where $\mathcal{E}pi_I(Q)_{\mathcal{A}ut(Q)}$ is the $k$-space of
  coinvariants of $\mathcal{E}pi_I(Q)$ under the action of
  $\mathcal{A}ut(Q)$.
\end{cor}

%%%%%%%%%%%%%%%%%%%%%%%%%%%%%%%%%%%%%%%%%%%%%%%%%%%%%%%%%%%%%%%%%%%%%%%%%%%%%%%%
\section{A Second Spectral Sequence}\label{chap.spec_seq2}
%%%%%%%%%%%%%%%%%%%%%%%%%%%%%%%%%%%%%%%%%%%%%%%%%%%%%%%%%%%%%%%%%%%%%%%%%%%%%%%%

In Section~\ref{sec.specseq}, a spectral sequence was found that
essentially separates the symmetric group elements from the morphisms
comprising the chains of the resolution $C_*( \mathrm{Epi}\Delta S,\,
B_*^{sym}I)$.  What remains is still rather cumbersome.  The following
can be viewed as an attempt to separate the combinatorial structure
inherent in the category $\mathrm{Epi}\Delta S$ from the more
complicated structure generally found in $B_*^{sym}I$.  Indeed, the
$E^1$ term of our spectral sequence isolates the multiplicative
structure of $B_*^{sym}I$ into the action of the differential.

%%%%%%%%%%%%%%%%%%%%%%%%%%%%%%%%%%%%%%%%%%%%%%%%%%%%%%%%%%%%%%%%%%%%%%%%%%%%%%
\subsection{Filtering by Degree}

Consider a filtration $\mathscr{F}_*$ of $C_*(\mathrm{Epi}\Delta S,\,
B_*^{sym}I)$ given by,
\[
  \mathscr{F}_pC_q(\mathrm{Epi}\Delta S,\, B_*^{sym}I) = \bigoplus_{p
    \geq m_0 \geq \ldots \geq m_q} 
    k\left[ \prod_{i=1}^{q}\mathrm{Epi}_{\Delta S}([m_{i-1}], [m_i])\right] 
    \otimes B_{m_0}^{sym}I
\]
$\mathscr{F}_*$ filters the complex by the {\it length} of $x \in
B^{sym}_*I$.  Here, length refers to the number of tensor components
of $x$.  Tensors of $B^{sym}_{m_0}I$ are considered to have length
$m_0 + 1$.  The only face map that can potentially change the length
of $x$ is $d_0$, and since all morphisms are epic, $d_0$ can only
reduce the length.  Thus, $\mathscr{F}_*$ is compatible with the
differential.  The filtration quotients are easily described.  The
non-zero chains are those corresponding to $p = m_0$.  In other words,
the length of $x$ is exactly $p+1$.  $E^0$ splits into a direct sum
based on the product of $x_i$'s in $u = (x_0, \ldots, x_p) \in
X^{p+1}$.  For $u \in X^{p+1}$, let $P_u$ be the set of all distinct
permutations of $u$.  Then,
\[
  E^0_{p,q} = 
  \bigoplus_{u \in X^{p+1}/\Sigma_{p+1}} 
  \bigoplus_{w \in P_u}
  \left(
  \bigoplus_{p = m_0 \geq \ldots \geq m_q} 
  k\left[\prod_{i=1}^q \mathrm{Epi}_{\Delta S}([m_{i-1}],
    [m_i])\right] \otimes w\right).
\]

%%%%%%%%%%%%%%%%%%%%%%%%%%%%%%%%%%%%%%%%%%%%%%%%%%%%%%%%%%%%%%%%%%%%%%%%%%%%%%
\subsection{The Categories $\widetilde{\mathcal{S}}_p$, 
$\widetilde{\mathcal{S}}'_p$, $\mathcal{S}_p$ and $\mathcal{S}'_p$}
Before proceeding with the main theorem of this section, we must
define four related categories.  In the definitions that follow, let
$\{z_0, z_1, z_2, \ldots z_p\}$ be a set of formal non-commuting
indeterminates.
\begin{definition}
  $\widetilde{\mathcal{S}}_p$ is the category with objects formal
  tensor products $Z_0 \otimes \ldots \otimes Z_s$, where each $Z_i$
  is a non-empty product of $z_i$'s, and every one of $z_0, z_1,
  \ldots, z_p$ occurs exactly once in the tensor product.  There is a
  unique morphism $Z_0 \otimes \ldots \otimes Z_s \to Z'_0 \otimes
  \ldots \otimes Z'_t$, if and only if the tensor factors of the
  latter are products of the factors of the former in some order.  In
  such a case, there is a unique $\beta \in \mathrm{Epi}\Delta S$ so
  that $\beta_*(Z_0 \otimes \ldots \otimes Z_s) = Z'_0 \otimes \ldots
  \otimes Z'_t$.
\end{definition}
$\widetilde{\mathcal{S}}_p$ has initial objects $\sigma_*(z_0 \otimes
z_1 \otimes \ldots \otimes z_p)$, for $\sigma \in
\Sigma^{\mathrm{op}}_{p+1}$, so $N\widetilde{\mathcal{S}}_p$ is a
contractible complex.  Let $\widetilde{\mathcal{S}}'_p$ be the full
subcategory of $\widetilde{\mathcal{S}}_p$ with all objects
$\sigma_*(z_0 \otimes \ldots \otimes z_p)$ deleted.
  
Let $\mathcal{S}_p$ be a skeletal category equivalent to 
$\widetilde{\mathcal{S}}_p$.  In fact, we may make $\mathcal{S}_p$ the quotient
category, identifying each object $Z_0 \otimes \ldots \otimes Z_s$ with any 
permutation of its tensor factors, and identifying morphisms $\phi$ and $\psi$
if their source and target are equivalent.  This category has nerve 
$N\mathcal{S}_p$ homotopy-equivalent to $N\widetilde{\mathcal{S}}_p$.  Now, 
$\mathcal{S}_p$ is a poset with unique initial object, $z_0 \otimes \ldots 
\otimes z_p$.  Let $\mathcal{S}'_p$ be the full subcategory (subposet) of 
$\mathcal{S}_p$ obtained by deleting the object $z_0 \otimes \ldots \otimes 
z_p$.  Clearly, $\mathcal{S}'_p$ is a skeletal category equivalent to 
$\widetilde{\mathcal{S}}'_p$.

%%%%%%%%%%%%%%%%%%%%%%%%%%%%%%%%%%%%%%%%%%%%%%%%%%%%%%%%%%%%%%%%%%%%%%%%%%%%%%
\subsection{Main Theorem}
\begin{theorem}\label{thm.E1_NS}
There is spectral sequence converging (weakly) to $\widetilde{H}S_*(A)$ with
\[
  E^1_{p,q} \cong \bigoplus_{u\in X^{p+1}/\Sigma_{p+1}} H_{p+q}
  \left(EG_{u}\ltimes_{G_u} |N\mathcal{S}_p/N\mathcal{S}'_p|;k\right),
\]
where $G_{u}$ is the isotropy subgroup for the chosen representative of $u \in 
X^{p+1}/ \Sigma_{p+1}$.
\end{theorem}
Recall, for a group $G$, right $G$-space X, and left $G$-space $Y$, $X
\ltimes_G Y$ denotes the \textit{equivariant half-smash product}.  If
$\ast$ is a chosen basepoint for $Y$ having trivial $G$-action, then
$X \ltimes_G Y \stackrel{def}{=} (X \times_G Y)/(X \times_G \ast) = X
\times Y/\approx$, with equivalence relation defined by $(x.g , y)
\approx (x, g.y)$ and $(x, \ast) \approx (x', \ast)$ for all $x, x'
\in X$, $y \in Y$ and $g \in G$ (cf.~\cite{M4}).  In our case, $X$ is
of the form $EG$, with canonical underlying complex $E_*G$, equipped
with a right $G$-action, $(g_0, g_1, \ldots, g_n).g = (g_0, g_1,
\ldots, g_ng)$.

Observe, both $N\widetilde{\mathcal{S}}_p$ and $N\widetilde{\mathcal{S}}'_p$ 
carry a left $\Sigma_{p+1}$-action (hence also a $G_u$-action).  The action is
defined on $0$-chains $Z_0 \otimes \ldots \otimes Z_s$ by permutation of
the individual indeterminates, $z_0, z_1, \ldots, z_p$.  This action extends
to $n$-chains in the straightforward manner.

Define for each $u \in X^{p+1}/\Sigma_{p+1}$, the following subcomplex of 
$E^0_{p,q}$:
\[
  \mathscr{M}_u \stackrel{def}{=}
  \bigoplus_{w \in P_u} \left(
  \bigoplus_{p = m_0 \geq \ldots \geq m_q} 
  k\left[ \prod_{i=1}^q \mathrm{Epi}_{\Delta S}
  \left([m_{i-1}], [m_i]\right)\right] \otimes w\right)
\]
\begin{lemma}\label{lem.G_u-identification}
  There is a chain-isomorphism, $\left(N\widetilde{\mathcal{S}}_p/
  N\widetilde{\mathcal{S}}'_p\right)/G_u \stackrel{\cong}{\to} \mathscr{M}_u$.
\end{lemma}
\begin{proof}
  Let $C_j$ denote objects of $\widetilde{\mathcal{S}}_p$.  As above, we may
  view each $C_j$ as a morphism of $\Delta S$.  By abuse of notation, let
  $C_j$ also represent a $0$-cell of $N\widetilde{\mathcal{S}}_p/
  N\widetilde{\mathcal{S}}'_p$.  Denote the chosen representative of $u$ again
  by $u$ (We view $u =(x_{i_0}, x_{i_1}, \ldots, x_{i_p}) \in X^{p+1}/
  \Sigma_{p+1}$ as represented by a $(p+1)$-tuple whose indices are in 
  non-decreasing order).
  
  First define a map $N\widetilde{\mathcal{S}}_p \longrightarrow \mathscr{M}_u$
  on $n$-cells by: 
  \begin{equation}\label{eq.forward-map}
    \alpha_*\co \left(C_q \stackrel{\phi_q}{\gets} \ldots
    \stackrel{\phi_1}{\gets} C_0\right) \mapsto 
    (\phi_q, \ldots, \phi_1) \otimes C_0(u).
  \end{equation}
  The notation $C_0(u)$ is used in place of the more correct
  $(C_0)_*(u)$ in order to avoid clutter.  I claim $\alpha_*$ factors
  through $N\widetilde{\mathcal{S}}_p/N\widetilde{\mathcal{S}}'_p$.
  Indeed, if $C_q \stackrel{\phi_q}{\gets} \ldots
  \stackrel{\phi_1}{\gets} C_0 \in
  N\widetilde{\mathcal{S}}'_p$, then we cannot have $C_0 = \sigma(z_0
  \otimes \ldots \otimes z_p)$ for any symmetric group element
  $\sigma$.  That is, $C_0$, viewed as a morphism, must be strictly
  epic.  Then the length of of $C_0(u)$ is strictly less than the
  length of $u$, which would make $(\phi_q, \ldots, \phi_1) \otimes
  C_0(u)$ trivial in $E^0_{p,q}$.
    
  The map $\alpha_*$ then factors through $\left(N\widetilde{\mathcal{S}}_p/
  N\widetilde{\mathcal{S}}'_p \right)/G_u$, since if $\gamma \in G_u$, then 
  $\gamma$ corresponds to an automorphism 
  $g \in \Sigma_{p+1}^{\mathrm{op}}$, and by definition, we have:
  \begin{eqnarray*}
    \gamma.\left(C_q \stackrel{\phi_q}{\gets} \ldots \stackrel{\phi_1}{\gets}
      C_0\right)
    = \left(C_q g \stackrel{\phi_q}{\gets} \ldots \stackrel{\phi_1}{\gets}
      C_0 g\right)
    &\mapsto& (\phi_q, \ldots, \phi_1) \otimes C_0(g(u))\\
    &=& (\phi_q, \ldots, \phi_1) \otimes C_0(u)
  \end{eqnarray*}
  (Note, $g(u) = u$ follows from the fact that $\gamma \in 
  G_u$, the isotropy subgroup for $u$).

  For the map in the opposite direction, consider $(\phi_q, \ldots,
  \phi_1) \otimes w$ for $w \in P_u$.  Let $t \in
  \Sigma^{\mathrm{op}}_{p+1}$ so that $w = t(u)$.  Define a map
  sending,
  \begin{equation}\label{eq.backward-map}
    \beta_* \co (\phi_q, \ldots, \phi_1)\otimes w
    \mapsto
    \left( (\phi_q\cdots\phi_1t) \stackrel{\phi_q}{\gets} \cdots
    \stackrel{\phi_2}{\gets} \phi_1t \stackrel{\phi_1}{\gets} t \right)
  \end{equation}
  We must check that the definition of $\beta_*$ does not depend on choice of 
  $t$.  Indeed, if $w = s(u)$ also, then $u = s^{-1}t(u)$, hence $s^{-1}t \in 
  G_u^{\mathrm{op}}$.
  Thus, 
  \begin{eqnarray*}
    \left( (\phi_q\ldots\phi_1s) \stackrel{\phi_q}{\gets} \ldots 
      \stackrel{\phi_2}{\gets} \phi_1s \stackrel{\phi_1}{\gets} s \right)
    &\approx& 
    (s^{-1}t) . \left( (\phi_q \cdots\phi_1s)\stackrel{\phi_q}{\gets} \ldots
       \stackrel{\phi_2}{\gets} \phi_1s \stackrel{\phi_1}{\gets} s\right) \\
    &=& 
    \left( (\phi_q \cdots\phi_1t)\stackrel{\phi_q}{\gets} \ldots
       \stackrel{\phi_2}{\gets} \phi_1t \stackrel{\phi_1}{\gets} t\right)
  \end{eqnarray*}
  The maps $\alpha_*$ and $\beta_*$ are clearly inverse to one
  another.  All that remains is to verify that they are chain maps.
  We need only check compatibility with the zeroth face maps in either
  case, since the $i^{th}$ face maps (for $i>0$) simply compose the
  morphisms $\phi_{i+1}$ and $\phi_i$ in either chain complex.  The
  zeroth face maps of either complex will be denoted $d_0$.
  
  First consider the map $\alpha_*$.
  \[
    \begin{diagram}
      \node{ \left(C_q \stackrel{\phi_q}{\gets} \ldots \stackrel{\phi_1}{\gets}
        C_0\right)}
      \arrow[2]{e,t,T}{d_0}
      \arrow{s,l,T}{\alpha_*}
      \node[2]{ \left(C_q \stackrel{\phi_q}{\gets} \ldots \stackrel{\phi_2}
        {\gets}C_1\right)}
      \arrow{s,l,T}{\alpha_*}
      \\
      \node{ (\phi_q, \ldots, \phi_1) \otimes C_0(u) }
      \arrow{e,t,T}{d_0}
      \node{ (\phi_q, \ldots, \phi_2) \otimes \phi_1C_0(u) }
      \arrow{e,t,=}{}
      \node{ (\phi_q, \ldots, \phi_2) \otimes C_1(u)}
    \end{diagram}
  \]
  The equality in the lower right of the diagram is simply a restatement that
  $C_0 \stackrel{\phi_1}{\to} C_1$ is a morphism of $\widetilde{\mathcal{S}}_p$.
  
  For the reverse direction, assume $w = t(u)$ as above.  The following
  diagram commutes.
  \[
    \begin{diagram}
      \node{ (\phi_q, \ldots, \phi_1)\otimes w }
      \arrow{e,t,T}{d_0}
      \arrow{s,l,T}{\beta_*}
      \node{ (\phi_q, \ldots, \phi_2)\otimes \phi_1(w) }
      \arrow{s,l,T}{\beta_*}
      \\
      \node{ \left( (\phi_q \cdots\phi_1t)\stackrel{\phi_q}{\gets} \ldots
       \stackrel{\phi_2}{\gets} \phi_1t \stackrel{\phi_1}{\gets} t\right) }
      \arrow{e,t,T}{d_0}
      \node{ \left( (\phi_q \cdots\phi_1t)\stackrel{\phi_q}{\gets} \ldots
       \stackrel{\phi_2}{\gets} \phi_1t \right)}
    \end{diagram}
  \]
  The lower righthand corner deserves some explanation.  Since the
  target of $\beta_*$ should be $\left(N\widetilde{\mathcal{S}}_p/
  N\widetilde{\mathcal{S}}'_p \right)/G_u$, we see that the lower
  righthand term is trivial unless $\phi_1$ is an isomorphism.  In
  other words, we only need to prove the diagram commutes when $\phi_1
  \in \Sigma^{\mathrm{op}}_{p+1}$.  As $\beta_*$ is applied to the
  chain on the upper righthand corner, the clear choice for group
  element to begin the chain is $\phi_1t$, since $\phi_1(w) =
  \phi_1t(u)$.
\end{proof}
Using Lemma~\ref{lem.G_u-identification}, we identify $\mathscr{M}_u$
with the orbit complex
$\big(N\widetilde{\mathcal{S}}_p/N\widetilde{\mathcal{S}}'_p\big)/G_u$.
Now, the complex
$N\widetilde{\mathcal{S}}_p/N\widetilde{\mathcal{S}}'_p$ is a free
$G_u$-complex, so we have an isomorphism,
\[
  H_*\left(\left(
  N\widetilde{\mathcal{S}}_p/N\widetilde{\mathcal{S}}'_p\right)/G_u\right)\cong
  H^{G_u}_*\big(N\widetilde{\mathcal{S}}_p/N\widetilde{\mathcal{S}}'_p\big),
\]
where $H^{G_u}_*$ is $G_u$-equivariant homology.  See~\cite{B} for
details).  Then, by definition,
\[
  H^{G_u}_*\left(N\widetilde{\mathcal{S}}_p/
  N\widetilde{\mathcal{S}}'_p\right) = H_*\left(G_u,
  N\widetilde{\mathcal{S}}_p/ N\widetilde{\mathcal{S}}'_p\right),
\] 
which may be computed using the free resolution, $E_*G_u$ of $k$ as
right $G_u$--module.  The resulting complex $k[E_*G_u] \otimes_{kG_u}
k\left[N\widetilde{\mathcal{S}}_p\right]/k\left[
  N\widetilde{\mathcal{S}}'_p\right]$ is a double complex isomorphic
to the quotient,
\begin{eqnarray*}
  \lefteqn{\left(k[E_*G_u]
    \otimes_{kG_u}k\left[N\widetilde{\mathcal{S}}_p\right]\right)/
    \left(k[E_*G_u]
    \otimes_{kG_u}k\left[N\widetilde{\mathcal{S}}'_p\right]\right)}\\
    &\cong& k\left[ \left(E_*G_u \times_{G_u}
      N\widetilde{\mathcal{S}}_p\right)/ \left(E_*G_u \times_{G_u}
      N\widetilde{\mathcal{S}}'_p\right)\right].
\end{eqnarray*}
This last complex may be identified with the simplicial complex of the space,
\[
  \left(EG_u \times_{G_u} |N\widetilde{\mathcal{S}}_p|\right)/\left(EG_u 
  \times_{G_u} |N\widetilde{\mathcal{S}}'_p|\right) \cong EG_u \ltimes_{G_u} 
  |N\widetilde{\mathcal{S}}_p/N\widetilde{\mathcal{S}}'_p|.
\]

The last piece of the puzzle involves simplifying the spaces 
$|N\widetilde{\mathcal{S}}_p/N\widetilde{\mathcal{S}}'_p|$.  Since $\mathcal{S}$
is a skeletal subcategory of $\widetilde{\mathcal{S}}$, there is an equivalence
of categories $\widetilde{\mathcal{S}} \simeq \mathcal{S}$, inducing a homotopy
equivalence of complexes (hence also of spaces) $|N\widetilde{\mathcal{S}}| 
\simeq |N\mathcal{S}|$.  Note that $N\mathcal{S}$ inherits a $G_u$-action from 
$N\widetilde{\mathcal{S}}$, and the map $\widetilde{\mathcal{S}} \to 
\mathcal{S}$ is $G_u$-equivariant.  
\begin{prop}\label{prop.weak-eq}
  There are weak equivalences, $EG_u \times_{G_u} |N\widetilde{\mathcal{S}}_p| 
  \to EG_u \times_{G_u} |N\mathcal{S}_p|$ and $EG_u \times_{G_u} 
  |N\widetilde{\mathcal{S}}'_p| \to EG_u \times_{G_u} |N\mathcal{S}'_p|$, 
  inducing a weak equivalence $EG_u \ltimes_{G_u} |N\widetilde{\mathcal{S}}_p/
  N\widetilde{\mathcal{S}}'_p| \to EG_u \ltimes_{G_u}|N\mathcal{S}_p/
  N\mathcal{S}'_p|$.
\end{prop}
\begin{proof}
  The case $p > 2$ will be handled first.  As long as the spaces are
  path connected, we can use the fibration sequences associated to $X
  \to EG_u \times_{G_u} X \to BG_u$ to show that a $G_u$-equivariant
  homotopy equivalence $X \simeq Y$ induces $EG_u \times_{G_u} X
  \simeq EG_u \times_{G_u} Y$.  
  $|N\widetilde{\mathcal{S}}_p|$ and $|N\mathcal{S}_p|$ are
  path-connected because they are contractible.  It suffices to
  show  $|N\mathcal{S}'_p|$ path connected, since 
  $|N\widetilde{\mathcal{S}}'_p| \simeq |N\mathcal{S}'_p|$.

  let $W_0 = z_0z_1 \otimes z_2 \otimes \ldots \otimes
  z_p$.  This represents a vertex of $N\mathcal{S}'_p$.  Suppose $W =
  Z_0 \otimes \ldots \otimes Z_i'z_0z_1Z_i'' \otimes \ldots \otimes
  Z_s$.  Then there is a morphism $W_0 \to W$, hence an edge between
  $W_0$ and $W$.
  
  Next, suppose $W = Z_0 \otimes \ldots \otimes Z_i'z_0Z_i''z_1Z_i''' \otimes 
  \ldots \otimes Z_s$.  There is a path:
  \[
    \begin{diagram}
      \node{Z_0 \otimes \ldots \otimes Z_i'z_0Z_i''z_1Z_i''' \otimes \ldots 
        \otimes Z_s}
      \arrow{s}\\
      \node{Z_0Z_1 \ldots Z_i'z_0Z_i''z_1Z_i''' \ldots Z_s}\\
      \node{Z_0Z_1 \ldots Z_i' \otimes z_0 \otimes Z_i''z_1Z_i''' \ldots Z_s}
      \arrow{n}
      \arrow{s}\\
      \node{z_0 \otimes Z_0Z_1 \ldots Z_i'Z_i''z_1Z_i''' \ldots Z_s}\\
      \node{z_0 \otimes Z_0Z_1 \ldots Z_i'Z_i'' \otimes z_1Z_i''' \ldots Z_s}
      \arrow{n}
      \arrow{s}\\
      \node{z_0z_1Z_i''' \ldots Z_s \otimes Z_0Z_1 \ldots Z_i'Z_i''}\\
      \node{W_0}
      \arrow{n}
    \end{diagram}
  \]
  
  Similarly, if $W = Z_0 \otimes \ldots \otimes Z_i'z_1Z_i''z_0Z_i''' \otimes 
  \ldots \otimes Z_s$, there is a path to $W_0$.  Finally, if
  $W = Z_0 \otimes \ldots \otimes Z_s$ with $z_0$ occurring in $Z_i$ and
  $z_1$ occurring in $Z_j$ for $i \neq j$, there is an edge to some 
  $W'$ in which $Z_iZ_j$ occurs, and thus a path to $W_0$.

  The cases $p = 0, 1$ and $2$ are handled individually:

  Observe that $|N\widetilde{\mathcal{S}}'_0|$ and $|N\mathcal{S}'_0|$ are empty
  spaces, since $\widetilde{\mathcal{S}}'_0$ has no objects.  Hence, $EG_u 
  \times_{G_u} |N\widetilde{\mathcal{S}}'_0| = EG_u \times_{G_u} 
  |N\mathcal{S}'_0| = \emptyset$.  Furthermore, any group $G_u$ must be trivial.
  Thus there is a chain of homotopy equivalences, $EG_u \ltimes_{G_u}
  |N\widetilde{\mathcal{S}}_0/N\widetilde{\mathcal{S}}'_0| \simeq 
  |N\widetilde{\mathcal{S}}_0| \simeq |N\mathcal{S}_0| \simeq EG_u \ltimes_{G_u}
  |N\mathcal{S}_p/N\mathcal{S}'_p|$.

  Next, since $|N\widetilde{\mathcal{S}}'_1|$ is homeomorphic to 
  $|N\mathcal{S}'_1|$, each space consisting of the two discrete points 
  $z_0z_1$ and $z_1z_0$ with the same group action, the proposition is true for
  $p = 1$ as well.  

  For $p=2$, observe that $|N\widetilde{\mathcal{S}}'_2|$ has two connected 
  components, $\widetilde{U}_1$ and $\widetilde{U}_2$ that are interchanged by 
  any odd permutation $\sigma \in \Sigma_3$.  Similarly, $|N\mathcal{S}'_2|$ 
  consists of two connected components, $U_1$ and $U_2$, interchanged by any odd
  permutation of $\Sigma_3$.  Now, resticted to the alternating group, $A_3$, we
  certainly have weak equivalences for any subgroup $H_u \subseteq A_3$,
  $EH_u \times_{H_u} \widetilde{U}_1 \stackrel{\simeq}{\longrightarrow} EH_u 
  \times_{H_u} U_1$ and $EH_u \times_{H_u} \widetilde{U}_2 \stackrel{\simeq}
  {\longrightarrow} EH_u \times_{H_u} U_2$.  The action of an odd permutation
  induces equivariant homeomorphisms $\widetilde{U}_1 \stackrel{\cong}
  {\longrightarrow} \widetilde{U}_2$ and $U_1 \stackrel{\cong}{\longrightarrow}
  U_2$, and so if we have a subgroup $G_u \subseteq \Sigma_3$ generated by $H_u
  \subseteq A_3$ and a transposition, then the two connected components are
  identified in an \mbox{$A\Sigma_3$-equivariant} manner.  Thus, if $G_u$ 
  contains a transposition, $EG_u \times_{G_u} |N\widetilde{\mathcal{S}}'_2| 
  \cong EH_u \times_{H_u} \widetilde{U}_1 \simeq EH_u \times_{H_u} U_1 \cong 
  EG_u \times_{G_u} |N\mathcal{S}'_2|$.  This completes the case $p=2$ and the 
  proof of Prop.~\ref{prop.weak-eq}.
\end{proof}
  
Prop.~\ref{prop.weak-eq} coupled with Lemma~\ref{lem.G_u-identification} 
produces the required isomorphism in homology, hence proving 
Thm.~\ref{thm.E1_NS}:
\[
  E^1_{p,q} \;=\; \bigoplus_{u \in X^{p+1}/\Sigma_{p+1}} H_{p+q}(\mathscr{M}_u) 
  \;\cong\; \bigoplus_{u \in X^{p+1}/\Sigma_{p+1}}H_{p+q}\left( EG_u 
  \ltimes_{G_u} |N\widetilde{\mathcal{S}}_p / N\widetilde{\mathcal{S}}'_p|; k
  \right)
\]
\[
  \cong\; \bigoplus_{u \in X^{p+1}/\Sigma_{p+1}} H_{p+q}\left(
  EG_u \ltimes_{G_u} |N\mathcal{S}_p/N\mathcal{S}'_p|; k\right).
\]
\begin{cor}\label{cor.square-zero}
If the augmentation ideal of $A$ satisfies $I^2 = 0$, then 
\[
  \widetilde{H}S_n(A) \cong \bigoplus_{p \geq 0} \bigoplus_{u \in
    X^{p+1}/\Sigma_{p+1}} H_{n}(EG_u \ltimes_{G_u}
  N\mathcal{S}_p/N\mathcal{S}'_p; k).
\]
\end{cor}
\begin{proof}
  This follows from consideration of the original $E^0$ term of the
  spectral sequence.  $E^0$ is generated by chains $(\phi_n, \ldots,
  \phi_1)\otimes w$ with induced differential $d^0$, agreeing with the
  differential $d$ of $C_*(\mathrm{Epi}\Delta S,\, B_*^{sym}I)$ when
  $\phi_0$ is an isomorphism.  When $\phi_0$ is a strict epimorphism,
  however, the zeroth face map of $d$ maps the generator to:
  $(\phi_n, \ldots, \phi_1) \otimes (\phi_0)_*(Y) = 0$,
  since $(\phi_0)_*(Y)$ would have at least one tensor factor that is
  the product of two or more elements of $I$.  Thus, $d^0$ also agrees
  with $d$ in the case that $\phi_0$ is strictly epic.  Hence, the
  spectral sequence collapses at level 1.
\end{proof}

%%%%%%%%%%%%%%%%%%%%%%%%%%%%%%%%%%%%%%%%%%%%%%%%%%%%%%%%%%%%%%%%%%%%%%%%%%%%%%
\subsection{The complex $Sym_*^{(p)}$}
Note, for $p > 0$, there are homotopy equivalences $|N\mathcal{S}_p/
N\mathcal{S}'_p| \simeq |N\mathcal{S}_p| \vee S|N\mathcal{S}'_p|
\simeq S|N\mathcal{S}'_p|$, since $|N\mathcal{S}_p|$ is contractible.
$|N\mathcal{S}_p|$ is a disjoint union of $(p+1)!$ $p$-cubes,
identified along certain faces.  Geometric analysis of
$S|N\mathcal{S}'_p|$, however, seems quite difficult.  Fortunately,
there is an even smaller chain complex homotopic to
$N\mathcal{S}_p/N\mathcal{S}'_p$.
\begin{definition}\label{def.sym_complex}
  Let $p \geq 0$ and impose an equivalence relation on $k\left[
  \mathrm{Epi}_{\Delta S} ([p], [q])\right]$ generated by:
  \[
    Z_0 \otimes \ldots \otimes Z_i \otimes Z_{i+1} \otimes \ldots \otimes Z_q
    \approx (-1)^{ab} Z_0 \otimes \ldots \otimes Z_{i+1} \otimes Z_{i} \otimes
    \ldots \otimes Z_q,
  \]
  where $Z_0 \otimes \ldots \otimes Z_q$ is a morphism expressed in
  tensor notation, and $a = deg(Z_i)$ and $b = deg(Z_{i+1})$, where
  $deg(Z) \stackrel{def}{=} |Z| - 1$, one less than the number of
  tensor factors of $Z$.  The complex $Sym_*^{(p)}$ is then defined by
  \[
    Sym_i^{(p)} \stackrel{def}{=} k\left[ \mathrm{Epi}_{\Delta S}([p],
    [p-i])\right]/\approx.
  \]
  The face maps will be defined
  recursively.  On monomials,
  \[
    d_i(z_{j_0} \ldots z_{j_s}) = \left\{\begin{array}{ll}
       0, & i < 0,\\
       z_{j_0}   \ldots   z_{j_i} \otimes z_{j_{i+1}}   \ldots   z_{j_s},
       \quad,& 0 \leq i < s,\\
       0, & i \geq s.
       \end{array}\right.
  \]
  Then, extend $d_i$ to tensor products via:  
  \begin{equation}\label{eq.face_i-tensor}
  d_i(W \otimes V) = d_i(W) \otimes V + W \otimes d_{i-deg(W)}(V),
  \end{equation}
  where $W$ and $V$ are formal tensors in $k\left[\mathrm{Epi}_{\Delta
      S}([p], [q])\right]$, and $deg(W) = deg(W_0 \otimes \ldots
  \otimes W_t) \stackrel{def}{=} \sum_{k=0}^t deg(W_k)$.  The boundary
  map $Sym_n^{(p)} \to Sym_{n-1}^{(p)}$ is then $d = \sum_{i=0}^n
  (-1)^i d_i = \sum_{i=0}^{n-1} (-1)^i d_i$.
\end{definition}
\begin{rmk}\label{rmk.action}
  There is an action $\Sigma_{p+1} \times Sym_i^{(p)} \to Sym_i^{(p)}$, given by
  permuting the formal indeterminates $z_i$.  Furthermore, this action is 
  compatible with the differential.
\end{rmk}
\begin{lemma}\label{lem.NS-Sym-homotopy}
   $Sym_*^{(p)}$ is homotopy-equivalent to $k[N\mathcal{S}_p]/
   k[N\mathcal{S}'_p]$.
\end{lemma}
\begin{proof}
  Let $v_0$ represent the common initial vertex of the $p$-cubes
  making up $N\mathcal{S}_p$.  Then, as cell-complex, $N\mathcal{S}_p$
  consists of $v_0$ together with all corners of the various $p$-cubes
  and $i$-cells for each $i$-face of the cubes.  Thus,
  $N\mathcal{S}_p$ consists of $(p+1)!$ $p$-cells with attaching maps
  \mbox{$\partial I^p \to (N\mathcal{S}_p)^{p-1}$} defined according
  to the face maps for $N\mathcal{S}_p$ given above.  Note that a
  chain of $N\mathcal{S}_p$ is non-trivial in
  $N\mathcal{S}_p/N\mathcal{S}'_p$ if and only if the initial vertex
  $v_0$ is included.  Thus, any (cubical) $k$-cell is uniquely
  determined by the label of the vertex opposite $v_0$.
  
  Label each top-dimensional cell with the permutation induced on the
  set $\{0, 1, \ldots, p\}$ by the order of indeterminates in its
  final vertex, $z_{i_0} z_{i_1}\ldots z_{i_p}$.  On a given $p$-cell,
  for each vertex $Z_0 \otimes \ldots \otimes Z_s$, there is an
  ordering of the tensor factors so that $Z_0 \otimes \ldots \otimes
  Z_s \to z_{i_0}z_{i_1}\ldots z_{i_p}$ preserves the order of formal
  indeterminates $z_i$.  Rewrite each vertex of this $p$-cell in this
  order.  Now, any $p$-chain $(z_{i_0} \otimes z_{i_1} \otimes \ldots
  \otimes z_{i_p}) \to \ldots \to z_{i_0}z_{i_1}\ldots z_{i_p}$ is
  obtained by choosing the order in which to combine the factors.  In
  fact, the $p$-chains for this cube are in bijection with the
  elements of the symmetric group $S_{p}$, as in the standard
  decomposition of a $p$-cube into $p!$ simplices.  A given
  permutation $\{1,2, \ldots, p\}\mapsto \{ j_1, j_2, \ldots, j_p \}$
  will represent the chain obtained by first combining $z_{j_0}
  \otimes z_{j_1}$ into $z_{j_0}z_{j_1}$, then combining $z_{j_1}
  \otimes z_{j_2}$ into $z_{j_1}z_{j_2}$.  In effect, we ``erase'' the
  tensor product symbol between $z_{j_{r-1}}$ and $z_{j_r}$ for each
  $j_r$ in order given by the list above.

  We shall declare that the {\it natural} order of combining the
  factors will be the one that always combines the last two: $(z_{i_0}
  \otimes \ldots \otimes z_{i_{p-1}} \otimes z_{i_p}) \to (z_{i_0}
  \otimes \ldots \otimes z_{i_{p-1}} z_{i_p}) \to (z_{i_0} \otimes
  \ldots \otimes z_{i_{p-2}}z_{i_{p-1}}z_{i_p}) \to \ldots \to
  (z_{i_0} \ldots z_{i_p})$.  This corresponds to a permutation $\rho
  \stackrel{def}{=} \{1,\ldots, p\} \mapsto \{p, p-1, \ldots, 2, 1\}$,
  and this chain will be regarded as {\it positive}.  A chain
  $C_\sigma$, corresponding to another permutation, $\sigma$, will be
  regarded as positive or negative depending on the sign of the
  permutation $\sigma\rho^{-1}$.  Finally, the entire $p$-cell should
  be identified with the sum $\sum_{\sigma \in S_p}
  sgn(\sigma\rho^{-1}) C_\sigma$.  It is this sign convention that
  permits the inner faces of the cube to cancel appropriately in the
  boundary maps.  Thus we have a map on the top-dimensional chains:
  \begin{equation}\label{eq.theta_p-map}
    \theta_p \co Sym_p^{(p)} \to
    \big(k[N\mathcal{S}_p]/k[N\mathcal{S}'_p]\big)_p.
  \end{equation}
  
  Extend the defintion $\theta_*$ to arbitrary $k$-cells by sending
  the $k$-chain $Z_0 \otimes \ldots \otimes Z_{p-k}$ to the sum of
  $k$-length chains with source $z_0 \otimes \ldots \otimes z_p$ and
  target $Z_0 \otimes \ldots \otimes Z_{p-k}$ with signs determined by
  the natural order of erasing tensor product symbols of $z_0 \otimes
  \ldots \otimes z_p$, excluding those tensor product symbols that
  never get erased.  The following example should clarify the point.
  Let $W = z_3z_0 \otimes z_1 \otimes z_2z_4$. This is a $2$-cell of
  $Sym_*^{(4)}$.  $W$ is obtained from $z_0 \otimes z_1 \otimes z_2
  \otimes z_3 \otimes z_4=z_3 \otimes z_0 \otimes z_1 \otimes z_2
  \otimes z_4$ by combining factors in some order.  There are only $2$
  erasable tensor product symbols in this example.  The natural order
  (last to first) corresponds to the chain, $z_3 \otimes z_0 \otimes
  z_1 \otimes z_2 \otimes z_4 \to z_3 \otimes z_0 \otimes z_1 \otimes
  z_2z_4 \to z_3z_0 \otimes z_1 \otimes z_2z_4$.  So, this chain shows
  up in $\theta_*(W)$ with positive sign, whereas the chain $z_3
  \otimes z_0 \otimes z_1 \otimes z_2 \otimes z_4 \to z_3z_0 \otimes
  z_1 \otimes z_2 \otimes z_4 \to z_3z_0 \otimes z_1 \otimes z_2z_4$
  shows up with a negative sign.
  
  Now, $\theta_*$ is easily seen to be a chain map $Sym_*^{(p)} \to
  k[N\mathcal{S}_p]/k[N\mathcal{S}'_p]$.  Geometrically, $\theta_*$
  has the effect of subdividing a cell-complex (defined with cubical
  cells) into a simplicial space, so $\theta_*$ is a
  homotopy-equivalence.  Furthermore, $\theta_*$ is equivariant with
  respect to the action of the symmetric group.
\end{proof}  

\begin{rmk}
  As an example, consider $|N\mathcal{S}_2|$. There are $6$
  \mbox{$2$-cells}, each represented by a copy of $I^2$.  The
  \mbox{$2$-cell} labelled by the permutation \mbox{$\{0,1,2\} \mapsto
    \{1,0,2\}$} consists of the chains $z_1 \otimes z_0 \otimes z_2
  \to z_1 \otimes z_0z_2 \to z_1z_0z_2$ and $-(z_1 \otimes z_0 \otimes
  z_2 \to z_1z_0 \otimes z_2 \to z_1z_0z_2)$.  Hence, the boundary is
  the sum of \mbox{$1$-chains}: $[(z_1 \otimes z_0z_2 \to z_1z_0z_2) -
    (z_1 \otimes z_0 \otimes z_2 \to z_1z_0z_2) + (z_1 \otimes z_0
    \otimes z_2 \to z_1 \otimes z_0z_2)] - [(z_1z_0 \otimes z_2 \to
    z_1z_0z_2) - (z_1 \otimes z_0 \otimes z_2 \to z_1z_0z_2) + (z_1
    \otimes z_0 \otimes z_2 \to z_1z_0 \otimes z_2)] = (z_1 \otimes
  z_0z_2 \to z_1z_0z_2) + (z_1 \otimes z_0 \otimes z_2 \to z_1 \otimes
  z_0z_2) - (z_1z_0 \otimes z_2 \to z_1z_0z_2) - (z_1 \otimes z_0
  \otimes z_2 \to z_1z_0 \otimes z_2)$.  This \mbox{$1$-chains}
  correspond to the $4$ edges of the square.  Thus, in our example
  this $2$-cell of $|N\mathcal{S}_p|$ will correspond to
  \mbox{$z_1z_0z_2 \in Sym_2^{(2)}$}, and its boundary in
  $|N\mathcal{S}_p/N\mathcal{S}'_p|$ will consist of the two edges
  adjacent to the vertex labeled$z_0 \otimes z_1 \otimes z_2$, with
  appropriate signs: $(z_0 \otimes z_1 \otimes z_2 \to z_1 \otimes
  z_0z_2) - (z_0 \otimes z_1 \otimes z_2 \to z_1z_0 \otimes z_2)$.
  The corresponding boundary in $Sym_1^{(2)}$ will be $(z_1 \otimes
  z_0z_2) - (z_1z_0 \otimes z_2)$, matching the boundary map already
  defined on $Sym_*^{(p)}$.  See Figs.~\ref{fig.NS_2}
  and~\ref{fig.sym2}.
\end{rmk}
\begin{center}
\begin{figure}[ht]
  \psset{unit=.6in}
  \begin{pspicture}(6,3.5)
  
  \psdots[linecolor=black, dotsize=4pt]
  (2, 2)(2, 3.5)(3.3, 2.75)(3.3, 1.25)(2, .5)
  (.7, 1.25)(.7, 2.75)
  
  \psline[linewidth=1pt, linecolor=black](2, 2)(2, 3.5)
  \psline[linewidth=1pt, linecolor=black](2, 2)(3.3, 1.25)
  \psline[linewidth=1pt, linecolor=black](2, 2)(.7, 1.25)
  \psline[linewidth=1pt, linecolor=black](2, 3.5)(3.3, 2.75)
  \psline[linewidth=1pt, linecolor=black](3.3, 2.75)(3.3, 1.25)
  \psline[linewidth=1pt, linecolor=black](3.3, 1.25)(2, .5)
  \psline[linewidth=1pt, linecolor=black](2, .5)(.7, 1.25)
  \psline[linewidth=1pt, linecolor=black](.7, 1.25)(.7, 2.75)
  \psline[linewidth=1pt, linecolor=black](.7, 2.75)(2, 3.5)
  
  \rput(1.9, 2.12){$z_0 \otimes z_1 \otimes z_2$}
  \rput(2, 3.62){$z_0z_1 \otimes z_2$}
  \rput(3.28, 1.16){$z_0 \otimes z_1z_2$}
  \rput(.9, 1.16){$z_1 \otimes z_2z_0$}
  \rput(3.5, 2.87){$z_0z_1z_2$}
  \rput(2, .38){$z_1z_2z_0$}
  \rput(.6, 2.9){$z_2z_0z_1$}
  
  \psdots[linecolor=black, dotsize=4pt]
  (6, 2)(6, 3.5)(7.3, 2.75)(7.3, 1.25)(6, .5)
  (4.7, 1.25)(4.7, 2.75)
  
  \psline[linewidth=1pt, linecolor=black](6, 2)(6, 3.5)
  \psline[linewidth=1pt, linecolor=black](6, 2)(7.3, 1.25)
  \psline[linewidth=1pt, linecolor=black](6, 2)(4.7, 1.25)
  \psline[linewidth=1pt, linecolor=black](6, 3.5)(7.3, 2.75)
  \psline[linewidth=1pt, linecolor=black](7.3, 2.75)(7.3, 1.25)
  \psline[linewidth=1pt, linecolor=black](7.3, 1.25)(6, .5)
  \psline[linewidth=1pt, linecolor=black](6, .5)(4.7, 1.25)
  \psline[linewidth=1pt, linecolor=black](4.7, 1.25)(4.7, 2.75)
  \psline[linewidth=1pt, linecolor=black](4.7, 2.75)(6, 3.5)

  \rput(5.9, 2.12){$z_0 \otimes z_1 \otimes z_2$}
  \rput(6, 3.62){$z_1z_0 \otimes z_2$}
  \rput(7.28, 1.16){$z_1 \otimes z_0z_2$}
  \rput(4.9, 1.16){$z_0 \otimes z_2z_1$}
  \rput(7.5, 2.87){$z_1z_0z_2$}
  \rput(6, .38){$z_0z_2z_1$}
  \rput(4.6, 2.9){$z_2z_1z_0$}

  \end{pspicture}

  \caption[$|N\mathcal{S}_2|$]{$|N\mathcal{S}_2|$ consists of six squares, 
    grouped into two hexagons that share a common center vertex}
  \label{fig.NS_2}
\end{figure}
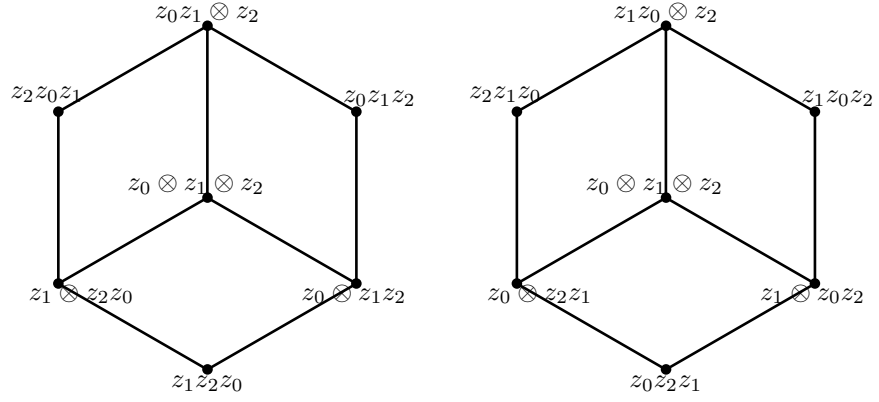
\end{center}

\begin{center}
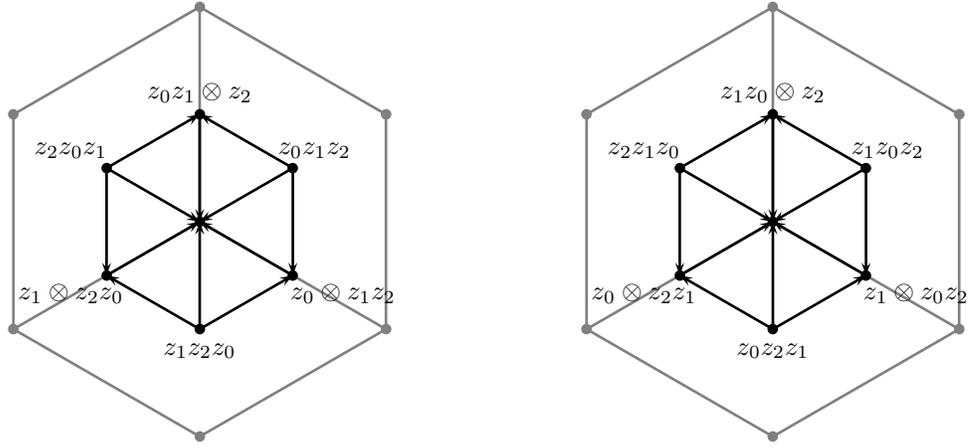
\begin{figure}[ht]
  \psset{unit=.75in}
  \begin{pspicture}(6,3.5)
  
  \psdots[linecolor=gray, dotsize=4pt]
  (2, 3.5)(3.3, 2.75)(3.3, 1.25)(2, .5)
  (.7, 1.25)(.7, 2.75)
  
  \psline[linewidth=1pt, linecolor=gray](2, 2)(2, 3.5)
  \psline[linewidth=1pt, linecolor=gray](2, 2)(3.3, 1.25)
  \psline[linewidth=1pt, linecolor=gray](2, 2)(.7, 1.25)
  \psline[linewidth=1pt, linecolor=gray](2, 3.5)(3.3, 2.75)
  \psline[linewidth=1pt, linecolor=gray](3.3, 2.75)(3.3, 1.25)
  \psline[linewidth=1pt, linecolor=gray](3.3, 1.25)(2, .5)
  \psline[linewidth=1pt, linecolor=gray](2, .5)(.7, 1.25)
  \psline[linewidth=1pt, linecolor=gray](.7, 1.25)(.7, 2.75)
  \psline[linewidth=1pt, linecolor=gray](.7, 2.75)(2, 3.5)
  
  \psdots[linecolor=black, dotsize=4pt]
  (2, 2)(2, 2.75)(2.65, 1.625)(1.35, 1.625)
  (2, 1.25)(1.35, 2.375)(2.65, 2.375)
  
  \psline[linewidth=1pt, linecolor=black]{->}(2, 1.25)(2.65, 1.625)
  \psline[linewidth=1pt, linecolor=black]{->}(2, 1.25)(1.35, 1.625)
  \psline[linewidth=1pt, linecolor=black]{->}(1.35, 2.375)(1.35, 1.625)
  \psline[linewidth=1pt, linecolor=black]{->}(1.35, 2.375)(2, 2.75)
  \psline[linewidth=1pt, linecolor=black]{->}(2.65, 2.375)(2, 2.75)
  \psline[linewidth=1pt, linecolor=black]{->}(2.65, 2.375)(2.65, 1.625)
  \psline[linewidth=1pt, linecolor=black]{->}(2.65, 2.375)(2, 2)
  \psline[linewidth=1pt, linecolor=black]{->}(2, 1.25)(2, 2)
  \psline[linewidth=1pt, linecolor=black]{->}(1.35, 2.375)(2, 2)
  \psline[linewidth=1pt, linecolor=black]{->}(2, 2.75)(2, 2)
  \psline[linewidth=1pt, linecolor=black]{->}(2.65, 1.625)(2, 2)
  \psline[linewidth=1pt, linecolor=black]{->}(1.35, 1.625)(2, 2)
    
  \rput(2, 1.1){$z_1z_2z_0$}
  \rput(1.1, 2.5){$z_2z_0z_1$}
  \rput(2.8, 2.5){$z_0z_1z_2$}
  \rput(3, 1.5){$z_0 \otimes z_1z_2$}
  \rput(1.1, 1.5){$z_1 \otimes z_2z_0$}
  \rput(2, 2.9){$z_0z_1 \otimes z_2$}
  
  \psdots[linecolor=gray, dotsize=4pt]
  (6, 3.5)(7.3, 2.75)(7.3, 1.25)(6, .5)
  (4.7, 1.25)(4.7, 2.75)
  
  \psline[linewidth=1pt, linecolor=gray](6, 2)(6, 3.5)
  \psline[linewidth=1pt, linecolor=gray](6, 2)(7.3, 1.25)
  \psline[linewidth=1pt, linecolor=gray](6, 2)(4.7, 1.25)
  \psline[linewidth=1pt, linecolor=gray](6, 3.5)(7.3, 2.75)
  \psline[linewidth=1pt, linecolor=gray](7.3, 2.75)(7.3, 1.25)
  \psline[linewidth=1pt, linecolor=gray](7.3, 1.25)(6, .5)
  \psline[linewidth=1pt, linecolor=gray](6, .5)(4.7, 1.25)
  \psline[linewidth=1pt, linecolor=gray](4.7, 1.25)(4.7, 2.75)
  \psline[linewidth=1pt, linecolor=gray](4.7, 2.75)(6, 3.5)

  \psdots[linecolor=black, dotsize=4pt]
  (6, 2)(6, 2.75)(6.65, 1.625)(5.35, 1.625)
  (6, 1.25)(5.35, 2.375)(6.65, 2.375)
  
  \psline[linewidth=1pt, linecolor=black]{->}(6, 1.25)(6.65, 1.625)
  \psline[linewidth=1pt, linecolor=black]{->}(6, 1.25)(5.35, 1.625)
  \psline[linewidth=1pt, linecolor=black]{->}(5.35, 2.375)(5.35, 1.625)
  \psline[linewidth=1pt, linecolor=black]{->}(5.35, 2.375)(6, 2.75)
  \psline[linewidth=1pt, linecolor=black]{->}(6.65, 2.375)(6, 2.75)
  \psline[linewidth=1pt, linecolor=black]{->}(6.65, 2.375)(6.65, 1.625)
  \psline[linewidth=1pt, linecolor=black]{->}(6.65, 2.375)(6, 2)
  \psline[linewidth=1pt, linecolor=black]{->}(6, 1.25)(6, 2)
  \psline[linewidth=1pt, linecolor=black]{->}(5.35, 2.375)(6, 2)
  \psline[linewidth=1pt, linecolor=black]{->}(6, 2.75)(6, 2)
  \psline[linewidth=1pt, linecolor=black]{->}(6.65, 1.625)(6, 2)
  \psline[linewidth=1pt, linecolor=black]{->}(5.35, 1.625)(6, 2)

  \rput(6, 1.1){$z_0z_2z_1$}
  \rput(5.1, 2.5){$z_2z_1z_0$}
  \rput(6.8, 2.5){$z_1z_0z_2$}
  \rput(7, 1.5){$z_1 \otimes z_0z_2$}
  \rput(5.1, 1.5){$z_0 \otimes z_2z_1$}
  \rput(6, 2.9){$z_1z_0 \otimes z_2$}

  \end{pspicture}

  \caption[$Sym^{(2)} \simeq N\mathcal{S}_2/N\mathcal{S}'_2$]
          {$Sym^{(2)} \simeq N\mathcal{S}_2/N\mathcal{S}'_2$.  The
            center of each hexagon is $z_0\otimes z_1 \otimes z_2$.}
  \label{fig.sym2}
\end{figure}
\end{center}

Now, with one piece of new notation, we may re-interpret Thm.~\ref{thm.E1_NS}.
\begin{definition}\label{def.ltimescirc}
  Let $G$ be a group.  Let $k_0$ be the chain complex consisting of
  $k$ concentrated in degree $0$, with trivial $G$-action.  If $X_*$
  is a right $G$-complex, $Y_*$ is a left $G$-complex with $k_0
  \hookrightarrow Y_*$ as a $G$-subcomplex, then define the
  \textit{equivariant half-smash tensor product} of the two complexes:
  \[
    X_* \ltimescirc_G Y_* \stackrel{def}{=} \left(X_* \otimes_{kG} 
    Y_*\right)/\left(X_* \otimes_{kG} k_0\right)
  \]
\end{definition}

\begin{cor}\label{cor.E1_Sym}
There is spectral sequence converging (weakly) to $\widetilde{H}S_*(A)$ with
\[
  E^1_{p,q} \cong 
  \bigoplus_{u \in X^{p+1}/\Sigma_{p+1}}
  H_{p+q}\left(E_*G_u \ltimescirc_{G_u} Sym_*^{(p)}; k\right),
\]
where $G_{u}$ is the isotropy subgroup for the chosen representative of $u \in 
X^{p+1}/ \Sigma_{p+1}$.
\end{cor}

%%%%%%%%%%%%%%%%%%%%%%%%%%%%%%%%%%%%%%%%%%%%%%%%%%%%%%%%%%%%%%%%%%%%%%%%%%%%%%%%
\section{Properties of the Complex $Sym_*^{(p)}$}\label{sec.sym_alg}         
%%%%%%%%%%%%%%%%%%%%%%%%%%%%%%%%%%%%%%%%%%%%%%%%%%%%%%%%%%%%%%%%%%%%%%%%%%%%%%%%

%%%%%%%%%%%%%%%%%%%%%%%%%%%%%%%%%%%%%%%%%%%%%%%%%%%%%%%%%%%%%%%%%%%%%%%%%%%%%%
\subsection{Algebra Structure of $Sym_*$}
We may consider $Sym_* \stackrel{def}{=} \bigoplus_{p \geq 0}
Sym_*^{(p)}$ as a bigraded differential algebra, where $bideg(W) =
(p+1, i)$ for $W \in Sym_i^{(p)}$.  The product $\boxtimes \co
Sym_i^{(p)} \otimes Sym_j^{(q)} \to Sym_{i+j}^{(p+q+1)}$ is defined
by: $W \boxtimes V \stackrel{def}{=} W \otimes V'$, where $V'$ is
obtained from $V$ by replacing each formal indeterminate $z_r$ by
$z_{r+p+1}$ for $0 \leq r \leq q$.  Eq.~\ref{eq.face_i-tensor} then
implies:
\begin{equation}\label{eq.partialboxtimes}
  d( W \boxtimes V ) = d(W) \boxtimes V + (-1)^{bideg(W)_2}W \boxtimes d(V),
\end{equation}
where $bideg(W)_2$ is the second component of $bideg(W)$.
\begin{prop}\label{prop.boxtimes}
  The product $\boxtimes$ is well-defined on the level of homology.
  Furthermore, this product (on both the chain level and homology
  level) is skew commutative in a twisted sense: $W \boxtimes V =
  (-1)^{ij}\tau(V \boxtimes W)$, where $bideg(W) = (p+1, i)$,
  $bideg(V) = (q+1, j)$, and $\tau$ is the permutation sending $\{0,
  1, \ldots, q, q+1, q+2, \ldots, p+q, p+q+1\} \mapsto \{p+1, p+2,
  \ldots, p+q+1, 0, 1, \ldots, p-1, p \}$.
\end{prop}
\begin{proof}
  Eqn.~(\ref{eq.partialboxtimes}) implies the product passes to
  homology classes.  Now, suppose $W = Y_0 \otimes Y_1 \otimes \ldots
  \otimes Y_{p-i} \in Sym_i^{(p)}$ and $V = Z_0 \otimes Z_1 \otimes
  \ldots \otimes Z_{q-j} \in Sym_j^{(p)}$.
  \begin{equation}\label{eq.boxtimesformula}
    V \boxtimes W = V \otimes W' = (-1)^\alpha W' \otimes V,
  \end{equation}
  where $W'$ is related to $W$ by replacing each $z_r$ by $z_{r+q+1}$.  The 
  exponent $\alpha= deg(V)deg(W) = ij$ arises from the relations in 
  $Sym_{i+j}^{(p+q+1)}$.  (The fact that $deg(V) = i$ and $deg(W) = j$ may be 
  made clear by observing that the degree of a formal tensor product in 
  $Sym_*^{(s)}$ is equal to the number of {\it cut points}, that is, the number
  of places where a tensor product symbol may be inserted.)
  Next, apply the block transformation $\tau$ to Eq.~(\ref{eq.boxtimesformula})
  to obtain $\tau(V \boxtimes W) = (-1)^\alpha \tau(W' \otimes V) = (-1)^\alpha
  W \otimes V' = (-1)^\alpha W \boxtimes V$, where $V'$ is obtained by replacing
  $z_r$ by $z_{r+p+1}$ in $V$.  
\end{proof}

%%%%%%%%%%%%%%%%%%%%%%%%%%%%%%%%%%%%%%%%%%%%%%%%%%%%%%%%%%%%%%%%%%%%%%%%%%%%%%
\subsection{Computer Calculations}
In principle, the homology of $Sym_*^{(p)}$ may be found by using a computer.  
In fact, we have the following results up to $p = 7$:
\begin{theorem}\label{thm.poincare_sym_complex}
  For $0 \leq p \leq 7$, the groups $H_*(Sym_*^{(p)})$ are free
  abelian and have Poincar\'e polynomials $P_p(t) \stackrel{def}{=}
  P\left(H_*(Sym_*^{(p)}); t\right)$:
  \[
    P_0(t) = 1,
  \]
  \[
    P_1(t) = t,
  \]
  \[
    P_2(t) = t + 2t^2,
  \]
  \[
    P_3(t) = 7t^2 + 6t^3,
  \]
  \[
    P_4(t) = 43t^3 + 24t^4,
  \]
  \[
    P_5(t) = t^3 + 272t^4 + 120t^5,
  \]
  \[
    P_6(t) = 36t^4 + 1847t^5 + 720t^6,
  \]
  \[
    P_7(t) = 829t^5 + 13710t^6 + 5040t^7.
  \]
\end{theorem}
\begin{proof}
  These computations were performed using scripts written for the computer 
  algebra systems \verb|GAP|~\cite{GAP4} and \verb|Octave|~\cite{E}.
\end{proof}

We conjecture that the $H_*(Sym_*^{(p)})$ is always free abelian.

%%%%%%%%%%%%%%%%%%%%%%%%%%%%%%%%%%%%%%%%%%%%%%%%%%%%%%%%%%%%%%%%%%%%%%%%%%%%%%
\subsection{Representation Theory of $H_*(Sym_*^{(p)})$}\label{sec.rep_sym}
By remark~\ref{rmk.action}, the groups $H_i(Sym_*^{(p)}; k)$ carry the
structure of $k\Sigma_{p+1}$--modules, so it seems natural to
investigate the irreducible representations comprising these modules.
\begin{prop}
  Let $C_{p+1} \hookrightarrow \Sigma_{p+1}$ be the cyclic group of
  order $p+1$, embedded into the symmetric group as the subgroup
  generated by the permutation $\tau_p \stackrel{def}{=} (0, p, p-1,
  \ldots, 1)$.  Then there is a $\Sigma_{p+1}$-isomorphism:
  $H_p(Sym_*^{(p)}) \cong AC_{p+1} \uparrow \Sigma_{p+1}$, {\it i.e.},
  the alternating representation of the cyclic group, induced up to
  the symmetric group.  Note, for $p$ even, $AC_{p+1}$ coincides with
  the trivial representation $IC_{p+1}$.
  
  Moreover, $H_p(Sym_*^{(p)})$ is generated by the elements
  $\sigma(b_p)$, for the distinct cosets $\sigma C_{p+1}$, where $b_p
  \stackrel{def}{=} \sum_{j = 0}^p (-1)^{jp} \tau_p^j(z_0z_1 \ldots
  z_p)$.
\end{prop}
\begin{proof}
  Let $w$ be a general element of $Sym_p^{(p)}$, $w = \sum_{\sigma \in 
  \Sigma_{p+1}} c_\sigma \sigma(z_0z_1 \ldots z_p)$, where $c_\sigma$ are 
  constants in $k$.  $H_p(Sym_*^{(p)})$ consists of those $w$ such that 
  $d(w) = 0$.  That is,
  \begin{equation}\label{eq.sum_sigma_zero}
    0 = \sum_{\sigma \in \Sigma_{p+1}} \sum_{i=0}^{p-1}
    (-1)^i c_\sigma \sigma(z_0 \ldots z_i \otimes z_{i+1} \ldots z_p).
  \end{equation}
  Now for each $\sigma$, the terms corresponding to $\sigma(z_0 \ldots
  z_i \otimes z_{i+1} \ldots z_p)$ occur in pairs in the above
  formula.  The obvious term of the pair is $(-1)^i c_{\sigma}
  \sigma(z_0 \ldots z_i \otimes z_{i+1} \ldots z_p)$.  Not so
  obviously, the second term of the pair is
  $(-1)^{(p-i-1)i}(-1)^{p-i-1} c_{\rho} \rho(z_0 \ldots z_{p-i-1}
  \otimes z_{p-i} \ldots z_p)$, where $\rho = \sigma
  \tau_p^{p-i}$. Thus, if $d(w) = 0$, then we must have $(-1)^i
  c_\sigma + (-1)^{(p-i-1)(i+1)}c_\rho = 0$, or $c_\rho =
  (-1)^{(p-i)(i+1)}c_\sigma$.  Set $j = p-i$, so that $c_\rho =
  (-1)^{j(p-j+1)}c_\sigma = (-1)^{jp}c_\sigma$.  This shows that the
  only restrictions on the coefficients $c_\sigma$ are that the
  absolute values of coefficients corresponding to $\sigma, \sigma
  \tau_p, \sigma \tau_p^2, \ldots$ must be the same, and their
  corresponding signs in $w$ alternate if and only if $p$ is odd;
  otherwise, they have the same signs.  Clearly, the elements
  $\sigma(b_p)$ for distinct cosets $\sigma C_{p+1}$ represents an
  independent set of generators over $k$ for $H_p(Sym_*^{(p)})$.

  Observe that $b_p$ is invariant under the action of $sgn(\tau_p)\tau_p$, and 
  so $b_p$ generates an alternating representation $A C_{p+1}$ over $k$.  
  Induced up to $\Sigma_{p+1}$, we obtain the representation $AC_{p+1} \uparrow
  \Sigma_{p+1}$ of dimension $(p+1)!/(p+1) = p!$, generated by the elements 
  $\sigma(b_p)$ as in the proposition.
\end{proof}

\begin{definition}
  For a given proper partition $\lambda = [\lambda_0, \lambda_1, \lambda_2, 
  \ldots, \lambda_s]$ of the $p+1$ integers $\{0, 1, \ldots, p\}$, an element 
  $W$ of $Sym_*^{(p)}$ will designated as {\it type $\lambda$} if it equivalent
  to $\pm(Y_0 \otimes Y_1 \otimes Y_2 \otimes \ldots \otimes Y_s)$ with
  $deg(Y_i) = \lambda_i - 1$.  That is, each $Y_i$ has $\lambda_i$ factors.
  The notation $Sym_\lambda^{(p)}$ or $Sym_\lambda$ will denote the 
  $k$-submodule of $Sym_{p-s}^{(p)}$ generated by all elements of type 
  $\lambda$.
\end{definition}
In what follows, $|\lambda|$ will refer to the number of components of 
$\lambda$.  The action of $\Sigma_{p+1}$ leaves $Sym_\lambda$ invariant for any
given $\lambda$, so the there is a decomposition as $k\Sigma_{p+1}$--module:
\[
  Sym_{p-s}^{(p)} = \bigoplus_{\lambda \vdash (p+1), |\lambda| = s+1} 
  Sym_{\lambda}.
\]

\begin{prop}\label{prop.alt-and-trivial-reps-in-sym}
  For a given proper partition $\lambda \vdash (p+1)$, 
  
  {\it (a)} $Sym_\lambda$ contains exactly one
  alternating representation $A\Sigma_{p+1}$ iff $\lambda$ contains no repeated 
  components.
  
  {\it (b)} $Sym_\lambda$ contains exactly one trivial representation
  $I\Sigma_{p+1}$ iff $\lambda$ contains no repeated even components.
\end{prop}
\begin{proof}
  $Sym_\lambda$ is a quotient of the regular representation, since it is the 
  image of the $\Sigma_{p+1}$-map $\pi_\lambda \co k\Sigma_{p+1} \to 
  Sym_\lambda,$ $\sigma \mapsto \psi_\lambda s$, where $s \in 
  \Sigma_{p+1}^{\mathrm{op}}$ is the $\Delta S$-automorphism of $[p]$ 
  corresponding to $\sigma$ and $\psi_\lambda$ is a $\Delta$-morphism $[p] \to 
  [ \,|\lambda|\, ]$ that sends the points $0, \ldots, \lambda_0-1$ to $0$, the
  points $\lambda_0, \ldots, \lambda_0 + \lambda_1 -1$ to $1$, and so on.  
  Hence, there can be at most $1$ copy of $A\Sigma_{p+1}$ and at most $1$ copy
  of $I\Sigma_{p+1}$ in $Sym_\lambda$.
  
  Let $W$ be the ``standard'' element of $Sym_\lambda$.  That is, the 
  indeterminates $z_i$ occur in $W$ in numerical order and the degrees of
  monomials of $W$ are in decreasing order.  $A\Sigma_{p+1}$ exists
  in $Sym_\lambda$ iff the element $V = \sum_{\sigma \in \Sigma_{p+1}} 
  sgn(\sigma)\sigma(W)$ is non-zero.  Suppose that some component of $\lambda$
  is repeated, say $\lambda_i = \lambda_{i+1} = \ell$.  If $W = Y_0 \otimes Y_1
  \otimes \ldots \otimes Y_s$, then $deg(Y_i) = deg(Y_{i+1}) = \ell-1$.  Now, we
  know that $W = (-1)^{deg(Y_i)deg(Y_{i+1})} Y_0 \otimes \ldots \otimes Y_{i+1}
  \otimes Y_i \otimes \ldots Y_s  = -(-1)^{\ell} \alpha(W)$, for the permutation
  $\alpha \in \Sigma_{p+1}$ that exchanges the indices of indeterminates in
  $Y_i$ with those in $Y_{i+1}$ in an order-preserving way.  In $V$, the term 
  $\alpha(W)$ shows up with sign $sgn(\alpha) = (-1)^\ell$, thus cancelling 
  with $W$.  Hence, $V = 0$, and no alternating representation exists.

  If, on the other hand, no component of $\lambda$ is repeated, then
  no term $\pm \alpha(W)$ can be equivalent to $W$ for $\alpha \neq
  \mathrm{id}$, so $V$ survives as the generator of $A\Sigma_{p+1}$ in
  $Sym_\lambda$.
 
  A similar analysis applies for trivial representations.  This time, we examine
  $U = \sum_{\sigma \in \Sigma_{p+1}} \sigma(W)$, which would be a generator for
  $I\Sigma_{p+1}$ if it were non-zero.  As before, if there is a repeated 
  component, $\lambda_i=\lambda_{i+1} = \ell$, then $W=(-1)^{\ell-1}\alpha(W)$.
  However, this time, $W$ cancels with $\alpha(W)$ only if $\ell - 1$ is odd.  
  That is, $|\lambda_i| = |\lambda_{i+1}|$ is even.  If $\ell - 1$ is even, or 
  if all $\lambda_i$ are distinct, then the element $U$ must be non-zero.
\end{proof}
\begin{prop}\label{prop.alternating_reps}
  $H_i(Sym_*^{(p)})$ contains an alternating representation for each partition 
  $\lambda \vdash (p+1)$ with \mbox{$|\lambda| = p-i$} such that no component of
  $\lambda$ is repeated.
\end{prop}
\begin{proof}
  This proposition will follow from the fact that $d(V) = 0$ for any generator
  $V$ of an alternating representation in $Sym_\lambda$.  Then, by Schur's 
  Lemma, the alternating representation must survive at the level of homology.
  
  Let $V= \sum_{\sigma \in \Sigma_{p+1}} sgn(\sigma)\sigma(W)$ be the generator
  mentioned in Prop.~\ref{prop.alt-and-trivial-reps-in-sym}.  $d(V)$ consists of
  individual terms $d_j(\sigma(W)) = \sigma(d_j(W))$ along with appropriate 
  signs.  For a given $j$, $d_j(W)$ is identical to $W$ except at some monomial 
  $Y_i$, where a tensor product symbol is inserted.  We will
  introduce some notation to make the argument a little cleaner.  If 
  $Y = z_{i_0}z_{i_1} \ldots z_{i_r}$ is a monomial, then the notation $Y\{s, 
  \ldots, t\}$ refers to the monomial $z_{i_s}z_{i_{s+1}} \ldots z_{i_t}$, 
  assuming $0 \leq s \leq t \leq r$.  Now, we write
  \begin{equation}\label{eq.d_jW}
    d_j(W) = (-1)^{a + \ell} Y_0 \otimes \ldots \otimes Y_i\{0, \ldots, \ell\}
    \otimes Y_i\{\ell+1, \ldots, m\} \otimes \ldots \otimes Y_s,
  \end{equation}
  where $a = deg(Y_0) + \ldots + deg(Y_{i-1})$.  Use the relations in $Sym_*$ to
  rewrite Eq.~(\ref{eq.d_jW}):
  \begin{equation}\label{eq.d_jW2}
    (-1)^{(a + \ell) + \ell(m - \ell - 1)}Y_0 \otimes \ldots \otimes Y_i\{
    \ell+1, \ldots, m\}\otimes Y_i\{0,\ldots,\ell\} \otimes \ldots \otimes Y_s.
  \end{equation}
  Let $\alpha$ be the permutation that relabels indices in such a way that
  $Y_i\{0, \ldots, m-\ell-1\} \mapsto Y_i\{\ell+1, \ldots, m\}$ and $Y_i\{m-
  \ell, \ldots, m\} \mapsto Y_i\{0,\ldots, \ell\}$, so that the following
  is equivalent to Eq.~(\ref{eq.d_jW2}).
  \begin{equation}\label{eq.d_jW-alpha}
    (-1)^{a + m\ell - \ell^2}\alpha\left(Y_0 \otimes \ldots \otimes Y_i\{0, 
    \ldots, m-\ell-1\}\otimes Y_i\{m-\ell, \ldots, m\} \otimes \ldots \otimes 
    Y_s\right)
  \end{equation}
  Now, Eq.(~\ref{eq.d_jW-alpha}) also occurs in $d_{j'}\left(sgn(\alpha)
  \alpha(W) \right)$ for some $j'$.  This term looks like:
  \begin{equation}
    sgn(\alpha)(-1)^{a + m-\ell-1}
    \alpha\big(Y_0 \otimes \ldots \otimes 
    Y_i\{0, \ldots, m-\ell-1\}
    \otimes Y_i\{m-\ell, \ldots, m\} \otimes \ldots \otimes Y_s\big)
  \end{equation}
  \begin{equation}\label{eq.alpha-d_jprime}
    = (-1)^{m\ell - \ell^2 + a - 1}
    \alpha\big(Y_0 \otimes \ldots \otimes 
    Y_i\{0, \ldots, m-\ell-1\}
    \otimes Y_i\{m-\ell, \ldots, m\} \otimes \ldots \otimes Y_s\big)
  \end{equation}
  Comparing the signs of Eq.~(\ref{eq.alpha-d_jprime}) and 
  Eq.~(\ref{eq.d_jW-alpha}), we verify the two terms canel each other out
  in the sum $d(V)$.
\end{proof}

By Proposition~\ref{prop.alternating_reps}, it is clear that if $p+1$ is a 
triangular number -- {\it i.e.}, $p+1$ is of the form $r(r+1)/2$ for some 
positive integer $r$, then the lowest dimension in which an alternating 
representation may occur is $p + 1 - r$, corresponding to the partition 
$\lambda = [r, r-1, \ldots, 2, 1]$.  A little algebra yields the following 
statement for any $p$.
\begin{cor}\label{cor.lowest_alternating_reps}
  $H_i(Sym_*^{(p)})$ contains an alternating representation in degree $p+1-r$, 
  where $r = \lfloor \sqrt{2p + 9/4} - 1/2 \rfloor$.
  Moreover, there are no alternating representations present for $i \leq p-r$.
\end{cor}

There is not much known about the other irreducible representations
occurring in the homology groups of $Sym_*^{(p)}$, however
computational evidence shows that $H_i(Sym_*^{(p)})$ contains no
trivial representation, $I\Sigma_{p+1}$, for $i \leq p-r$ ($r$ as in
the conjecture above) up to $p = 50$.

%%%%%%%%%%%%%%%%%%%%%%%%%%%%%%%%%%%%%%%%%%%%%%%%%%%%%%%%%%%%%%%%%%%%%%%%%%%%%%
\subsection{Connectivity of $Sym_*^{(p)}$}\label{sub.connectivity_of_HS}
Quite recently, Vre\'cica and \v{Z}ivaljevi\'c~\cite{VZ} observed that the
complex $Sym_*^{(p)}$ is isomorphic to the suspension of the cycle-free 
chessboard complex $\Omega_{p+1}$ (in fact, the isomorphism takes the form 
$k\left[S\Omega_{p+1}^+\right] \to Sym_*^{(p)}$, where $\Omega_{p+1}^+$ is the 
augmented complex).

The $m$-chains of the complex $\Omega_n$ are generated by ordered
lists, $L = \{ (i_0, j_0)$, $(i_1, j_1)$, $\ldots, (i_m, j_m) \}$, where
$1 \leq i_0 < i_1 < \ldots < i_m \leq n$, all $1 \leq j_s \leq n$ are
distinct integers, and the list $L$ is {\it cycle-free}.  It may be
easier to say what it means for $L$ not to be cycle free: $L$ is not
cycle-free if there exists a subset $L_c \subseteq L$ and ordering of
$L_c$ so that $L_c = \{ (\ell_0, \ell_1), (\ell_1, \ell_2), \ldots,
(\ell_{t-1}, \ell_t), (\ell_t, \ell_0) \}$.  The face maps
are defined by deltion of elements of $L$:
\[
  d_s\left( \{ (i_0, j_0), \ldots, (i_m, j_m) \} \right)
  \stackrel{def}{=} \{ (i_0, j_0), \ldots,
  (i_{s-1}, j_{s-1}), (i_{s+1}, j_{s+1}), \ldots, (i_m, j_m) \}.
\]
For completeness, an explicit isomorphism shall be provided as part
of the proof of the following proposition.
\begin{prop}\label{prop.iso_omega_sym}
  Let $\Omega^+_n$ denote the augmented cycle-free $(n \times n)$-chessboard 
  complex, where the unique $(-1)$-chain is represented by the empty $n 
  \times n$ chessboard, and the boundary map on $0$-chains takes a vertex to the
  unique $(-1)$-chain.  For each $p \geq 0$, there is a chain isomorphism,
  $\omega_* \co k\left[S\Omega^+_{p+1}\right] \to Sym_*^{(p)}$.
\end{prop}
\begin{proof}
  Note that we may define the generating $m$-chains of
  $k\left[\Omega_{p+1} \right]$ as cycle-free lists $L$, with no
  requirement on the order of $L$, under the equivalence relation:
  $\sigma. L \stackrel{def}{=} \{ (i_{\sigma^{-1}(0)},
  j_{\sigma^{-1}(0)}), \ldots, (i_{\sigma^{-1}(m)},
  j_{\sigma^{-1}(m)}) \} \approx sgn(\sigma)L$, for $\sigma \in
  \Sigma_{m+1}$.  Suppose $L$ is an $(m+1)$-chain of $S\Omega^+_{p+1}$
  ({\it i.e.} an $m$-chain of $\Omega^+_{p+1}$).  Call a subset $L'
  \subseteq L$ a {\it queue} if there is a reordering of $L'$ such
  that $L' = \{(\ell_0, \ell_1), (\ell_1, \ell_2), \ldots,
  (\ell_{t-1}, \ell_t)\}$.  $L'$ is called a {\it maximal queue} if it
  is not properly contained in any other queue.  Since $L$ is supposed
  to be cycle-free, we can partition $L$ into some number of maximal
  queues, $L_1', L_2', \ldots, L_q'$.  Let $\sigma$ be a permutation
  representing the reordering of $L$ into maximal ordered queues.

   Now, each maximal ordered queue $L_i'$ will correspond to a monomial of 
   formal indeterminates $z_i$ as follows.
   \begin{equation}\label{eq.monomial_correspondence}
     L_s' = \{ (\ell_0, \ell_1), (\ell_1, \ell_2),
     \ldots, (\ell_{t-1}, \ell_{t}) \} \mapsto
     z_{\ell_0-1}z_{\ell_1-1}\cdots z_{\ell_{t}-1}.
   \end{equation}
   For each maximal ordered queue, $L'_s$, denote the monomial obtained by 
   formula~(\ref{eq.monomial_correspondence}) by $Z_s$.  Let $k_1, k_2, \ldots,
   k_u$ be the numbers in $\{0, 1, 2, \ldots, p\}$ such that $k_{r} + 1$ does 
   not appear in any pair $(i_s, j_s) \in L$.  Now we may define $\omega_*$ on 
   $L = L'_1 \cup L'_2 \cup \ldots \cup L'_q$.
   \begin{equation}\label{eq.omega-def}
     \omega_{m+1}(L) \stackrel{def}{=} Z_1 \otimes Z_2 \otimes \ldots
     \otimes Z_q \otimes z_{k_1} \otimes z_{k_2} \otimes \ldots
     \otimes z_{k_u}.
   \end{equation}
   Observe, if $L = \emptyset$ is the $(-1)$-chain of $\Omega^+_{p+1}$, then 
   there are no maximal queues in $L$, and so $\omega_0(\emptyset) = z_0 \otimes
   z_1 \otimes \ldots \otimes z_p$.

   $\omega_*$ is a (well-defined) chain map with inverse given by essentially
   reversing the process.  To each monomial $Z = z_{i_0}z_{i_1}\cdots z_{i_t}$ 
   with $t > 0$, there is an associated ordered queue $L' = \{ (i_0+1, i_1+1),
   (i_1+1, i_2+1), \ldots (i_{t-1} +1, i_t + 1) \}$.  If the monomial is a 
   singleton, $Z = z_{i_0}$, the associated ordered queue will be the empty set.
   Now, given a generator $Z_1 \otimes Z_2 \otimes \ldots \otimes Z_q \in 
   Sym_*^{(p)}$, map it to the list $L \stackrel{def}{=} L'_1 \cup L'_2 \cup \ldots \cup L'_q$,
   preserving the original order of indices.
\end{proof}
\begin{theorem}\label{thm.connectivity}
  $Sym_*^{(p)}$ is $\lfloor\frac{2}{3}(p-1)\rfloor$-connected.
\end{theorem}
\begin{proof}
  See Thm.~10 of~\cite{VZ}.
\end{proof}
This remarkable fact yields the following useful corollaries:
\begin{cor}\label{cor.finitely-generated}
  The spectral sequences of Thm.~\ref{thm.E1_NS} and Cor.~\ref{cor.E1_Sym} 
  converge strongly to $\widetilde{H}S_*(A)$.
\end{cor}
\begin{proof}
  The connectivity of the complexes $Sym_*^{(p)}$ is a non-decreasing
  function of $p$.  So for large enough $p$, we have
  $H_{n}\left(E_*G_{u} \ltimescirc_{G_u} Sym_*^{(p)}\right) = 0$.  For
  fixed $n$, the spectral sequence is guaranteed to collapse at a
  finite level.
\end{proof}
\begin{cor}\label{cor.trunc-isomorphism}
  For each $i \geq 0$, there is a positive integer $N_i$ so that if $p
  \geq N_i$, there is an isomorphism $\widetilde{H}S_i(A) \cong
  H_i\left(\mathscr{F}_p C_*(\mathrm{Epi}\Delta S,\,
  B_*^{sym}I\right))$.
\end{cor}
\begin{cor}\label{cor.fin-gen-restated}
  If $A$ is finitely-generated over a Noetherian ground ring $k$, then
  $HS_*(A)$ is finitely-generated over $k$ in each degree.
\end{cor}
The bounds on connectivity are conjectured to be tight.  This is
certainly true for $p \equiv 1$ (mod $3$), based on Thm.~16
of~\cite{VZ}.  Corollary~12 of the same paper establishes that either
$H_{2k}\left( Sym_*^{(3k-1)}\right) \neq 0$ or
$H_{2k}\left(Sym_*^{(3k)}\right) \neq 0$.  For $k \leq 2$, both
statements are true.  When the latter condition is true, this gives a
tight bound on connectivity for $p \equiv 0$ (mod $3$).  When the
former is true, there is not enough information for a tight bound,
since we are more interested in proving that
$H_{2k-1}\left(Sym_*^{(3k-1)}\right)$ is non-zero, since for $k = 1,
2$, we have computed the integral homology, $H_1\left(
Sym_*^{(2)}\right) = \mathbb{Z}$ and $H_3\left(Sym_*^{(5)}\right)
=\mathbb{Z}$.
  
%%%%%%%%%%%%%%%%%%%%%%%%%%%%%%%%%%%%%%%%%%%%%%%%%%%%%%%%%%%%%%%%%%%%%%%%%%%%%%%%
\section{A Partial Resolution}\label{sec.partres}                     
%%%%%%%%%%%%%%%%%%%%%%%%%%%%%%%%%%%%%%%%%%%%%%%%%%%%%%%%%%%%%%%%%%%%%%%%%%%%%%%%

As before, $k$ is a commutative ground ring.  In this section, we find
an explicit partial resolution of the trivial $\Delta
S^{\mathrm{op}}$--module $\underline{k}$ by projective modules,
allowing the computation of $HS_0(A)$ and $HS_1(A)$ for a unital
associative $k$-algebra $A$.  The resolution will be constructed
through a number of technical lemmas.

For any small category $\mathscr{C}$ and object $X \in
\mathrm{Obj}\mathscr{C}$, the $\mathscr{C}$--modules $k\left[
  \mathrm{Mor}_{\mathscr{C}}\left( X, -)\right)\right]$ are projective
(as $\mathscr{C}$--module), as are the
$\mathscr{C}^{\mathrm{op}}$--modules $k\left[
  \mathrm{Mor}_{\mathscr{C}}\left( -, X \right)\right]$.  In
particular, each $k\left[ \mathrm{Mor}_{\Delta S}\left( -,
  [q]\right)\right]$ is a projective $\Delta S^{\mathrm{op}}$--module.
In proving exactness, it suffices to examine the individual
sub-$k$--modules, $k\left[ \mathrm{Mor}_{\Delta S}\left( [n], [q]
  \right)\right]$.

\begin{lemma}\label{lem.0-stage}
  For each $n \geq 0$, the sequence, $0 \gets k
  \stackrel{\epsilon}{\gets} k\left[\mathrm{Mor}_{\Delta S}\left([n],
    [0]\right)\right]\stackrel{\rho} {\gets}
  k\big[\mathrm{Mor}_{\Delta S}\left([n], [2])\right]$ is exact, where
  $\epsilon$ is defined by $\epsilon(\phi) = 1$ for any morphism $\phi
  \co [n] \to [0]$, and $\rho$ is defined by $\rho(\psi) =
  (x_0x_1x_2)\circ \psi - (x_2x_1x_0)\circ\psi$ for any morphism $\psi
  \co [n] \to [2]$.  ($x_0x_1x_2$ and $x_2x_1x_0$ are $\Delta S$
  morphisms $[2] \to [0]$ written in tensor notation.)
\end{lemma}
\begin{proof}
  Clearly, $\epsilon$ is surjective.  Now, $\epsilon\rho = 0$, since 
  $\rho(\psi)$ consists of two morphisms with opposite signs.  Let $\phi_0 = x_0
  x_1 \ldots x_n \co [n] \to [0]$.  The kernel of $\epsilon$ is spanned by 
  elements $\phi - \phi_0$ for $\phi \in \mathrm{Mor}_{\Delta S}([n],[0])$.  So,
  it suffices to show that the submodule of $k\left[\mathrm{Mor}_{\Delta S}
  \left([n], [0]\right)\right]$ generated by $(x_0x_1x_2)\psi - (x_2x_1x_0)\psi$
  for $\psi \co [n] \to [2]$ contains all of the elements $\phi - \phi_0$.  In 
  other words, it suffices to find a sequence $\phi = \phi_k, \phi_{k-1},  
  \ldots, \phi_2, \phi_1, \phi_0$ so that each $\phi_i$ is obtained from 
  $\phi_{i+1}$ by reversing the order of 3 (possibly emtpy) blocks, $XYZ \to 
  ZYX$.  Let $\phi = x_{i_0}x_{i_1}\ldots x_{i_n}$.  If $\phi = \phi_0$, we may
  stop here.  Otherwise, we may produce a sequence ending in $\phi_0$ by way of
  a certain family of rearrangements:  
  
  \textit{$k$-rearrangment}: $x_{i_0}x_{i_1} \ldots
  x_{i_{k-1}}x_{i_k}x_{k+1} \ldots x_n \leadsto x_{k+1}\ldots x_n
  x_{i_k}x_{i_0}x_{i_1} \ldots x_{i_{k-1}}$, such that $i_k \neq k$.
  That is, a $k$-rearrangement only applies to those monomials that
  agree with $\phi_0$ in the final $n-k$ indeterminates, but not in
  the final $n-k+1$ indeterminates.  If $k = n$, then this
  rearrangement reduces to the cyclic rearrangment, $x_{i_0}x_{i_1}
  \ldots x_{i_n} \leadsto x_{i_n} x_{i_0}x_{i_1} \ldots x_{i_{n-1}}$.

  Beginning with $\phi$, perform $n$-rearrangements until the final 
  indeterminate is $x_n$.  For convenience of notation, let this new monomial be
  $x_{j_0}x_{j_1} \ldots x_{j_n}$.  (Of course, $j_n = n$.)  If $j_k = k$ for
  all $k = 0, 1, \ldots, n$, then we are done.  Otherwise, there will be a 
  number $k$ such that $j_k \neq k$ but $j_{k+1} = k + 1, \ldots, j_n = n$.   
  Perform a $k$-rearrangement followed by enough $n$-rearrangements so that
  the final indeterminate is again $x_n$. The net result of these rearrangements 
  is that the ending block $x_{k+1}x_{k+2}\ldots x_n$ remains fixed while the 
  beginning block $x_{j_0}x_{j_1}\ldots x_{j_{k}}$ becomes cyclically permuted
  to $x_{j_k}x_{j_0} \ldots x_{j_{k-1}}$.  It is clear that applying this 
  combination of rearrangements repeatedly will finally obtain a monomial 
  $x_{\ell_0}x_{\ell_1} \ldots x_{\ell_{k-1}} x_k x_{k+1} \ldots x_n$.  Now 
  repeat the process, until after a finite number of steps, we finally obtain 
  $\phi_0$.
\end{proof}
  
Let $\mathscr{B}_n \stackrel{def}{=} \{ x_{i_0}x_{i_1}\ldots
x_{i_{k-1}} \otimes x_{i_k} \otimes x_{k+1}x_{k+2} \ldots x_{n} \;|\;
1\leq k \leq n, i_k \neq k \}$.  $k[\mathscr{B}_n]$ is a free
submodule of $k\left[\mathrm{Mor}_{\Delta S}\left( [n],
  [2]\right)\right]$ of size $(n+1)! - 1$.

\begin{cor}\label{cor.B_n}
  When restricted to $k[\mathscr{B}_n]$, the map $\rho$ of 
  Lemma~\ref{lem.0-stage} is surjective onto the kernel of $\epsilon$.
\end{cor}
\begin{proof}
  In the proof of Lemma~\ref{lem.0-stage}, the $n$-rearrangements correspond to
  the image of elements $x_{i_0} \ldots x_{i_{n-1}} \otimes x_{i_n} \otimes 1$,
  with $i_n \neq n$.  For $k < n$, $k$-rearrangements correspond to the image of
  elements $x_{i_0} \ldots x_{i_{k-1}} \otimes x_{i_k} \otimes x_{k+1} \ldots 
  x_{n}$, with $i_k \neq k$.
\end{proof}
\begin{lemma}\label{lem.rank}
  $k\left[\mathrm{Mor}_{\Delta S}\left([n], [m]\right)\right]$ is a free 
  $k$--module of rank $(m+n+1)!/m!$.
\end{lemma}
\begin{proof}
  A morphism $\phi \co [n] \to [m]$ of $\Delta S$ is nothing more than an 
  assignment of $n+1$ objects into $m+1$ compartments, along with a total 
  ordering of the original $n+1$ objects, hence $\#\mathrm{Mor}_{\Delta S}
  \left([n], [m]\right) = \binom{m+n+1}{m}(n+1)! = \frac{(m+n+1)!}{m!}$.
\end{proof}
\begin{lemma}\label{lem.rho-iso}
  $\rho|_{k[\mathscr{B}_n]}$ is an isomorphism $k[\mathscr{B}_n] \cong 
  \mathrm{ker}\,\epsilon$.
\end{lemma}
\begin{proof}
  Since the rank of $k\left[\mathrm{Mor}_{\Delta S}\left([n], [0]\right)\right]$
  is $(n+1)!$, the rank of the $\mathrm{ker}\,\epsilon$ is $(n+1)! - 1$.  The
  isomorphism then follows from Corollary~\ref{cor.B_n}.
\end{proof}
\begin{lemma}\label{lem.4-term-relation}
  The relations of the form:
  \begin{equation}\label{eq.4-term}
    XY \otimes Z \otimes W + W \otimes ZX \otimes Y + YZX \otimes 1 \otimes W
    + W \otimes YZ \otimes X \approx 0
  \end{equation}
  \begin{equation}\label{eq.1-term}
    \qquad \mathrm{and} \qquad
    1 \otimes X \otimes 1 \approx 0
  \end{equation}
  collapse $k\big[\mathrm{Mor}_{\Delta S}([n], [2])\big]$ onto 
  $k[\mathscr{B}_n]$.
\end{lemma}
\begin{proof}
  This proof proceeds in multiple steps.
  
  {\bf Step 1. [Degeneracy Relations]} $X \otimes Y \otimes 1 \approx X \otimes
  1 \otimes Y \approx 1 \otimes X \otimes Y$.

  First, observe that letting $X = Y = W = 1$ in Eq.~(\ref{eq.4-term}) yields
  $Z \otimes 1 \otimes 1 \approx 0$, since $1 \otimes Z \otimes 1 \approx 0$.
  Then, letting $X = Z = W = 1$ in Eq.~(\ref{eq.4-term}) produces $1 \otimes 1 
  \otimes Y \approx 0$.  Thus, any formal tensor with two trivial factors
  is equivalent to $0$.    
  Next, let $Z = W = 1$ in Eq.~(\ref{eq.4-term}).  Then using the above 
  observation, we obtain $1 \otimes X \otimes Y + 1 \otimes Y \otimes X \approx 
  0$, that is, $1 \otimes X \otimes Y \approx -(1 \otimes Y \otimes X)$.  Then, 
  if we let $X = W = 1$, we obtain $Y \otimes Z \otimes 1 + 1 \otimes Z 
  \otimes Y \approx 0$, which is equivalent to $Y \otimes Z \otimes 1 - 1 
  \otimes Y \otimes Z \approx 0$. 
  Finally, let $X = Y = 1$ in Eq.~(\ref{eq.4-term}).  The expression reduces
  to $Z \otimes 1 \otimes W - 1 \otimes Z \otimes W \approx 0$.
  
  {\bf Step 2. [Sign Relation]} $X \otimes Y \otimes Z \approx -(Z \otimes Y 
  \otimes X)$. 

  Let $Y = 1$ in Eq.~(\ref{eq.4-term}), and use the degeneracy relations
  to rewrite the result as $X \otimes Z \otimes W + 1 \otimes W \otimes ZX + 
  1 \otimes ZX \otimes W + W \otimes Z \otimes X \approx 0$.  Since
  $1 \otimes ZX \otimes W \approx -(1 \otimes W \otimes ZW)$, the desired
  result follows: $X \otimes Z \otimes W + W \otimes Z \otimes X \approx 0$.
  
  {\bf Step 3. [Hochschild Relation]} $XY \otimes Z \otimes 1 - X \otimes YZ 
  \otimes 1 + ZX \otimes Y \otimes 1 \approx 0$.  This relation is named
  after the similar relation, $XY \otimes Z - X \otimes YZ + ZX \otimes Y$,
  that arises in the Hochschild complex.
  
  Let $W = 1$ in Eq.~(\ref{eq.4-term}), and use the degeneracy and sign
  relations to obtain the desired result.

  {\bf Step 4. [Cyclic Relation]} $\ds{\sum_{j = 0}^n \tau_n^j\left(x_{i_0}
  x_{i_1} \ldots x_{i_{n-1}} \otimes x_{i_n} \otimes 1\right) \approx 0}$,
  where $\tau_n \in \Sigma_{n+1}$ is the $(n+1)$-cycle $(0,n,n-1,\ldots,2,1)$,
  which acts by permuting the indices.  
  
  For $n = 0$, there are no such relations
  (indeed, no relations at all).  For $n=1$, the cyclic relation takes the form 
  $x_0 \otimes x_1 \otimes 1 + x_1 \otimes x_0 \otimes 1 \approx 0$, which 
  follows from degeneracy and sign relations.
  
  Assume now that $n \geq 2$.  For each $k = 1, 2, \ldots, n-1$, define:
  \[
    \left\{\begin{array}{ll}
      A_k \stackrel{def}{=} & x_{i_0}x_{i_1}\ldots x_{i_{k-1}},\\
      B_k \stackrel{def}{=} & x_{i_k},\\
      C_k \stackrel{def}{=} & x_{i_{k+1}} \ldots x_{i_n}.
    \end{array}\right.
  \]
  By the Hochschild relation, $0 \approx \sum_{k=1}^{n-1} 
  (A_kB_k \otimes C_k \otimes 1 - A_k \otimes B_kC_k \otimes 1 + C_kA_k \otimes
  B_k \otimes 1)$.  But for $k \leq n-2$, $A_kB_k \otimes C_k \otimes 1 = 
  A_{k+1} \otimes B_{k+1}C_{k+1} \otimes 1$.
  Thus, after some cancellation, 
  \begin{equation}\label{eq.cyclic-relation}
    0 \approx - A_1 \otimes B_1 C_1 \otimes 1 + A_{n-1}B_{n-1} \otimes C_{n-1}
    \otimes 1 + \sum_{k=1}^{n-1} C_kA_k \otimes B_k \otimes 1.
  \end{equation}
  Now observe that sign and degeneracy relations imply
  that $- A_1 \otimes B_1 C_1 \otimes 1 \approx x_{i_1} \ldots x_{i_n} \otimes 
  x_{i_0} \otimes 1$.  The term $A_{n-1}B_{n-1} \otimes C_{n-1} \otimes 1$ is
  equal to $x_{i_0} \ldots x_{i_{n-1}} \otimes x_{i_n} \otimes 1$, and for
  $1 \leq k \leq n-1$, $C_kA_k \otimes B_k \otimes 1 = x_{i_{k+1}} \ldots 
  x_{i_n}x_{i_0} \ldots x_{i_{k-1}} \otimes x_{i_k} \otimes 1$.  Thus,
  Eq.~(\ref{eq.cyclic-relation}) can be rewritten as the cyclic relation,
  \[
    0 \approx (x_{i_0} \ldots x_{i_{n-1}} \otimes x_{i_n} \otimes 1)
    + \sum_{k=0}^{n-1} x_{i_{k+1}} \ldots x_{i_n}x_{i_0} \ldots x_{i_{k-1}} 
    \otimes x_{i_k} \otimes 1.
  \]

  {\bf Step 5.} Every element of the form $X \otimes Y \otimes 1$ is equivalent
  to a linear combination of elements of $\mathscr{B}_n$.

  To prove this, we shall induct on the size of $Y$.  Suppose $Y$ consists of
  a single indeterminate.  That is, $X \otimes Y \otimes 1 =
  x_{i_0} \ldots x_{i_{n-1}} \otimes x_{i_n} \otimes 1$.  Now, if $i_n \neq n$,
  we are done.  Otherwise, we use the cyclic relation to write
  $x_{i_0} \ldots x_{i_{n-1}} \otimes x_{i_n} \otimes 1 \approx -\sum_{j = 1}^n
  \tau_n^j\left(x_{i_0} \ldots x_{i_{n-1}} \otimes x_{i_n} \otimes 1\right)$.

  Now suppose $k \geq 1$ and any element $Z \otimes W \otimes 1$ with $|W| = k$
  is equivalent to an element of $k[\mathscr{B}_n]$.  Consider $X \otimes Y 
  \otimes 1 = x_{i_0} \ldots x_{i_{n-k-1}} \otimes x_{i_{n-k}} \ldots x_{i_n} 
  \otimes 1$.
  Let
  \[
    \left\{\begin{array}{ll}
      A = & x_{i_0}x_{i_1}\ldots x_{i_{n-k-1}},\\
      B = & x_{i_{n-k}} \ldots x_{i_{n-1}},\\
      C = & x_{i_n}.
    \end{array}\right.
  \]
  Then, by the Hochschild relation, $X \otimes Y \otimes 1 = A \otimes BC 
  \otimes 1 \approx AB \otimes C \otimes 1 + CA \otimes B \otimes 1$.
  But since $C$ has one indeterminate and $B$ has $k$ indeterminates, this last
  expression is equivalent to an element of $k[\mathscr{B}_n]$ by inductive
  hypothesis.
  
  {\bf Step 6. [Modified Hochschild Relation]} $XY \otimes Z \otimes W - X 
  \otimes YZ \otimes W + ZX \otimes Y \otimes W \approx 0$, modulo $k\left[
  \mathscr{B}_n\right]$.
  
  First, we show that $X \otimes Y \otimes W + Y \otimes X \otimes W \approx 0
  \pmod{k\left[\mathscr{B}_n\right]}$.  Indeed, if we let $Z = 1$ in 
  Eq.~(\ref{eq.4-term}), then sign and degeneracy relations yield:
  $ X \otimes Y \otimes W + Y \otimes X \otimes W \approx XY \otimes W \otimes
  1 + YX \otimes W \otimes 1$, {\it i.e.}, by step 5,
  \begin{equation}\label{eq.step8a}
    X \otimes Y \otimes W \approx -(Y \otimes X \otimes W)
    \pmod{k\left[\mathscr{B}_n\right]}.
  \end{equation}
  Now, using sign and degeneracy relations, Eq.~(\ref{eq.4-term}) can be 
  re-expressed:
  \[
    XY \otimes Z \otimes W - Y \otimes ZX \otimes W
    + YZX \otimes W \otimes 1 - X \otimes YZ \otimes W \approx 0.
  \]
  Using Eq.~(\ref{eq.step8a}), we then arrive at the modified Hochschild 
  relation,
  \begin{equation}\label{eq.step8b}
    XY \otimes Z \otimes W - X \otimes YZ \otimes W + ZX \otimes Y \otimes W 
    \approx 0 \pmod{k\left[\mathscr{B}_n\right]}.
  \end{equation}
  
  {\bf Step 7. [Modified Cyclic Relation]} $\ds{\sum_{j = 0}^k \tau_k^j\left(
  x_{i_0}x_{i_1} \ldots x_{i_{k-1}} \otimes x_{i_k} \otimes x_{i_{k+1}}\ldots
  x_{i_n}\right)} \approx 0$, modulo $k\left[\mathscr{B}_n\right]$.  Note, the
  $(k+1)$-cycle $\tau_k$ permutes the indices $i_0, i_1, \ldots, i_k$, and fixes
  the rest.
  
  Eq.~(\ref{eq.step8a}) proves the modified cyclic relations for $k=1$.
  The modified cyclic relation for $k \geq 2$ follows from the modified 
  Hochschild relation in the same manner as in step 4.  Of course, this time
  all equivalences are taken modulo $k\left[\mathscr{B}_n\right]$.
  
  {\bf Step 8.}  Every element of the form $X \otimes Y \otimes x_n$ is 
  equivalent to an element of $k[\mathscr{B}_n]$.
  
  We shall use the modified cyclic and modified Hochschild relations in a 
  similar way as cyclic and Hochschild relations were used in step 5.  Again we
  induct on the size of $Y$.  If $|Y| = 1$, then $X \otimes Y \otimes x_n =
  x_{i_0} \ldots x_{i_{n-2}} \otimes x_{i_{n-1}} \otimes x_{n}$.  If $i_{n-1} 
  \neq n-1$, then we are done.  Otherwise, use the modified cyclic relation to 
  re-express $X \otimes Y \otimes x_n$ as a sum of elements of $k\left[
  \mathscr{B}_n\right]$.
  
  Next, suppose $k \geq 1$ and any element $Z \otimes W \otimes x_n$ with $|W| =
  k$ is equivalent to an element of $k[\mathscr{B}_n]$.  Consider an element
  $X \otimes Y \otimes x_n$ with $|Y| = k + 1$.  Write $Y = BC$ with
  $|B| = k$ and $|C| = 1$ and use the modified Hochschild relation to rewrite
  $X \otimes BC \otimes x_n$ in terms of two elements whose middle tensor 
  factors are either $B$ or $C$ (modulo $k\left[\mathscr{B}_n\right]$).  By 
  inductive hypothesis, the rewritten expression must lie in $k[\mathscr{B}_n]$.
  
  {\bf Step 9.} Every element of $k\left[\mathrm{Mor}_{\Delta S}\left([n], [2]
  \right)\right]$ is equivalent to a linear combination of elements from the 
  following set:
  \begin{equation}\label{eq.step10}
    \mathscr{C}_n \stackrel{def}{=} \{X \otimes x_{i_n} \otimes 1 \;|\; i_n \neq n\} \cup 
    \{X \otimes x_{i_{n-1}} \otimes x_n \;|\; i_{n-1} \neq n-1\} \cup
    \{X \otimes Y \otimes Zx_n \;|\; |Z| \geq 1\}
  \end{equation}
  Note, the $k$--module generated by $\mathscr{C}_n$ contains $k\left[
  \mathscr{B}_n\right]$.
  
  Let $X \otimes Y \otimes Z$ be an arbitrary element of $k\left[
  \mathrm{Mor}_{\Delta S}\left([n], [2]\right)\right]$.  If $|X| = 0$, $|Y|=0$, 
  or $|Z| = 0$, then the degeneracy relations and step 5 imply that $X \otimes Y
  \otimes Z$ is equivalent to an element of $k\left[\mathscr{B}_n\right]$.  
  
  Suppose now that $|X|, |Y|, |Z| \geq 1$.  If $x_n$ occurs in $X$, use the 
  relation $X \otimes Y \otimes W \approx -(Y \otimes X \otimes W) \pmod{k\left[
  \mathscr{B}_n\right]}$ to ensure that $x_n$ occurs in the middle factor.  If 
  $x_n$ occurs in $Z$, use the sign relation and the above relation to put 
  $x_n$ into the middle factor.  In any case, it suffices to assume our element
  has the form: $X \otimes Ux_nV \otimes Z$.  Using the modified Hochschild
  relation, $X \otimes Ux_nV \otimes Z$ $\approx -(Z \otimes V \otimes XUx_n) + 
  Z \otimes VX \otimes Ux_n,$ $\pmod{k\left[\mathscr{B}_n\right]}$.    The first
  term is certainly in $k[\mathscr{C}_n]$, since $|X| \geq 1$.  If $|U| > 0$, 
  the second term also lies in $k[\mathscr{C}_n]$.  If, on the other hand, $|U|
  = 0$, then step 8 implies that $Z \otimes VX \otimes x_n$ is an element 
  of $k[\mathscr{B}_n]$.
  
  Observe that Step 9 proves Lemma~\ref{lem.4-term-relation} for $n = 0, 1, 2$,
  since in these cases, any elements that fall within the set
  $\{X \otimes Y \otimes Zx_n \;|\; |Z| \geq 1\}$ must have either $|X| = 0$
  or $|Y| = 0$, hence are equivalent via the degeneracy relation to elements of
  $k\left[\{X \otimes x_{i_n} \otimes 1 \;|\; i_n \neq n\}\right]$.  
  In what follows, assume $n \geq 3$.

  {\bf Step 10.} Every element of $k\left[\mathrm{Mor}_{\Delta S}\left([n], [2]
  \right)\right]$ is equivalent, modulo $k\left[\mathscr{B}_n\right]$, to a 
  linear combination of elements from the following set:
  \begin{equation}\label{eq.step10a}
    \mathscr{D}_n \stackrel{def}{=} \{X \otimes x_{i_{n-2}} \otimes x_{n-1}x_{n} \;|\; i_{n-2}
    \neq n-2\} \cup \{X \otimes Y \otimes Zx_{n-1}x_n \;|\; |Z| \geq 1\}.
  \end{equation}
  
  First, we require a relation that transports $x_n$ from the end of a tensor: 
  \begin{equation}\label{eq.step10b}
    W \otimes Z \otimes Xx_n \approx W \otimes x_nZ \otimes X
    \pmod{k\left[\mathscr{B}_n\right]}.
  \end{equation}
  Letting $Y = x_n$ in Eq.~(\ref{eq.4-term}), and making use of the sign
  relation, we have: $W \otimes Z \otimes Xx_n \approx W \otimes ZX \otimes x_n
  + x_nZX \otimes W \otimes 1 + W \otimes x_nZ \otimes X$.  By steps 5 and 8, 
  $W \otimes Z \otimes Xx_n \approx W \otimes x_nZ \otimes X$, modulo elements 
  of $k\left[\mathscr{B}_n\right]$.

  Now, let $X \otimes Y \otimes Z$ be an arbitrary element of $k\left[
  \mathrm{Mor}_{\Delta S}\left([n], [2]\right)\right]$.  Locate $x_{n-1}$ and 
  use the techniques of Step 9 to re-express $X \otimes Y \otimes Z$ as a linear
  combination of terms of the form: $X_j \otimes Y_j \otimes Z_jx_{n-1}$,
  modulo $k\left[\mathscr{B}_n\right]$.  Our goal is to re-express each term
  as a linear combination of vectors in which $x_n$ occurs only in the second
  tensor factor.
  
  If $x_n$ occurs in $X_j$, then observe $X_j \otimes Y_j \otimes Z_jx_{n-1} 
  \approx -(Y_j \otimes X_j \otimes Z_jx_{n-1})$, $\pmod{k\left[\mathscr{B}_n
  \right]}$.

  If $x_n$ occurs in $Z_j$, then first substitute $Y = x_{n-1}$ into 
  Eq.~(\ref{eq.4-term}), obtaining the relation:
  \[
    Xx_{n-1} \otimes Z \otimes W + W \otimes ZX \otimes x_{n-1}
    + x_{n-1}ZX \otimes 1 \otimes W + W \otimes x_{n-1}Z \otimes X \approx 0
  \]
  \[
    \Rightarrow \; W \otimes Z \otimes Xx_{n-1} \approx W \otimes ZX \otimes 
    x_{n-1} + W \otimes x_{n-1}Z \otimes X \pmod{k\left[\mathscr{B}_n\right]}
  \]
  By the modified Hochschild relation, $W \otimes x_{n-1}Z \otimes X \approx
  Wx_{n-1} \otimes Z \otimes X + ZW \otimes x_{n-1} \otimes X$, $\pmod{k
  \left[\mathscr{B}_n\right]}$, then using sign relations etc., we obtain:
  \[
    W \otimes Z \otimes Xx_{n-1} \approx W \otimes ZX \otimes x_{n-1}
    + Z \otimes X \otimes Wx_{n-1} - ZW \otimes X \otimes x_{n-1},
    \pmod{k\left[\mathscr{B}_n\right]}.
  \]
  Thus, we can express our original element $X \otimes Y \otimes Z$ as a linear
  combination of elements of the form $X' \otimes U'x_nV' \otimes Z'x_{n-1}$,
  $\pmod{k\left[\mathscr{B}_n\right]}$.  Then using modified Hochschild etc.,
  rewrite each such term as follows:
  \[
    X' \otimes U'x_nV' \otimes Z'x_{n-1} \approx X'U' \otimes x_nV' \otimes 
    Z'x_{n-1} - U' \otimes x_nV'X' \otimes Z'x_{n-1}, \pmod{k\left[
    \mathscr{B}_n\right]}.
  \]
  By Eq.~(\ref{eq.step10b}), we transport $x_n$ to the end of each term, so
  $X' \otimes U'x_nV' \otimes Z'x_{n-1} \approx X'U' \otimes V' \otimes Z'
  x_{n-1}x_n - U' \otimes V'X' \otimes Z'x_{n-1}x_n$, modulo elements of
  $k\left[\mathscr{B}_n\right]$.  If $|Z'| \geq 1$, then we are done. Otherwise,
  we have some elements of the form $X'' \otimes Y'' \otimes x_{n-1}x_n$.  Use
  an induction argument analogous to that in step 8 to re-express this type of
  element as a linear combination of elements of the form $U \otimes x_{i_{n-2}}
  \otimes x_{n-1}x_n$ such that $i_{n-2} \neq n-2$, modulo elements of $k\left[
  \mathscr{B}_n\right]$.
  
  {\bf Step 11.}  Every element of $k\left[\mathrm{Mor}_{\Delta S}\left([n], 
  [2]\right)\right]$ is equivalent to an element of $k\left[\mathscr{B}_n
  \right]$.
    
  We shall use an iterative re-writing procedure.  First of all, define sets:
  \[
    \mathscr{B}_n^{j} \stackrel{def}{=} \{ A \otimes x_{i_{n-j}} \otimes x_{n-j+1} \ldots
    x_n \;|\; i_{n-j} \neq n-j\},
  \]
  \[
    \mathscr{C}_n^{j} \stackrel{def}{=} \{ A \otimes B \otimes Cx_{n-j+1} \ldots
    x_n \;|\; |C| \geq 1\}.
  \]
  Now clearly, $\mathscr{B}_n  = \bigcup_{j=0}^{n-1} \mathscr{B}_n^{j}$.
  In what follows, `reduced' will always mean reduced modulo elements of
  $k\left[\mathscr{B}_n\right]$.  By steps 9 and 10, we can reduce an arbitrary
  element $X \otimes Y \otimes Z$ to linear combinations of elements in 
  $\mathscr{B}_n^0 \cup \mathscr{B}_n^1 \cup \mathscr{B}_n^2 \cup 
  \mathscr{C}_n^2$.  Suppose now that we have reduced  elements to linear 
  combinations of elements from the set $\mathscr{B}_n^0 \cup \mathscr{B}_n^1 
  \cup \ldots \cup \mathscr{B}_n^j \cup \mathscr{C}_n^j$, for some $j \geq 2$.
  I claim any element of $\mathscr{C}_n^j$ can be re-expressed as a linear 
  combination of elements from the set $\mathscr{B}_n^0 \cup \mathscr{B}_n^1 
  \cup \ldots \cup \mathscr{B}_n^{j+1} \cup \mathscr{C}_n^{j+1}$.  Indeed,
  let $X \otimes Y \otimes Zx_{n-j+1} \ldots x_n$, with $|Z| \geq 1$.  Let
  $w \stackrel{def}{=} x_{n-j+1} \ldots x_n$.  We may now think of $X \otimes Y \otimes Zw$
  as consisting of the `indeterminates' $x_0, x_1, \ldots, x_{n-j}, w$, hence, 
  by step 10, we may reduce this element to a linear combination of elements
  from the set $\{ X \otimes x_{i_{n-j-1}} \otimes x_{n-j}w \;|\; i_{n-j-1} 
  \neq n-j-1\} \cup \{ X \otimes Y \otimes Zx_{n-j}w \;|\; |Z| \geq 1 \}$.
  This implies the element may written as a linear combination
  of elements from the set $\mathscr{B}_n^{j+1} \cup \mathscr{C}_n^{j+1}$, 
  modulo elements of the form $A \otimes B \otimes 1$ and $A \otimes B \otimes 
  x_{n-j+1}\ldots x_n$.  Since $\{A \otimes B \otimes x_{n-j+1}x_{n-j+2}\ldots
  x_n\} \subseteq \mathscr{C}_n^{j-1}$, the inductive hypothesis ensures that
  the there is set containment $\{A \otimes B \otimes x_{n-j+1}\ldots
  x_n\} \subseteq \mathscr{B}_n^0 \cup \ldots \cup \mathscr{B}_n^{j}$.  This
  completes the inductive step.

  After a finite number of iterations, then, we can re-express any element
  $X \otimes Y \otimes Z$ as a linear combination from the set
  $\mathscr{B}_n^0 \cup \ldots \mathscr{B}_n^{n-1} \cup \mathscr{C}_n^{n-1}
  = \mathscr{B}_n \cup \mathscr{C}_n^{n-1}$.  But $\mathscr{C}_n^{n-1} = 
  \{ A \otimes B \otimes Cx_{2} \ldots x_n \;|\; |C| \geq 1\}$.  Any element 
  from this set has either $|A| = 0$ or $|B| = 0$, therefore is equivalent to an
  element of $k[\mathscr{B}_n]$ already.  
\end{proof}
\begin{cor}\label{cor.k-contains-one-half}
  If $\frac{1}{2} \in k$, then the four-term relation $XY \otimes Z \otimes W +
  W \otimes ZX \otimes Y + YZX \otimes 1 \otimes W + W \otimes YZ \otimes X 
  \approx 0$ is sufficient to collapse $k\left[\mathrm{Mor}_{\Delta S}\left([n],
  [2]\right)\right]$ onto $k\left[\mathscr{B}_n\right]$.
\end{cor}
\begin{proof}
  We only need to modify step 1 of the previous proof.  We will establish
  that $X \otimes 1 \otimes 1 \approx 1 \otimes X \otimes 1 \approx 1 \otimes 1
  \otimes X \approx 0$.
  
  Setting three variables at a time equal to $1$ in Eq.~(\ref{eq.4-term})
  we obtain,
  \begin{equation}\label{eq.step1_1}
    2(W \otimes 1 \otimes 1) + 2(1 \otimes 1 \otimes W) \approx 0,
    \quad \textrm{when $X = Y = Z = 1$.}
  \end{equation}
  \begin{equation}\label{eq.step1_2}
    Z \otimes 1 \otimes 1 + 3(1 \otimes Z \otimes 1) \approx 0,
    \quad \textrm{when $X = Y = W = 1$.}
  \end{equation}
  \begin{equation}\label{eq.step1_3}
    2(Y \otimes 1 \otimes 1) + 1 \otimes Y \otimes 1 + 1 \otimes 1 \otimes Y
    \approx 0, \quad \textrm{when $X = Z = W = 1$.}
  \end{equation}
  Equivalently, we have a system of linear equations,
  \[
    \left[\begin{array}{ccc}
       2 & 0 & 2 \\
       1 & 3 & 0 \\
       2 & 1 & 1
    \end{array}\right] 
    \left[ \begin{array}{c}
       z_1 \\
       z_2 \\
       z_3
    \end{array}\right]
    = 0,
  \]
  where $z_1 = X \otimes 1 \otimes 1$, $z_2 = 1 \otimes X \otimes 1$, and
  $z_3 = 1 \otimes 1 \otimes X$.  Since the determinant of the coefficient
  matrix is $-4$, the matrix is invertible in the ring $k$ as long as $1/2 \in 
  k$.
\end{proof}
  
Define for each $m \geq 0$, $P_m \stackrel{def}{=} k\left[
  \mathrm{Mor}_{\Delta S}\left(-, [m]\right)\right]$ and $P^n_m =
P_m([n])$.  Lemma~\ref{lem.4-term-relation} together with
Lemmas~\ref{lem.rho-iso} and~\ref{lem.0-stage} show the following
sequence of $k$--modules is exact for each $n \geq 0$:
\begin{equation}\label{eq.part_res-n}
  0 \gets k \stackrel{\epsilon}{\gets} P^n_0 \stackrel{\rho}{\gets}
  P^n_2 \stackrel{(\alpha, \beta)}{\longleftarrow} P^n_3 \oplus P^n_0
\end{equation}
where $\alpha \co P^n_3 
\to P^n_2$ is given by
composition with the $\Delta S$ morphism, $x_0x_1 \otimes x_2 \otimes x_3 + x_3
\otimes x_2x_0 \otimes x_1 + x_1x_2x_0 \otimes 1 \otimes x_3 + x_3 \otimes 
x_1x_2 \otimes x_0$, and $\beta \co P^n_0 \to P^n_2$ 
is induced by $1 \otimes x_0 \otimes 1$.  This holds for all $n \geq 0$, so
we have constructed a partial resolution of $\underline{k}$ by projective 
$\Delta S^\mathrm{op}$--modules:
\begin{equation}\label{eq.part_res}
  0 \gets k \stackrel{\epsilon}{\gets} P_0 \stackrel{\rho}{\gets}
  P_2 \stackrel{(\alpha, \beta)}{\longleftarrow} P_3 \oplus P_0
\end{equation}

%%%%%%%%%%%%%%%%%%%%%%%%%%%%%%%%%%%%%%%%%%%%%%%%%%%%%%%%%%%%%%%%%%%%%%%%%%%%%%%%
\section{Using the Partial Resolution for Low Degree Computations}
\label{sec.lowdeg}                     
%%%%%%%%%%%%%%%%%%%%%%%%%%%%%%%%%%%%%%%%%%%%%%%%%%%%%%%%%%%%%%%%%%%%%%%%%%%%%%%%

Let $A$ be a unital associative algebra over $k$.  Since Eq.~(\ref{eq.part_res})
is a partial resolution of $\underline{k}$, it can be used to find $HS_i(A)$
for $i = 0, 1$.

%%%%%%%%%%%%%%%%%%%%%%%%%%%%%%%%%%%%%%%%%%%%%%%%%%%%%%%%%%%%%%%%%%%%%%%%%%%%%%
\subsection{Main Theorem}

\begin{theorem}\label{thm.partial_resolution-restated}
  $HS_i(A)$ for $i=0,1$ may be computed as the degree $0$ and degree $1$ 
  homology groups of the following (partial) chain complex:
  \begin{equation}\label{eq.partial_complex}
    0\longleftarrow A \stackrel{\partial_1}{\longleftarrow} A\otimes A\otimes A
    \stackrel{\partial_2}{\longleftarrow}(A\otimes A\otimes A\otimes A)\oplus A,
  \end{equation}
  where 
  \[
    \partial_1 \co a\otimes b\otimes c \mapsto abc - cba,
  \]
  \[
    \partial_2 \co \left\{\begin{array}{lll}
         a\otimes b\otimes c\otimes d &\mapsto& ab\otimes c\otimes d + 
         d\otimes ca\otimes b \\
         && \quad + bca\otimes 1\otimes d + d\otimes bc\otimes a,\\
         a &\mapsto& 1\otimes a\otimes 1.
       \end{array}\right.
  \]
\end{theorem}
\begin{proof}
  Tensoring the complex~(\ref{eq.part_res}) with $B_*^{sym}A$ over $\Delta S$,
  we obtain a complex that computes $HS_0(A)$ and $HS_1(A)$.  The statement
  of the theorem then follows from isomorphisms induced by the evaluation
  map, $k\left[\mathrm{Mor}_{\Delta S}\left(-, [p]\right)\right] 
  \otimes_{\Delta S} B_*^{sym}A \stackrel{\cong}{\longrightarrow} B_p^{sym}A$.
\end{proof}

%%%%%%%%%%%%%%%%%%%%%%%%%%%%%%%%%%%%%%%%%%%%%%%%%%%%%%%%%%%%%%%%%%%%%%%%%%%%%%
\subsection{Degree $0$ Symmetric Homology}

\begin{theorem}\label{thm.HS_0}
  For a unital associative algebra $A$ over commutative ground ring $k$,
  $HS_0(A) \cong A/([A,A])$, where $([A,A])$ is the ideal generated by the 
  commutator submodule $[A,A]$.
\end{theorem}
\begin{proof}
  By Thm.~\ref{thm.partial_resolution}, $HS_0(A) \cong A/k\left[\{abc-cba\}
  \right]$
  as $k$--module.  But $k\big[\{abc-cba\}\big]$ is an ideal of $A$.  Now
  clearly $[A,A] \subseteq k\left[\{abc-bca\}\right]$.  On the other hand,
  $k\left[\{abc-cba\}\right] \subseteq ([A,A])$ since $abc - cba = a(bc-cb) + 
  a(cb) - (cb)a$.
\end{proof}
\begin{cor}
  If $A$ is commutative, then $HS_0(A) \cong A$.
\end{cor}

\begin{rmk} 
  Theorem~\ref{thm.HS_0} implies that symmetric homology does not preserve 
  Morita equivalence, since for $n>1$, $HS_0\left(M_n(A)\right) = M_n(A)/\left(
  [M_n(A),M_n(A)]\right) = 0$, while in general $HS_0(A) = A/([A,A]) \neq 0$.
\end{rmk}

%%%%%%%%%%%%%%%%%%%%%%%%%%%%%%%%%%%%%%%%%%%%%%%%%%%%%%%%%%%%%%%%%%%%%%%%%%%%%%
\subsection{Degree $1$ Symmetric Homology}\label{sub.deg-1-HS}

Using \verb|GAP|, we have made the following explicit computations of degree 1
integral symmetric homology.  See section~\ref{sec.gap-fermat} for a discussion
of how computer algebra systems were used in symmetric homology computations. 

\begin{center}
\begin{tabular}{l|l}
  $A$ & $HS_1(A \;|\; \Z)$ \\
  \hline
  $\Z[t]/(t^2)$ & $\Z/2\Z \oplus \Z/2\Z$ \\
  $\Z[t]/(t^3)$ & $\Z/2\Z \oplus \Z/2\Z$ \\
  $\Z[t]/(t^4)$ & $(\Z/2\Z)^4$ \\
  $\Z[t]/(t^5)$ & $(\Z/2\Z)^4$ \\
  $\Z[t]/(t^6)$ & $(\Z/2\Z)^6$ \\
  \hline
  $\Z[C_2]$ & $\Z/2\Z \oplus \Z/2\Z$ \\
  $\Z[C_3]$ & $0$ \\
  $\Z[C_4]$ & $(\Z/2\Z)^4$ \\
  $\Z[C_5]$ & $0$ \\
  $\Z[C_6]$ & $(\Z/2\Z)^6$ \\
  \hline
\end{tabular}
\end{center}

Based on these calculations, we conjecture:
\begin{conj}
  \[
    HS_1\big(k[t]/(t^n)\big) = \left\{\begin{array}{ll}
                                 (k/2k)^n, & \textrm{if $n \geq 0$ is even.}\\
                                 (k/2k)^{n-1} & \textrm{if $n \geq 1$ is odd.}
                               \end{array}\right.
  \]
\end{conj}
\begin{rmk}
  The computations of $HS_1\big(\Z[C_n]\big)$ are consistent with those of
  Brown and Loday~\cite{BL}.  See~\ref{sub.2-torsion} for a more 
  detailed treatment of $HS_1$ for group rings.
\end{rmk}
  
Additionally, $HS_1$ has been computed for the following examples.  These
computations were done using \verb|GAP| in some cases and in others,
\verb|Fermat|~\cite{LEW} computations on sparse matrices
were used in conjunction with the \verb|GAP| scripts. ({\it e.g.} when the 
algebra has dimension greater than $6$ over $\Z$).

\begin{center}
\begin{tabular}{l|l}
  $A$ & $HS_1(A \;|\; \Z)$\\
  \hline
  $\Z[t,u]/(t^2, u^2)$ & $\Z \oplus (\Z/2\Z)^{11}$\\
  $\Z[t,u]/(t^3, u^2)$ & $\Z^2 \oplus (\Z/2\Z)^{11} \oplus \Z/6\Z$\\
  $\Z[t,u]/(t^3, u^2, t^2u)$ & $\Z^2 \oplus (\Z/2\Z)^{10}$\\
  $\Z[t,u]/(t^3, u^3)$ & $\Z^4 \oplus (\Z/2\Z)^7 \oplus (\Z/6\Z)^5$\\
  $\Z[t,u]/(t^2, u^4)$ & $\Z^3 \oplus (\Z/2\Z)^{20} \oplus \Z/4\Z$\\
  $\Z[t,u,v]/(t^2, u^2, v^2)$ & $\Z^6 \oplus (\Z/2\Z)^{42}$\\
  $\Z[t,u]/(t^4, u^3)$ & $\Z^6 \oplus (\Z/2\Z)^{19} \oplus \Z/6\Z \oplus 
    (\Z/12\Z)^2$\\
  $\Z[t,u,v]/(t^2, u^2, v^3)$ & $\Z^{11} \oplus (\Z/2\Z)^{45} \oplus 
    (\Z/6\Z)^4$\\
  $\Z[i,j,k], i^2=j^2=k^2=ijk=-1$ & $(\Z/2\Z)^8$\\
  $\Z[C_2 \times C_2]$ & $(\Z/2\Z)^{12}$\\
  $\Z[C_3 \times C_2]$ & $(\Z/2\Z)^{6}$\\
  $\Z[C_3 \times C_3]$ & $(\Z/3\Z)^{9}$\\
  $\Z[S_3]$ & $(\Z/2\Z)^2$\\
  \hline
\end{tabular}
\end{center}

%%%%%%%%%%%%%%%%%%%%%%%%%%%%%%%%%%%%%%%%%%%%%%%%%%%%%%%%%%%%%%%%%%%%%%%%%%%%%%
\subsection{Splittings of the Partial Complex}\label{sub.splittings}
Under certain circumstances, the partial complex in 
Thm.~\ref{thm.partial_resolution}
splits as a direct sum of smaller complexes.  This observation becomes 
increasingly important as the dimension of the algebra increases.  Indeed, some
of the computations of the previous section were done using splittings.
\begin{definition}
  For a commutative $k$-algebra $A$ and $u \in A$, define the $k$--modules:
  \[
    \left(A^{\otimes n}\right)_u \stackrel{def}{=} \{ a_1 \otimes a_2 \otimes \ldots \otimes 
    a_n \in A^{\otimes n} \;|\; a_1a_2\cdot \ldots \cdot a_n = u \}
  \]
\end{definition}
\begin{prop}
  If $A = k[M]$ for a commutative monoid $M$, then 
  complex~(\ref{eq.partial_complex}) splits as a direct sum of complexes
  \begin{equation}\label{eq.u-homology}
    0\longleftarrow (A)_u \stackrel{\partial_1}{\longleftarrow} 
    (A\otimes A\otimes A)_u
    \stackrel{\partial_2}{\longleftarrow}
    (A\otimes A\otimes A\otimes A)_u\oplus (A)_u,
  \end{equation}
  where $u$ ranges over the elements of $M$.  Thus, for $i = 0,1$, we have
  $HS_i(A) \cong \bigoplus_{u \in M} HS_i(A)_u$.
\end{prop}
\begin{proof}
  Since $M$ is a commutative monoid, there are direct sum decompositions as
  $k$--module: $A^{\otimes n} = \bigoplus_{u \in M} \left( A^{\otimes n}
  \right)_u$.  The boundary maps $\partial_1$ and $\partial_2$ preserve the 
  products of tensor factors, so the inclusions $\left(A^{\otimes n}\right)_u 
  \hookrightarrow A^{\otimes n}$ induce maps of complexes, hence the complex 
  itself splits as a direct sum.
\end{proof}
\begin{definition}
  For each $u$, the homology groups of complex~(\ref{eq.u-homology}) will be 
  called the {\it $u$-layered symmetric homology} of $A$, denoted $HS_i(A)_u$.
\end{definition}

We may use layers to investigate the symmetric homology of $k[t]$.  This algebra
is monoidal, generated by the monoid $\{1, t, t^2, t^3, \ldots \}$.  Now, the 
$t^m$-layer symmetric homology of $k[t]$ will be the same as the $t^m$-layer 
symmetric homology of $k\left[M^{m+2}_{m+1}\right]$, where $M^p_q$ denotes the 
cyclic monoid generated by an indeterminate $s$ with the property that $s^p = 
s^q$.  Using this observation and subsequent computation, we conjecture:
\begin{conj}\label{conj.HS_1freemonoid}
  \[
    HS_1\left(k[t]\right)_{t^m} = \left\{\begin{array}{ll}
                                 0 & m = 0, 1\\
                                 k/2k, & m \geq 2\\
                               \end{array}\right.
  \]
\end{conj}
This conjecture has been verified up to $m = 18$, in the case $k = \Z$.

%%%%%%%%%%%%%%%%%%%%%%%%%%%%%%%%%%%%%%%%%%%%%%%%%%%%%%%%%%%%%%%%%%%%%%%%%%%%%%
\subsection{2-torsion in $HS_1$}\label{sub.2-torsion}
The occurrence of 2-torsion in $HS_1(A)$ for the examples considered 
in~\ref{sub.deg-1-HS} and~\ref{sub.splittings} comes as no surprise,
based on Thm.~\ref{thm.HS_group}.  First consider the following chain
of isomorphisms:  $\pi_{2}^s(B\Gamma) = \pi_2\left(\Omega^{\infty}S^{\infty}
(B\Gamma)\right) \cong \pi_{1}\left(\Omega\Omega^{\infty}S^{\infty}(B\Gamma)
\right) \cong \pi_{1}\left(\Omega_0\Omega^{\infty}S^{\infty}(B\Gamma)\right) 
\stackrel{h}{\to} H_1\left(\Omega_0\Omega^{\infty}S^{\infty}(B\Gamma)\right)$.
Here, $\Omega_0\Omega^{\infty}S^{\infty}(B\Gamma)$ denotes the component of the
constant loop, and $h$ is the Hurewicz homomorphism, which is an isomorphism 
since $\Omega_0\Omega^{\infty}S^{\infty}(B\Gamma)$ is path-connected and $\pi_1$
is abelian (since it is actually $\pi_2$ of a space).

On the other hand, by Thm.~\ref{thm.HS_group},
\[
  HS_1(k[\Gamma]) \cong H_1
  \left(\Omega\Omega^{\infty}S^{\infty}(B\Gamma); k\right) \cong
  H_1\left( \Omega\Omega^{\infty}S^{\infty}(B\Gamma)\right) \otimes k.
\]
Note, all tensor products will be over $\mathbb{Z}$ in this section.
Now $\Omega\Omega^{\infty} S^{\infty}(B\Gamma)$ consists of disjoint
homeomorphic copies of $\Omega_0 \Omega^{\infty} S^{\infty}(B\Gamma)$,
one for each element of $\Gamma/[\Gamma, \Gamma]$, (where $[\Gamma,
  \Gamma]$ is the commutator subgroup of $\Gamma$), so we may write
\[
  H_1\left(\Omega\Omega^{\infty}S^{\infty}(B\Gamma)\right) \otimes k
  \cong H_1\left(\Omega_0\Omega^{\infty}S^{\infty}(B\Gamma)\right)
  \otimes k \left[\Gamma/[\Gamma, \Gamma]\right]
\]
to obtain the following result:
\begin{prop}\label{prop.stablegrouphomotopy}
  If $\Gamma$ is a group, then $HS_1(k[\Gamma]) \cong
  \pi_{2}^s(B\Gamma) \otimes k\left[ \Gamma/[\Gamma, \Gamma] \right]$.
\end{prop}
As an immediate corollary, if $\Gamma$ is abelian, then $HS_1(k[\Gamma])
\cong \pi_2^s(B\Gamma) \otimes k[\Gamma]$.  Moreover, by results of Brown and 
Loday~\cite{BL}, if $\Gamma$ is abelian, then $\pi_2^s(B\Gamma)$ is the reduced
tensor square, $\Gamma\, \widetilde{\wedge} \,\Gamma = \left(\Gamma \otimes 
\Gamma\right)/\approx$, where $g \otimes h \approx - h \otimes g$ for all $g, h
\in \Gamma$.  

\begin{prop}\label{prop.HS_1-C_n}
  \[
    HS_1(k[C_n]) = \left\{\begin{array}{ll}
                       (\Z/2\Z)^n & \textrm{$n$ even.}\\
                       0 & \textrm{$n$ odd.}
                  \end{array}\right.
  \]
\end{prop}
\begin{proof}
  $\pi_2^s(BC_n) = \Z/2\Z$ if $n$ is even, and $0$ if $n$ is odd.
  The result then follows from Prop.~\ref{prop.stablegrouphomotopy}, as
  $k\left[ C_n/[C_n, C_n] \right] \cong k[C_n] \cong k^n$, as $k$--module.
\end{proof}

%%%%%%%%%%%%%%%%%%%%%%%%%%%%%%%%%%%%%%%%%%%%%%%%%%%%%%%%%%%%%%%%%%%%%%%%%%%%%%%%
\section{Relations to Cyclic Homology}\label{sec.cyc-homology}
%%%%%%%%%%%%%%%%%%%%%%%%%%%%%%%%%%%%%%%%%%%%%%%%%%%%%%%%%%%%%%%%%%%%%%%%%%%%%%%%

The relation between the symmetric bar construction and the cyclic bar
construction arising from the inclusions $\Delta C \hookrightarrow \Delta S$
gives rise to a natural map $HC_*(A) \to HS_*(A)$
Indeed, by remark~\ref{rmk.HC}, we may define cyclic homology thus:
$HC_*(A) = \mathrm{Tor}_*^{\Delta C}( \underline{k}, B_*^{sym}A )$, where
we understand $B_*^{sym}A$ as the restriction of the functor to $\Delta C$.

Using the partial complex of Thm.~\ref{thm.partial_resolution}, and an
analogous one for computing cyclic homology (c.f.~\cite{L}, p.~59), the
map $HC_*(A) \to HS_*(A)$ for degrees $0$ and $1$ is induced by the following
partial chain map:
\[
  \begin{diagram}
    \node{ 0 }
    \node{ A }
    \arrow{w}
    \arrow{s,r}{ \gamma_0 = \mathrm{id} }
    \node{ A \otimes A }
    \arrow{w,t}{ \partial_1^C }
    \arrow{s,r}{ \gamma_1 }
    \node{ A^{\otimes 3} \oplus A }
    \arrow{w,t}{ \partial_2^C }
    \arrow{s,r}{ \gamma_2 }
    \\
    \node{ 0 }
    \node{ A }
    \arrow{w}
    \node{ A^{\otimes 3} }
    \arrow{w,t}{ \partial_1^S }
    \node{ A^{\otimes 4} \oplus A }
    \arrow{w,t}{ \partial_2^S }
  \end{diagram}
\] 
In this diagram, the boundary maps in the upper row are defined as follows:
\[
  \partial_1^C \co a \otimes b \mapsto ab - ba
\]
\[
  \partial_2^C \co \left\{\begin{array}{ll}
                   a \otimes b \otimes c &\mapsto ab \otimes c - a \otimes bc + ca \otimes b\\
                   a &\mapsto 1 \otimes a - a \otimes 1
                 \end{array}\right.
\]
The boundary maps in the lower row are defined as in Thm.~\ref{thm.partial_resolution}.
\[
  \partial_1^S \co a \otimes b \otimes c \mapsto abc - cba
\]
\[
  \partial_2^S \co \left\{\begin{array}{ll}
                   a \otimes b \otimes c \otimes d &\mapsto ab \otimes c \otimes d 
                     - d \otimes ca \otimes b + bca \otimes 1 \otimes d
                     + d \otimes bc \otimes a\\
                   a &\mapsto 1 \otimes a \otimes 1
                 \end{array}\right.
\]
The partial chain map is given in degree 1 by $\gamma_1(a \otimes b)
\stackrel{def}{=} a \otimes b \otimes 1$.  In degree 2, $\gamma_2$ is
defined on the summand $A^{\otimes 3}$ via
\begin{eqnarray*}
  \lefteqn{a \otimes b \otimes c}\\
  &\mapsto& (a \otimes b \otimes c \otimes 1
                     - 1 \otimes a \otimes bc \otimes 1 
                     + 1 \otimes ca \otimes b \otimes 1\\
          &&         \quad  + 1 \otimes 1 \otimes abc \otimes 1
                     - b \otimes ca \otimes 1 \otimes 1, \\
          &&           - 2abc - cab),
\end{eqnarray*}
and on the summand $A$ via
\[
  a \mapsto (-1 \otimes 1 \otimes a \otimes 1, 4a).
\]

\subsection{Examples}

To provide some examples, consider the maps $\gamma_1 \co HC_1(\Z[t]/(t^n)) \to
HS_1(\Z[t]/(t^n))$.

It can be shown ({\it e.g.} by direct computation) that $HC_1(\Z[t]/(t^2)) \cong
\Z / 2\Z$ is generated by the $1$-chain $t \otimes t$.  $\gamma_1(t \otimes t) = t \otimes t
\otimes 1 \in HS_1(\Z[t]/(t^2)$ is a non-trivial element of $\Z/2\Z \oplus \Z/2\Z$ (which may
be verified by direct computation as well).  

The map $HC_1(\Z[t]/(t^3)) \to HS_1(\Z[t]/(t^2))$ may be similarly analyzed.  Here, the chain 
$t\otimes t + t \otimes t^2$ is a generator of $HC_1(\Z[t]/(t^3)) \cong \Z/6\Z$, which gets sent by
$\gamma_1$ to $t\otimes t\otimes 1 + t \otimes t^2 \otimes 1$, a non-trivial element of
$HS_1(\Z[t]/(t^3)) \cong \Z/2\Z \oplus \Z/2\Z$.

The case $n=4$ is bit more interesting.  Here, $HC_1(\Z[t]/(t^4)) \cong \Z/2\Z \oplus \Z/12\Z$,
generated by $t \otimes t$ and $t\otimes t^2 + t \otimes t^3$, respectively.  The image 
of the map $\gamma_1$ in $HS_1(\Z[t]/(t^4))$ is $(\Z/2\Z)^2 \subseteq (\Z/2\Z)^4$.

%%%%%%%%%%%%%%%%%%%%%%%%%%%%%%%%%%%%%%%%%%%%%%%%%%%%%%%%%%%%%%%%%%%%%%%%%%%%%%%%
\section{Using Computer Algebra Systems for Computing Symmetric Homology}
\label{sec.gap-fermat}
%%%%%%%%%%%%%%%%%%%%%%%%%%%%%%%%%%%%%%%%%%%%%%%%%%%%%%%%%%%%%%%%%%%%%%%%%%%%%%%%

The computer algebra systems \verb+GAP+, \verb+Octave+ and \verb+Fermat+
were used to verify proposed theorems and also to obtain some concrete
computations of symmetric homology for some small algebras.  A tar-file of the
scripts that were created and used for this work is available at 
\url{http://arxiv.org/e-print/0807.4521v1/}.  This tar-file contains the 
following files:
\begin{itemize}
  \item \verb+Basic.g+ \quad - Some elementary functions, necessary for some 
  functions in \verb+DeltaS.g+
  
  \item \verb+HomAlg.g+ \quad - Homological Algebra functions, such as 
  computation of homology groups for chain complexes.
  
  \item \verb+Fermat.g+ \quad - Functions necessary to invoke \verb+Fermat+ 
  for fast sparse matrix computations.
  
  \item \verb+fermattogap+, \verb+gaptofermat+ \quad - Auxiliary text files for
  use when invoking \verb+Fermat+ from \verb+GAP+.

  \item \verb+DeltaS.g+ \quad - This is the main repository of scripts used to
  compute various quantities associated with the category $\Delta S$, including
  $HS_1(A)$ for finite-dimensional algebras $A$.
\end{itemize}
  
In order to use the functions of \verb+DeltaS.g+, simply copy the above files 
into the working directory (such as \verb+~/gap/+), invoke \verb+GAP+, then read
in \verb+DeltaS.g+ at the prompt.  The dependent modules will automatically be
loaded (hence they must be present in the same directory as \verb+DeltaS.g+).
Note, most of the computations involving homology require substantial memory
to run.  I recommend calling \verb+GAP+ with the command line option 
``\verb+-o +{\it mem}'', where {\it mem} is the amount of memory to be allocated
to this instance of \verb+GAP+.  All computations done in this dissertation can
be accomplished by allocating 20 gigabytes of memory.  The following provides a
few examples of using the functions of \verb+DeltaS.g+

{\footnotesize
\begin{verbatim}
[ault@math gap]$ gap -o 20g

gap> Read("DeltaS.g");
gap> 
gap> ## Number of morphisms [6] --> [4]
gap> SizeDeltaS( 6, 4 );
1663200
gap> 
gap> ## Generate the set of morphisms of Delta S, [2] --> [2]
gap> EnumerateDeltaS( 2, 2 );
[ [ [ 0, 1, 2 ], [  ], [  ] ], [ [ 0, 2, 1 ], [  ], [  ] ], 
  [ [ 1, 0, 2 ], [  ], [  ] ], [ [ 1, 2, 0 ], [  ], [  ] ], 
  [ [ 2, 0, 1 ], [  ], [  ] ], [ [ 2, 1, 0 ], [  ], [  ] ], 
  [ [ 0, 1 ], [ 2 ], [  ] ], [ [ 0, 2 ], [ 1 ], [  ] ], 
  [ [ 1, 0 ], [ 2 ], [  ] ], [ [ 1, 2 ], [ 0 ], [  ] ], 
  [ [ 2, 0 ], [ 1 ], [  ] ], [ [ 2, 1 ], [ 0 ], [  ] ], 
  [ [ 0, 1 ], [  ], [ 2 ] ], [ [ 0, 2 ], [  ], [ 1 ] ], 
  [ [ 1, 0 ], [  ], [ 2 ] ], [ [ 1, 2 ], [  ], [ 0 ] ], 
  [ [ 2, 0 ], [  ], [ 1 ] ], [ [ 2, 1 ], [  ], [ 0 ] ], 
  [ [ 0 ], [ 1, 2 ], [  ] ], [ [ 0 ], [ 2, 1 ], [  ] ], 
  [ [ 1 ], [ 0, 2 ], [  ] ], [ [ 1 ], [ 2, 0 ], [  ] ], 
  [ [ 2 ], [ 0, 1 ], [  ] ], [ [ 2 ], [ 1, 0 ], [  ] ], 
  [ [ 0 ], [ 1 ], [ 2 ] ], [ [ 0 ], [ 2 ], [ 1 ] ], [ [ 1 ], [ 0 ], [ 2 ] ], 
  [ [ 1 ], [ 2 ], [ 0 ] ], [ [ 2 ], [ 0 ], [ 1 ] ], [ [ 2 ], [ 1 ], [ 0 ] ], 
  [ [ 0 ], [  ], [ 1, 2 ] ], [ [ 0 ], [  ], [ 2, 1 ] ], 
  [ [ 1 ], [  ], [ 0, 2 ] ], [ [ 1 ], [  ], [ 2, 0 ] ], 
  [ [ 2 ], [  ], [ 0, 1 ] ], [ [ 2 ], [  ], [ 1, 0 ] ], 
  [ [  ], [ 0, 1, 2 ], [  ] ], [ [  ], [ 0, 2, 1 ], [  ] ], 
  [ [  ], [ 1, 0, 2 ], [  ] ], [ [  ], [ 1, 2, 0 ], [  ] ], 
  [ [  ], [ 2, 0, 1 ], [  ] ], [ [  ], [ 2, 1, 0 ], [  ] ], 
  [ [  ], [ 0, 1 ], [ 2 ] ], [ [  ], [ 0, 2 ], [ 1 ] ], 
  [ [  ], [ 1, 0 ], [ 2 ] ], [ [  ], [ 1, 2 ], [ 0 ] ], 
  [ [  ], [ 2, 0 ], [ 1 ] ], [ [  ], [ 2, 1 ], [ 0 ] ], 
  [ [  ], [ 0 ], [ 1, 2 ] ], [ [  ], [ 0 ], [ 2, 1 ] ], 
  [ [  ], [ 1 ], [ 0, 2 ] ], [ [  ], [ 1 ], [ 2, 0 ] ], 
  [ [  ], [ 2 ], [ 0, 1 ] ], [ [  ], [ 2 ], [ 1, 0 ] ], 
  [ [  ], [  ], [ 0, 1, 2 ] ], [ [  ], [  ], [ 0, 2, 1 ] ], 
  [ [  ], [  ], [ 1, 0, 2 ] ], [ [  ], [  ], [ 1, 2, 0 ] ], 
  [ [  ], [  ], [ 2, 0, 1 ] ], [ [  ], [  ], [ 2, 1, 0 ] ] ]
gap> 
gap> ## Generate only the epimorphisms [2] -->> [2]
gap> EnumerateDeltaS( 2, 2 : epi ); 
[ [ [ 0 ], [ 1 ], [ 2 ] ], [ [ 0 ], [ 2 ], [ 1 ] ], 
  [ [ 1 ], [ 0 ], [ 2 ] ], [ [ 1 ], [ 2 ], [ 0 ] ], 
  [ [ 2 ], [ 0 ], [ 1 ] ], [ [ 2 ], [ 1 ], [ 0 ] ] ]
gap> 
gap> ## Compose two morphisms of Delta S.
gap> a := Random(EnumerateDeltaS(4,3));
[ [ 0 ], [ 2, 4, 1 ], [  ], [ 3 ] ]
gap> b := Random(EnumerateDeltaS(3,2));
[ [  ], [ 3, 0, 2 ], [ 1 ] ]
gap> MultDeltaS(b, a);
[ [  ], [ 3, 0 ], [ 2, 4, 1 ] ]
gap> MultDeltaS(a, b);
Maps incomposeable
[  ]
gap> 
gap> ## Examples of using morphisms of Delta S to act on simple tensors
gap> A := TruncPolyAlg([3,2]);
<algebra of dimension 6 over Rationals>
gap> ## TruncPolyAlg is defined in Basic.g
gap> ##  TruncPolyAlg([i_1, i_2, ..., i_n]) is generated by
gap> ##  x_1, x_2, ..., x_n, under the relation (x_j)^(i_j) = 0.
gap> g := GeneratorsOfLeftModule(A);
[ X^[ 0, 0 ], X^[ 0, 1 ], X^[ 1, 0 ], X^[ 1, 1 ], X^[ 2, 0 ], X^[ 2, 1 ] ]
gap> x := g[2]; y := g[3];
X^[ 0, 1 ]
X^[ 1, 0 ]
gap> v := [ x*y, 1, y^2 ];
gap> ## v represents the simple tensor  xy \otimes 1 \otimes y^2.
[ X^[ 1, 1 ], 1, X^[ 2, 0 ] ]
gap> ActByDeltaS( v, [[2], [], [0], [1]] );
[ X^[ 2, 0 ], 1, X^[ 1, 1 ], 1 ]
gap> ActByDeltaS( v, [[2], [0,1]] );       
[ X^[ 2, 0 ], X^[ 1, 1 ] ]
gap> ActByDeltaS( v, [[2,0], [1]] );  
[ 0*X^[ 0, 0 ], 1 ]
gap> 
gap> ## Symmetric monoidal product on DeltaS_+
gap> a := Random(EnumerateDeltaS(4,2));
[ [  ], [ 2, 1, 0 ], [ 3, 4 ] ]
gap> b := Random(EnumerateDeltaS(3,3));
[ [  ], [  ], [  ], [ 1, 3, 2, 0 ] ]
gap> MonoidProductDeltaS(a, b);
[ [  ], [ 2, 1, 0 ], [ 3, 4 ], [  ], [  ], [  ], [ 6, 8, 7, 5 ] ]
gap> MonoidProductDeltaS(b, a);
[ [  ], [  ], [  ], [ 1, 3, 2, 0 ], [  ], [ 6, 5, 4 ], [ 7, 8 ] ]
gap> MonoidProductDeltaS(a, []);
[ [  ], [ 2, 1, 0 ], [ 3, 4 ] ]
gap> 
gap> ## Symmetric Homology of the algebra A, in degrees 0 and 1.
gap> SymHomUnitalAlg(A);
[ [ 0, 0, 0, 0, 0, 0 ], [ 2, 2, 2, 2, 2, 2, 2, 2, 2, 2, 2, 6, 0, 0 ] ]
gap> ## '0' represents a factor of Z, while a non-zero p represents
gap> ## a factor of Z/pZ.
gap> 
gap> ## Using layers to compute symmetric homology
gap> C2 := CyclicGroup(2);                                       
<pc group of size 2 with 1 generators>
gap> A := GroupRing(Rationals, DirectProduct(C2, C2));
<algebra-with-one over Rationals, with 2 generators>
gap> ## First, a direct computation without layers:
gap> SymHomUnitalAlg(A);
[ [ 0, 0, 0, 0 ], [ 2, 2, 2, 2, 2, 2, 2, 2, 2, 2, 2, 2 ] ]
gap> ## Next, compute HS_0(A)_u and HS_1(A)_u for each generator u.
gap> g := GeneratorsOfLeftModule(A);
[ (1)*<identity> of ..., (1)*f2, (1)*f1, (1)*f1*f2 ]
gap> SymHomUnitalAlgLayered(A, g[1]);
[ [ 0 ], [ 2, 2, 2 ] ]
gap> SymHomUnitalAlgLayered(A, g[2]);
[ [ 0 ], [ 2, 2, 2 ] ]
gap> SymHomUnitalAlgLayered(A, g[3]);
[ [ 0 ], [ 2, 2, 2 ] ]
gap> SymHomUnitalAlgLayered(A, g[4]);
[ [ 0 ], [ 2, 2, 2 ] ]
gap> ## Computing HS_1( Z[t] ) by layers:
gap> SymHomFreeMonoid(0,10);
HS_1(k[t])_{t^0} :  [  ]
HS_1(k[t])_{t^1} :  [  ]
HS_1(k[t])_{t^2} :  [ 2 ]
HS_1(k[t])_{t^3} :  [ 2 ]
HS_1(k[t])_{t^4} :  [ 2 ]
HS_1(k[t])_{t^5} :  [ 2 ]
HS_1(k[t])_{t^6} :  [ 2 ]
HS_1(k[t])_{t^7} :  [ 2 ]
HS_1(k[t])_{t^8} :  [ 2 ]
HS_1(k[t])_{t^9} :  [ 2 ]
HS_1(k[t])_{t^10} :  [ 2 ]
gap> ## Poincare polynomial of Sym_*^{(p)} for small p.
gap> ##  There is a check for torsion, using a call to Fermat
gap> ##  to find Smith Normal Form of the differential matrices.
gap> PoincarePolynomialSymComplex(2);
C_0  Dimension: 1
C_1  Dimension: 6
C_2  Dimension: 6
D_1
SNF(D_1)
D_2
SNF(D_2)
2*t^2+t
gap> PoincarePolynomialSymComplex(5);
C_0  Dimension: 1
C_1  Dimension: 30
C_2  Dimension: 300
C_3  Dimension: 1200
C_4  Dimension: 1800
C_5  Dimension: 720
D_1
SNF(D_1)
D_2
SNF(D_2)
D_3
SNF(D_3)
D_4
SNF(D_4)
D_5
SNF(D_5)
120*t^5+272*t^4+t^3
\end{verbatim}
}

%%%%%%%%%%%%%%%%%%%%   End of main body of article
%
%                             References
%
%   BiBTeX users uncomment the following line:
%
\bibliographystyle{gtart}

\end{document}